\documentclass[11pt,a4paper]{amsart}
\usepackage[T1]{fontenc}
\usepackage{lmodern}
\usepackage{amsmath,amssymb,amsthm,mathrsfs,mathtools}
\usepackage{esint}
\usepackage[a4paper,left=25mm,right=25mm,top=22mm,bottom=23mm]{geometry}
\usepackage{microtype}
\usepackage{enumitem}
\usepackage{booktabs,tabularx,array}
\usepackage{cite}
\usepackage[unicode,hidelinks]{hyperref}
\numberwithin{equation}{section}
\allowdisplaybreaks[2]
\theoremstyle{plain}
\newtheorem{theorem}{Theorem}[section]
\newtheorem{proposition}[theorem]{Proposition}
\newtheorem{lemma}[theorem]{Lemma}
\newtheorem{corollary}[theorem]{Corollary}
\theoremstyle{definition}
\newtheorem{definition}[theorem]{Definition}
\newtheorem{remark}[theorem]{Remark}

\newcommand{\R}{\mathbb R}

\newcommand{\pv}{\operatorname{p.v.}}
\newcommand{\supp}{\operatorname{supp}}
\newcommand{\Lip}{\operatorname{Lip}}
\newcommand{\dist}{\operatorname{dist}}

\hypersetup{pdftitle={Reflectionless Riesz Measures: Rigidity and Stability},
 pdfauthor={Lu Chen and Haolin Liu},
 pdfsubject={Rigidity, energy estimates, and Sobolev estimates for Riesz transforms}}
\begin{document}
\title[Rigidity of reflectionless Riesz measures and Sobolev estimates]{Rigidity of reflectionless Riesz measures and Sobolev estimates for Riesz transforms on Lipschitz graphs}
\author{Lu Chen}
\address[Lu Chen]{Key Laboratory of Algebraic Lie Theory and Analysis of Ministry of Education, School of Mathematics and Statistics, Beijing Institute of Technology, Beijing
100081, PR China;  Tangshan Research Institute, Beijing Institute of Technology, Tangshan 063000, PR China}
\email{chenlu5818804@163.com}

\author{Haolin Liu$^{*}$}
\address[Haolin Liu]{Key Laboratory of Algebraic Lie Theory and Analysis of Ministry of Education, School of Mathematics and Statistics, Beijing Institute of Technology, Beijing
100081, PR China}
\email{haolinliu2023@126.com}

\author{Hongyu Liu}
\address[Hongyu Liu]{Department of Mathematics, City University of Hong Kong, Hong Kong SAR, China}
\email{hongyu.liuip@gmail.com, hongyliu@cityu.edu.hk}

\thanks{$*$ Corresponding author.}
\thanks{The first author was partly supported by the  National Natural Science Foundation of China (No. 12622108), a grant from Beijing Institute of Technology (No.2022CX01002) and Hebei Natural
Science Foundation (No. A2025105003).}

\date{}
\begin{abstract}
In this paper, we focus on the positive reflectionless measures rigidity problem proposed by Xavier Tolsa in \cite{TolsaICM}. Let \(1\le n<d\), we consider positive reflectionless measures for the
\(n\)-dimensional Riesz transform in \(\mathbb R^d\).
Under polynomial upper growth and a lower mass bound at large scales
about one point, we obtain two flatness results for reflectionless measures.
The first concerns globally defined Lipschitz graph measures with finite
\(\dot H^{1/2}\) energy and bounded, uniformly positive parameter
densities whose deviations from a constant belong to \(L^2\).
The second case concerns a reflectionless measure that coincides with a flat measure outside a compact set. In both cases, the measure is a constant multiple of the \(n\)-dimensional
Hausdorff measure on an affine \(n\)-plane.
For integers \(s>n/2\), we further obtain \(H^s\) estimates controlling
the graph gradient and density perturbation by the normalized Riesz
field under smallness of the slope and density perturbation in
\(L^\infty\). The constants are independent of the supports.
\end{abstract}
\maketitle
\tableofcontents

\section{Introduction}\label{sec:introduction}

Let $1\le n<d$ and put $m=d-n$. The $n$-dimensional Riesz kernel is
\begin{equation}\label{eq:Riesz-kernel}
 K_n(Z)=\frac{Z}{|Z|^{n+1}},\qquad Z\in\R^d\setminus\{0\}.
\end{equation}
For infinite measures, its action on the constant function must be
normalized at infinity. Throughout this paper, reflectionlessness means
the balanced weak identity introduced by Jaye--Nazarov in the class with uniformly
$L^2$-bounded Riesz truncations; the precise convention is given in
Definition~\ref{def:reflectionless}. A nonzero measure is called flat
if it has the form $c\mathcal H^n|_L$, where $c>0$ and $L$ is an
affine $n$-plane.

For integer-dimensional Riesz transforms, Jaye--Nazarov
\cite[Question~1.2(b)]{JayeNazarovII} ask whether every
AD-regular reflectionless measure is flat.
At ICM 2026, Tolsa \cite[Question~4.10]{TolsaICM} formulates the following corresponding
classification question for measures with polynomial growth.

\begin{flushleft}
 \textbf{Problem} Let $n$ be an integer such that $2\leq n\leq d-1$. If $\mu$ is a reflectionless measure with polynomial
$n-$growth (or even $n-$Ahlfors regular) for the $n-$Riesz transform, then is $\mu$ $n-$flat?
\end{flushleft}

That survey records the classification for $n=1$ and lists
the case $2\le n\le d-1$ as open, including the AD-regular
codimension-one case.
The results below concern measures satisfying the additional
finite-energy or exterior-flatness assumptions stated in our
theorems.

Our results provide a resolution of a restricted version of Tolsa's
classification problem for reflectionless Riesz measures under additional
assumptions controlling the behavior at infinity. Theorem~\ref{prop:ambient-planarity} establishes
an orthogonal energy rigidity criterion, showing that finite interaction
energy in directions orthogonal to an \(n\)-plane forces the support of the
measure to lie in a translate of that plane. This rigidity criterion is the
key ingredient in the proof of Theorem~\ref{thm:compact-flat}. Theorem~\ref{thm:global-rigidity} and Theorem~\ref{thm:compact-flat} give affirmative answers to the classification problem in
two different settings with controlled behavior at infinity. Finally, Theorem~\ref{thm:high-stability} provides a Sobolev stability estimate for the
normalized Riesz field, which can be viewed as a local stability result
for recovering geometric and density perturbations from the associated
nonlocal field.

\subsection{Rigidity under a finite-energy assumption}
We assume polynomial upper growth
\begin{equation}\label{eq:upper-growth}
 \mu(B(U,r))\le A_0r^n
 \qquad(U\in\R^d,\ r>0),
\end{equation}
and, when specified, the following lower growth condition at large scales
\begin{equation}\label{eq:large-lower-growth}
 \mu(B(O,r))\ge a_0r^n\qquad(r\ge R_*),
\end{equation}
where $O\in\R^d$ is one fixed point and $a_0,R_*>0$.
For an $n$-dimensional linear subspace $P$, let $\Pi_{P^\perp}$ be the
orthogonal projection onto $P^\perp$, and set
\begin{equation}\label{eq:ambient-perp-energy-intro}
 \mathscr E_P^\perp(\mu)
 =\frac12\iint
 \frac{|\Pi_{P^\perp}(U-V)|^2}{|U-V|^{n+1}}
 \,d\mu(U)d\mu(V).
\end{equation}
The integrand is set equal to zero on the diagonal.

\begin{theorem}
\label{prop:ambient-planarity}
Let $\mu$ be reflectionless and satisfy
\eqref{eq:upper-growth}--\eqref{eq:large-lower-growth}.
If $\mathscr E_P^\perp(\mu)<\infty$ for an $n$-dimensional linear
subspace $P$, then
\begin{equation}\label{eq:ambient-planarity}
 \supp\mu\subset b+P
 \quad\text{for some }b\in P^\perp.
\end{equation}

\end{theorem}

The proof uses the coordinate tests and normalized cutoff functions
of \cite[Section~10]{JayeNazarovII}. Their finite
\emph{full}-energy argument rules out nonzero reflectionless measures.
In contrast, a nonzero flat measure has zero orthogonal energy but
\[
 \iint |U-V|^{1-n}\,d\mu(U)d\mu(V)=\infty.
\]
Thus their full-energy hypothesis does not apply to the flat backgrounds
considered here. The cutoff estimates that permit the use of finite
orthogonal energy are
\[
 \mathscr E_\mu(\beta_t)\le Ct^{-n-1},\qquad
 \mathscr E_\mu\bigl(\chi_R(u-u_{B_{4R},\mu})\bigr)
 \le C\mathscr E_\mu(u).
\]
These estimates are proved in Section~\ref{sec:balancing}.

\subsection{Classification of two classes of measures}
First consider a globally Lipschitz map $\varphi:\R^n\to\R^m$ and write
\begin{equation}\label{eq:graph-measure}
 X_\varphi(x)=(x,\varphi(x)),\qquad
 \Gamma_\varphi=X_\varphi(\R^n),\qquad
 \mu_{\varphi,a}=(X_\varphi)_\#(a(x)\,dx).
\end{equation}
The parameter density $a$ and the intrinsic Hausdorff density $g$ satisfy
\begin{equation}\label{eq:hausdorff-density}
 d\mu_{\varphi,a}=g\,d\mathcal H^n\big|_{\Gamma_\varphi},\qquad
 a=(g\circ X_\varphi)J_\varphi,\qquad
 J_\varphi=\sqrt{\det(I+D\varphi^TD\varphi)}.
\end{equation}
We define the homogeneous seminorm by
\begin{equation}\label{eq:half-seminorm}
 [u]_{1/2}^2=[u]_{\dot H^{1/2}(\R^n)}^2
 :=\iint_{\R^n\times\R^n}
 \frac{|u(x)-u(y)|^2}{|x-y|^{n+1}}\,dx\,dy.
\end{equation}
\begin{theorem}\label{thm:global-rigidity}
Let $\varphi:\R^n\to\R^m$ be globally Lipschitz and satisfy
$[\varphi]_{1/2}<\infty$. Suppose that
\begin{equation}\label{eq:density-bounds}
 0<c_0\le a\le C_0<\infty,\qquad
 a-a_\infty\in L^2(\R^n),\qquad a_\infty>0.
\end{equation}
If $\mu_{\varphi,a}$ is reflectionless and satisfies the polynomial upper growth condition
\eqref{eq:upper-growth} and the lower mass bound at large scales
\eqref{eq:large-lower-growth} then for some $b\in\R^m$,
\begin{equation}\label{eq:rigidity-conclusion}
 \varphi\equiv b,\qquad a\equiv a_\infty,\qquad
 \mu_{\varphi,a}=a_\infty\mathcal H^n
       \big|_{\R^n\times\{b\}}.
\end{equation}
 If $\varphi\in L^2$, or if
$\varphi=0$ outside a compact set, then $b=0$.
\end{theorem}

The graph measure satisfies the growth assumptions of
Theorem~\ref{prop:ambient-planarity}, and its orthogonal energy is
comparable to $[\varphi]_{1/2}^2$. Once the graph is contained in a plane, the weak
identity implies that the Riesz transform $\mathcal R(a-a_\infty)$ is constant. Its $L^2$ membership
then forces that constant to vanish. The density assumption does not
include an estimate of
$\|a-a_\infty\|_2$ in terms of the graph energy.

A different application retains exact flatness at infinity but removes
the graph assumption in the exceptional region.

\begin{theorem}
\label{thm:compact-flat}
Let $\mu$ be a positive reflectionless measure satisfying
\eqref{eq:upper-growth} and \eqref{eq:large-lower-growth}. Suppose that for an affine $n$-plane $L$,
$c>0$, and $R_0<\infty$,
\begin{equation}\label{eq:exterior-flat}
 \mu\big|_{\R^d\setminus\overline B_{R_0}}
 =c\mathcal H^n\big|_{L\setminus\overline B_{R_0}}.
\end{equation}
Then $\mu=c\mathcal H^n|_L$ on all of $\R^d$.

\end{theorem}

In the proof of Theorem~\ref{thm:compact-flat}, there is no graph assumption and no lower density assumption in
$\overline B_{R_0}$. The exterior hypothesis supplies both large-scale mass and finite
orthogonal energy. Once the support lies in $L$, upper growth gives a
bounded nonnegative density. To prove that this density is constant,
we use the exterior identity and the anti-locality of the Riesz transform.
We give an elementary Fourier proof of the required special case of the
fractional unique-continuation principle
\cite[Theorem~1.2]{GhoshSaloUhlmann}. Theorems~\ref{thm:global-rigidity}
and~\ref{thm:compact-flat} impose different conditions at infinity;
neither set of assumptions implies the other.

\subsection{Energy estimates and Sobolev stability on graphs}
The quantitative statements do not assume reflectionlessness. For
$0<c_0\le a\le C_0$ set
\begin{equation}\label{eq:graph-weight}
 w_{\varphi,a}(x,y)
 =\frac{a(x)a(y)}{|X_\varphi(x)-X_\varphi(y)|^{n+1}}.
\end{equation}
The distribution associated with the weighted orthogonal component is defined by
\begin{equation}\label{eq:weak-normal-field}
 \langle\mathcal F_{\varphi,a},v\rangle
 =\frac12\iint
 (\varphi(x)-\varphi(y))\cdot(v(x)-v(y))
 w_{\varphi,a}(x,y)\,dx\,dy,
 \qquad v\in C_c^\infty(\R^n;\R^m).
\end{equation}
For finite graph energy and bounded slope, Proposition~\ref{thm:half-stability}
gives the estimate
\begin{equation}\label{eq:half-stability}
 c[\varphi]_{1/2}\le
 \|\mathcal F_{\varphi,a}\|_{\dot H^{-1/2}}
 \le C[\varphi]_{1/2}.
\end{equation}
Here the norm is the supremum against smooth compactly supported test functions with
$[v]_{1/2}\le1$, as specified in \eqref{eq:dual-half-norm}.
This is a coercivity consequence of the energy form; it requires no
continuity or limiting value of $a$.

When the principal-value representation is available,
$\mathcal F_{\varphi,a}=a\mathcal N_{\varphi,a}$, where
\begin{equation}\label{eq:normal-field}
 \mathcal N_{\varphi,a}(x)=\pv\int
 \frac{\varphi(x)-\varphi(y)}{|X_\varphi(x)-X_\varphi(y)|^{n+1}}
 a(y)\,dy.
\end{equation}
The weak definition, not multiplication of an arbitrary distribution by
$a$, is used for merely measurable densities. Section~\ref{sec:vertical-energy}
also proves the conditional inequality
\begin{equation}\label{eq:intro-conditional-half}
 [\varphi]_{1/2}^2\le
 C\|\varphi\|_{L^2(a\,dx)}\|\mathcal N_{\varphi,a}\|_{L^2(a\,dx)}
\end{equation}
when $\varphi\in H^{1/2}$ and the weak field has the indicated $L^2$
representation. A compactly supported $C^{1,\alpha}$ graph admits this
representation for any bounded positive measurable density, and
fractional Poincar\'e recovers the linear $L^2$ estimate of
Corollary~\ref{cor:compact-half}.

For higher-order estimates write $a=1+\rho$ and, for $z=x-y$, put
\begin{equation}\label{eq:QAB-intro}
 Q_\varphi(x,y)=\frac{\varphi(x)-\varphi(y)}{|x-y|},\qquad
 A(q)=(1+|q|^2)^{-(n+1)/2},\qquad B(q)=qA(q).
\end{equation}
Subtract the horizontal constant-density background and define
\begin{align}
 \mathbf T^\top(\varphi,\rho)(x)
 &=\pv\int\frac{z}{|z|^{n+1}}
 \bigl(A(Q_\varphi(x,y))(1+\rho(y))-1\bigr)\,dy,
       \label{eq:T-tangent}\\
 \mathbf T^\perp(\varphi,\rho)(x)
 &=\pv\int\frac{B(Q_\varphi(x,y))(1+\rho(y))}{|z|^n}\,dy.
       \label{eq:T-normal}
\end{align}
The projections are relative to the fixed horizontal plane, not the
varying tangent plane. We write $\mathbf T=(\mathbf T^\top,\mathbf T^\perp)$
and first use centered parameter-space truncations.
Appendix~\ref{app:principal-values} identifies them with graph-distance
principal values under the regularity below.

\begin{theorem}\label{thm:high-stability}
Let $s\in\mathbb N$ satisfy $s>n/2$. There exist $\varepsilon_s,C_s>0$,
depending only on $n,m,s$, such that, if
\[
 \varphi\in H^{s+1}(\R^n;\R^m),\qquad \rho\in H^s(\R^n),\qquad
 \delta:=\|D\varphi\|_\infty+\|\rho\|_\infty\le\varepsilon_s,
\]
then $1+\rho\ge1/2$, $\mathbf T(\varphi,\rho)\in H^s(\R^n;\R^d)$, and
\begin{equation}\label{eq:high-stability}
 \|D\varphi\|_{H^s}+\|\rho\|_{H^s}
 \le C_s\|\mathbf T(\varphi,\rho)\|_{H^s}.
\end{equation}
 The full inhomogeneous norm satisfies
\begin{equation}\label{eq:high-with-low-frequency}
 \|\varphi\|_{H^{s+1}}+\|\rho\|_{H^s}
 \le C_s\bigl(\|\mathbf T(\varphi,\rho)\|_{H^s}+\|\varphi\|_2\bigr).
\end{equation}
If $\supp\varphi\subset B_R$, the $L^2$ term can be omitted with a
constant depending on $R$.
\end{theorem}

The linear part is
\begin{equation}\label{eq:intro-linearization}
 D\mathbf T(0,0)[\psi,\sigma]
 =(c_{n,R}\mathcal R\sigma,c_{n,\Lambda}|D|\psi),
\end{equation}
where the constants are nonzero and depend only on the kernel
normalization. The main nonlinear estimate is
\begin{equation}\label{eq:intro-remainder}
 \|\mathbf T(\varphi,\rho)-D\mathbf T(0,0)[\varphi,\rho]\|_{H^s}
 \le C_s\delta\bigl(\|D\varphi\|_{H^s}+\|\rho\|_{H^s}\bigr).
\end{equation}
The proof uses Theorem~\ref{thm:CJ-holder}, which is obtained by applying
\cite[Theorem~1.1 and Section~2]{MuscaluII} on lines and using the method
of rotations. Interpolation and the product rule then give an estimate
with one Sobolev norm and the remaining factors in $L^\infty$.
Smallness is required only of $\delta$, not of the Sobolev norms.
All differentiations are
integer-order; no endpoint $s=n/2$ or fractional-order assertion is made.
The space remains $H^{s+1}\times H^s$, rather than its completion
under the derivative norm.

\subsection{Relation to previous work and organization}
The direct graph comparison is Tolsa's estimate
\cite[Theorem~1.3]{Tolsa2008}. For a compactly supported graph function
$A$, small $\|DA\|_\infty$, and a bounded positive Hausdorff density with
$\|g-1\|_{L^2(\Gamma)}\le C_2\|DA\|_2$, it gives
\begin{equation}\label{eq:Tolsa-comparison}
 \|\pv\mathcal R^{n,\perp}\mu\|_{L^2(\mu)}
 \asymp\|\pv\mathcal R^n\mu\|_{L^2(\mu)}
 \asymp\|DA\|_2.
\end{equation}
Our arbitrary-slope result is exact rigidity under a finite-energy
hypothesis, not a replacement for this $H^1$ lower bound. The parameter
density and Hausdorff density must be distinguished, particularly without
compact support; see Appendix~\ref{app:density-conversion}.

The graph estimates are part of a broader theory. In codimension one,
Nazarov--Tolsa--Volberg \cite{NazarovTolsaVolberg} obtain uniform
rectifiability from AD regularity and $L^2$ boundedness of the Riesz
transform. Tolsa \cite{TolsaNewCriteria} gives local rectifiability
criteria from Riesz-field oscillation and density information. These are
geometric regularity results for general measures, rather than the
classification results under the additional assumptions made here. On the nonlinear
operator side, shape dependence of layer potentials and Sobolev estimates
for graph-dependent singular integrals have an established theory; see,
for example, \cite{DallaRivaLuzziniMusolino,MatiocMatioc}. The estimates below use these established multilinear methods for
the particular graph and density variables in \eqref{eq:T-tangent}--
\eqref{eq:T-normal}.

Section~\ref{sec:reflectionless-framework} fixes the weak pairing and the
operator conventions. Section~\ref{sec:balancing} proves
Theorems~\ref{prop:ambient-planarity} and~\ref{thm:global-rigidity}.
The classification under the exterior flatness assumption and its annular compactness consequence
are proved in Sections~\ref{sec:compact-flat}.
Section~\ref{sec:vertical-energy} treats the orthogonal energy field.
Sections~\ref{sec:linearization}--\ref{sec:high-order-proof} prove
Theorem~\ref{thm:high-stability} and the local two-point estimate.

\section{Preliminaries}\label{sec:reflectionless-framework}
This section establishes the weak formulation and operator conventions used throughout the paper. Using the growth condition and cancellation of zero-mass tests, we define the balanced Riesz pairing and justify its representation by an integrable potential under the stated hypotheses. We also record the basic properties of weighted graph measures and the relevant Fourier and fractional Poincar\'e identities, providing the framework for the rigidity proofs and subsequent quantitative estimates.
\subsection{Growth, truncations, and the weak pairing}
Balls in the ambient space are denoted by $B(U,r)$; $B_R\subset\R^n$
denotes a parameter-space ball when graph coordinates are in use.
All unlabelled integrals over graph parameters are Lebesgue integrals.
We use the upper-growth condition \eqref{eq:upper-growth} throughout.
The lower bound \eqref{eq:large-lower-growth} is assumed only when stated.
A globally $n$-AD regular measure has both upper and lower bounds at all
radii about every point of its support. Thus \eqref{eq:large-lower-growth}
is weaker than global AD regularity.
For functions on a parameter space, $D$ denotes the full derivative;
$\|Du\|_{H^s}$ is the Euclidean sum of the component Sobolev norms.
The spaces $H^s$ are the usual inhomogeneous $L^2$ Sobolev spaces.
We write $\Lip_c$ for compactly supported Lipschitz functions.
Constants may change between occurrences; their dependencies are stated
in the corresponding assertion.

For compactly supported Lipschitz $f,\psi$, set
\[
 H_{f,\psi}(U,V)=\tfrac12(f(V)\psi(U)-f(U)\psi(V)).
\]
The identity
\[
 2H_{f,\psi}=(f(V)-f(U))\psi(U)+f(U)(\psi(U)-\psi(V))
\]
gives
\begin{equation}\label{eq:H-Lipschitz}
 |H_{f,\psi}(U,V)|\le
 \tfrac12\bigl(\Lip(f)\|\psi\|_\infty+
 \|f\|_\infty\Lip(\psi)\bigr)|U-V|.
\end{equation}

We shall repeatedly use two integrals that follow from upper growth.
For the first one decompose the punctured ball into
\[
 A_j(U,r)=\{V:2^{-j-1}r<|U-V|\le 2^{-j}r\},\qquad j\ge0.
\]
Since $n\ge1$, on $A_j(U,r)$ one has
$|U-V|^{1-n}\le(2^{-j-1}r)^{1-n}$. Consequently
\begin{align}
 \int_{0<|U-V|<r}|U-V|^{1-n}\,d\mu(V)
 &\le\sum_{j\ge0}(2^{-j-1}r)^{1-n}
                  \mu(B(U,2^{-j}r))\notag\\
 &\le 2^{n-1}A_0r\sum_{j\ge0}2^{-j}
 \le C_nA_0r.\label{eq:near-growth}
\end{align}
For the second one use the exterior annuli
$2^jr<|U-V|\le2^{j+1}r$:
\begin{align}
 \int_{|U-V|>r}|U-V|^{-n-1}\,d\mu(V)
 &\le\sum_{j\ge0}(2^jr)^{-n-1}\mu(B(U,2^{j+1}r))\notag\\
 &\le 2^nA_0r^{-1}\sum_{j\ge0}2^{-j}
 \le C_nA_0r^{-1}.\label{eq:far-growth}
\end{align}
These estimates are valid for arbitrary ambient centers; the constants
are independent of $U$ and $r$. Upper growth also gives
$\mu(\{U\})\le\inf_{r>0}A_0r^n=0$. Thus the product diagonal has zero
$\mu\otimes\mu$ measure. Indeed, on every bounded product set this
follows from Tonelli applied to $\mu(\{U\})$, and the whole diagonal is
the union of its bounded portions.
We say that a Radon measure $\mu$ has restricted growth at
infinity if
\[
 \int_{|U|>1}|U|^{-n-1}\,d\mu(U)<\infty.
\]
This condition is unchanged if the origin is replaced by any
fixed center and the cutoff radius $1$ by any positive radius.
Under \eqref{eq:upper-growth}, it follows from
\eqref{eq:far-growth}.
Put $K=\supp f\cup\supp\psi$ and choose $H$ with $K\subset B_H$.
The factor $H_{f,\psi}$ vanishes unless both variables belong to $K$.
Near the diagonal, \eqref{eq:H-Lipschitz} and \eqref{eq:near-growth} give
\[
 \iint_{\substack{U,V\in K\\0<|U-V|<1}}
 |K_n(U-V)H_{f,\psi}(U,V)|\,d\mu(U)d\mu(V)
 \le C_{f,\psi}A_0\mu(K)<\infty.
\]
On $K\times K\cap\{|U-V|\ge1\}$ the integrand is bounded and both
measures have finite mass. Hence the compact pairing
\begin{equation}\label{eq:JN-pairing}
 \langle\mathcal R^n(f\mu),\psi\rangle_\mu
 =\iint K_n(U-V)H_{f,\psi}(U,V)\,d\mu(U)d\mu(V)
\end{equation}
is absolutely convergent. A value of zero may be assigned on the
diagonal. If an odd regularization changes the kernel only for
$|U-V|<C\varepsilon$ and is bounded in absolute value by
$C|U-V|^{-n}$, its error in this pairing is at most
$C_{f,\psi}\mu(K)A_0\varepsilon$. This also explains the independence
of the small-scale regularization in the compact weak form.

A compactly supported signed measure $\nu$ is balanced if
$\nu(\R^d)=0$.  For $\nu=f\mu$, this means $\int f\,d\mu=0$.
If $\supp\nu\subset B(0,R_f)$ and $|Z|>2R_f$, then
\begin{align}
 \mathcal R^n\nu(Z)
 &=\int\bigl(K_n(Z-W)-K_n(Z)\bigr)\,d\nu(W),
       \label{eq:balanced-decay-identity}\\
 |\mathcal R^n\nu(Z)|
 &\le C|Z|^{-n-1}\int |W|\,d|\nu|(W).
       \label{eq:balanced-decay-bound}
\end{align}
Indeed $|DK_n(Z-tW)|\le C|Z|^{-n-1}$ for $0\le t\le1$.
Together with \eqref{eq:far-growth}, this makes the far-field potential
integrable against $\mu$.

For $0\le\chi_k\le1$, $\chi_k\in\Lip_c(\R^d)$ and
$\chi_k=1$ on $B(0,k)$, define
\begin{equation}\label{eq:JN-balanced-extension}
 \langle\mathcal R^n(f\mu),1\rangle_\mu
 :=\lim_{k\to\infty}\langle\mathcal R^n(f\mu),\chi_k\rangle_\mu,
 \qquad \int f\,d\mu=0.
\end{equation}
For large $k,j$, the difference of two cutoffs vanishes near $\supp f$.
Their pairing difference is therefore the ordinary integral of
$\mathcal R^n(f\mu)$ against that cutoff difference, and tends to zero
by \eqref{eq:balanced-decay-bound}.  This proves existence and independence
of the cutoff.  Small-scale symmetric regularizations give the same
compact pairing by \eqref{eq:H-Lipschitz}--\eqref{eq:near-growth}.
This is the balanced extension of \cite[Sections~3.3--3.5]{JayeNazarovI}.
For clarity, the cutoff argument can be written without assuming that a
principal value exists on the support. Put $\nu=f\mu$ and
$M_1(\nu)=\int|W|\,d|\nu|(W)$. If $h\in\Lip_c$ vanishes in
$B(0,2R_f)$, the two supports in the pairing are separated, and oddness
and Fubini give
\begin{align}
 \langle\mathcal R^n\nu,h\rangle_\mu
 &=\frac12\iint K_n(U-V)
       \bigl(f(V)h(U)-f(U)h(V)\bigr)\,d\mu(U)d\mu(V)\notag\\
 &=\int h(U)\left[\int K_n(U-V)\,d\nu(V)\right]d\mu(U).
 \label{eq:separated-pairing}
\end{align}
The inner integral is ordinary, because $U$ lies away from $\supp\nu$.
For $S>2R_f$ and $\supp h\subset B_S^c$, the zero mass of $\nu$ yields
\begin{equation}\label{eq:balanced-tail-quantitative}
 |\langle\mathcal R^n\nu,h\rangle_\mu|
 \le C\|h\|_\infty M_1(\nu)
      \int_{|U|>S}|U|^{-n-1}\,d\mu(U)
 \le C A_0\|h\|_\infty M_1(\nu)/S.
\end{equation}
In particular, if $\chi_j$ and $\chi_k$ are two admissible cutoffs,
then $h=\chi_j-\chi_k$ vanishes on $B_{\min(j,k)}$ and
$\|h\|_\infty\le1$. Thus their pairings form a Cauchy family.
The same estimate compares two entirely different cutoff sequences.
After taking this limit, it also defines the notation
$\langle\mathcal R^n\nu,1-\chi_S\rangle_\mu$ and bounds it by
$C A_0M_1(\nu)/S$. Neither \eqref{eq:separated-pairing} nor this argument
requires a pointwise potential near the support.

\begin{definition}[Reflectionless measures]
\label{def:reflectionless}
A positive Radon measure with \eqref{eq:upper-growth} is called
reflectionless in this paper if the symmetric truncations
$\mathcal R^n_{\mu,\varepsilon}f(U)
=\int_{|U-V|>\varepsilon}K_n(U-V)f(V)\,d\mu(V)$ are uniformly bounded
on $L^2(\mu)$ and
\begin{equation}\label{eq:reflectionless}
 \langle\mathcal R^n(f\mu),1\rangle_\mu=0
 \quad\text{for all }f\in\Lip_c(\R^d),\quad\int f\,d\mu=0.
\end{equation}
The identity is the Jaye--Nazarov weak reflectionless identity.  The
operator bound is an additional standing condition specifying the class
considered here.
\end{definition}

The untruncated double integral with the test $1$ need not be absolutely
convergent, even on a plane: its far-field absolute value can behave
like $r^{-n}$, whose radial integral is logarithmically divergent.
Consequently we do not split an unbalanced full-vector pairing into
separate terms.  Absolute representations for individual finite-energy
coordinates will be established before they are used.

\begin{lemma}
\label{lem:potential-representation}
Let $f\in\Lip_c(\R^d)$, $\int f\,d\mu=0$, and suppose that the symmetric
potentials $\mathcal R^n_\varepsilon(f\mu)$ converge $\mu$-a.e. to $V$.
If the truncated operators are uniformly $L^2(\mu)$ bounded and
$V\in L^1(\mu)$, then
\begin{equation}\label{eq:potential-representation}
 \langle\mathcal R^n(f\mu),1\rangle_\mu=\int V\,d\mu.
\end{equation}
The same conclusion holds if local $L^1(\mu)$ convergence of the
truncations is known directly.
\end{lemma}
\begin{proof}
We separate the two limits. For fixed $\varepsilon>0$ and a fixed
compact cutoff $\chi$, the variables in
\[
 I_{\varepsilon,\chi}
 =\int\chi(U)\mathcal R^n_\varepsilon(f\mu)(U)\,d\mu(U)
\]
are confined to $\supp\chi\times\supp f$ and the diagonal has been
removed. This is an absolutely convergent integral. Exchanging $U,V$
in a second copy and using $K_n(V-U)=-K_n(U-V)$ gives
\begin{align}
 I_{\varepsilon,\chi}
 &=\iint_{|U-V|>\varepsilon}K_n(U-V)f(V)\chi(U)\,d\mu(U)d\mu(V)\notag\\
 &=-\iint_{|U-V|>\varepsilon}K_n(U-V)f(U)\chi(V)\,d\mu(U)d\mu(V)\notag\\
 &=\frac12\iint_{|U-V|>\varepsilon}K_n(U-V)
       \bigl(f(V)\chi(U)-f(U)\chi(V)\bigr)\,d\mu(U)d\mu(V).
 \label{eq:potential-truncated-symmetry}
\end{align}
By the compact absolute majorant, the last expression converges to
$\langle\mathcal R^n(f\mu),\chi\rangle_\mu$.

To pass to the limit in the first expression, note that compactness
of $\supp f$ gives $f\in L^2(\mu)$. If the operator bound is $M$, then
\[
 \sup_{\varepsilon>0}\|\mathcal R^n_\varepsilon(f\mu)\|_{L^2(\mu)}
 \le M\|f\|_{L^2(\mu)}.
\]
For every measurable $E\subset\supp\chi$,
\begin{equation}\label{eq:potential-uniform-integrability}
 \sup_\varepsilon\int_E|\mathcal R^n_\varepsilon(f\mu)|\,d\mu
 \le M\|f\|_{L^2(\mu)}\mu(E)^{1/2}.
\end{equation}
Thus the potentials are uniformly integrable on the finite-measure set
$\supp\chi$. Their almost-everywhere convergence, together with
\eqref{eq:potential-uniform-integrability}, implies local $L^1$ convergence
by Vitali's theorem. The argument applies to every sequence
$\varepsilon_j\downarrow0$, and hence to the full truncation limit.
We have proved
\[
 \langle\mathcal R^n(f\mu),\chi\rangle_\mu
   =\int\chi V\,d\mu.
\]
Finally take $\chi=\chi_k$. Since $0\le\chi_k\le1$, $\chi_k\to1$
pointwise and $V\in L^1(\mu)$,
\[
 \langle\mathcal R^n(f\mu),1\rangle_\mu
 =\lim_k\int\chi_kV\,d\mu=\int V\,d\mu.
\]
Under the alternative hypothesis of local $L^1$ convergence, the Vitali
step is unnecessary. Local integrability of the limit alone would not
justify that step; the uniform integrability or the stated convergence
hypothesis is essential.
\end{proof}

\subsection{Weighted graphs and flat operators}
\begin{proposition}
\label{prop:graph-structural}
If $\varphi$ is globally Lipschitz and $0<c_0\le a\le C_0$, then
$\mu_{\varphi,a}$ is globally $n$-AD regular.  It has restricted growth without any additional assumption
at infinity.  The Riesz truncations are uniformly
bounded on $L^2(\mu_{\varphi,a})$.
\end{proposition}
\begin{proof}
Put $L_M=(1+\Lip(\varphi)^2)^{1/2}$.  Then
\begin{equation}\label{eq:graph-biLip}
 |x-y|\le |X_\varphi(x)-X_\varphi(y)|\le L_M|x-y|,
\end{equation}
and
\[
 B(x,r/L_M)\subset X_\varphi^{-1}(B(X_\varphi(x),r))\subset B(x,r).
\]
Consequently
\begin{equation}\label{eq:AD-regular}
 c_0\omega_nL_M^{-n}r^n\le
 \mu_{\varphi,a}(B(X_\varphi(x),r))\le C_0\omega_nr^n.
\end{equation}
The upper bound extends to arbitrary centers by enlarging a ball meeting
the support by a factor two.  Restricted growth follows from
\eqref{eq:far-growth}.  For the unweighted graph measure, the uniform $L^2$ boundedness
of the hard truncations follows from the Lipschitz-graph result
in \cite[Part~I, Section~1.2, pp.~13--14]{DavidSemmes}.
The Riesz kernels belong to the class in Definition~1.20 of
that reference, and the truncation convention is specified in
Definition~1.23. The bound depends only on the dimensions and
the Lipschitz constant of the graph.
Multiplication by the bounded, boundedly invertible Hausdorff
density in \eqref{eq:hausdorff-density} transfers this bound
to the weighted measure. More explicitly, put
$\sigma=\mathcal H^n|_{\Gamma_\varphi}$ and $g=a/J_\varphi$. Since
$1\le J_\varphi\le C_{n,m,M}$,
\[
 0<g_-:=c_0/C_{n,m,M}\le g\le C_0=:g_+.
\]
The weighted truncation is
$\mathcal R^n_{\mu,\varepsilon}f
=\mathcal R^n_{\sigma,\varepsilon}(gf)$. If $C_\Gamma$ is the
unweighted graph bound, then
\begin{align*}
 \|\mathcal R^n_{\mu,\varepsilon}f\|_{L^2(\mu)}
 &\le g_+^{1/2}C_\Gamma\|gf\|_{L^2(\sigma)}\\
 &\le g_+C_\Gamma\|f\|_{L^2(\mu)}.
\end{align*}
This estimate is uniform in $\varepsilon$. No limiting value of $a$ or
compact support of either perturbation is involved. In particular the
large-scale lower condition holds with center $X_\varphi(0)$, while
restricted growth follows from upper growth alone.
\end{proof}

We denote the normalized vector Riesz transform on $\R^n$ by
$\mathcal R=(\mathcal R_1,\ldots,\mathcal R_n)$, with
\[
 \widehat{\mathcal R_jf}(\xi)=-i\frac{\xi_j}{|\xi|}\widehat f(\xi),
 \qquad \widehat{|D|f}(\xi)=|\xi|\widehat f(\xi).
\]
Throughout the paper we use the Fourier transform
\begin{equation}\label{eq:Fourier-convention}
 \widehat f(\xi)=\int_{\R^n}e^{-ix\cdot\xi}f(x)\,dx,
 \qquad
 f(x)=(2\pi)^{-n}\int_{\R^n}e^{ix\cdot\xi}\widehat f(\xi)\,d\xi
\end{equation}
for Schwartz functions, with the usual extensions to $L^2$ and tempered
distributions. In particular,
\begin{equation}\label{eq:Sobolev-Fourier-convention}
 \|f\|_{H^s}^2=(2\pi)^{-n}
   \int_{\R^n}(1+|\xi|^2)^s|\widehat f(\xi)|^2\,d\xi.
\end{equation}
For vector-valued functions we sum the squared component norms.
Positive constants $c_{n,R}$ and $c_{n,\Lambda}$ convert the kernels
used in \eqref{eq:Riesz-kernel} to these normalized operators.
Plancherel's theorem gives
\begin{equation}\label{eq:Riesz-coercivity}
 \sum_{j=1}^n\|\mathcal R_jf\|_{H^s}^2=\|f\|_{H^s}^2,
 \qquad \||D|u\|_{H^s}=\|Du\|_{H^s}.
\end{equation}
For real $f,g\in L^2$, $\int(\mathcal R_jf)g=-\int f\mathcal R_jg$.
We use the standard $L^2$ principal-value realization of these flat
operators \cite{Stein}.

\begin{lemma}[Fractional Poincar\'e inequality]\label{lem:fractional-Poincare}
If $u\in H^{1/2}(\R^n)$ and $\supp u\subset B_R$, then
\begin{equation}\label{eq:fractional-Poincare}
 \|u\|_2\le C_nR^{1/2}[u]_{1/2}.
\end{equation}
\end{lemma}
\begin{proof}
For $x\in B_R$, restrict the $y$-integral in \eqref{eq:half-seminorm} to
$B_{3R}\setminus B_{2R}$.  There $u(y)=0$ and $|x-y|\le4R$, so
\[
 \int_{B_{3R}\setminus B_{2R}}
 \frac{|u(x)-u(y)|^2}{|x-y|^{n+1}}\,dy
 \ge c_nR^{-1}|u(x)|^2.
\]
Integrate in $x$.  The vector-valued assertion follows by summing components.
\end{proof}

\section{Proof of Theorem~\ref{prop:ambient-planarity} and Theorem~\ref{thm:global-rigidity}: Finite orthogonal energy and rigidity}\label{sec:balancing}
This section proves Theorems \ref{prop:ambient-planarity} and \ref{thm:global-rigidity} by converting finite orthogonal energy into planar support and then recovering the density. Under the stated growth assumptions, energy cutoff estimates and normalized coordinate tests allow us to use reflectionlessness to show that each orthogonal coordinate is constant on the support. For weighted Lipschitz graphs, the flat Riesz identity and the \(L^2\) density perturbation then force the density to be constant. The support theorem also provides the geometric reduction needed for the compact perturbation result in Section \ref{sec:compact-flat}.
\subsection{An energy cutoff lemma}
Translate the reference point $O$ in \eqref{eq:large-lower-growth} to the
origin.  In this section write
\begin{equation}\label{eq:ambient-bilinear}
 \mathscr B_\mu(u,v)=\frac12\iint
 \frac{(u(U)-u(V))(v(U)-v(V))}{|U-V|^{n+1}}\,d\mu(U)d\mu(V),
 \qquad \mathscr E_\mu(u)=\mathscr B_\mu(u,u).
\end{equation}
All mixed integrals between finite-energy functions are absolutely
convergent and satisfy
\begin{equation}\label{eq:energy-CS}
 |\mathscr B_\mu(u,v)|\le
 \mathscr E_\mu(u)^{1/2}\mathscr E_\mu(v)^{1/2}.
\end{equation}

\begin{lemma}\label{lem:energy-cutoff}
Assume \eqref{eq:upper-growth} and \eqref{eq:large-lower-growth}.
Let $u\in L^2_{\rm loc}(\mu)$ have finite $\mathscr E_\mu(u)$.  Choose
$0\le\chi\le1$, $\chi\in C_c^\infty(B(0,2))$, $\chi=1$ on $B(0,1)$.
For $R\ge R_*$ put
\begin{equation}\label{eq:mean-cutoff}
 \chi_R(U)=\chi(U/R),\quad
 c_R=\frac1{\mu(B_{4R})}\int_{B_{4R}}u\,d\mu,\quad
 v_R=\chi_R(u-c_R).
\end{equation}
Then
\begin{equation}\label{eq:uniform-cutoff-energy}
 \mathscr E_\mu(v_R)\le C\mathscr E_\mu(u),
 \qquad
 \mathscr B_\mu(u,v_R)\longrightarrow\mathscr E_\mu(u),
\end{equation}
with $C$ independent of $R$.  If $u$ is an ambient Lipschitz function,
then $v_R\in\Lip_c(\R^d)$.
\end{lemma}
\begin{proof}
Throughout this proof $B_r=B(0,r)$ is an ambient ball and
\[
 \mathscr E_{\mu,B_r}(u)
 =\frac12\iint_{B_r\times B_r}
   \frac{|u(U)-u(V)|^2}{|U-V|^{n+1}}\,d\mu(U)d\mu(V).
\]
All integrals involving $u$ on a fixed ball are legitimate because
$u\in L^2_{\rm loc}(\mu)$ and $\mu$ is locally finite.

Writing $m_R=\mu(B_{4R})$ and expanding the square gives
\[
 \iint_{B_{4R}^2}|u(U)-u(V)|^2\,d\mu(U)d\mu(V)
 =2m_R\int_{B_{4R}}|u|^2\,d\mu
  -2\left|\int_{B_{4R}}u\,d\mu\right|^2.
\]
Since $c_R=m_R^{-1}\int_{B_{4R}}u\,d\mu$, this yields
\begin{align}
 \int_{B_{4R}}|u-c_R|^2\,d\mu
 &=\frac1{2m_R}\iint_{B_{4R}^2}|u(U)-u(V)|^2\,d\mu(U)d\mu(V)\notag\\
 &\le\frac{(8R)^{n+1}}{m_R}\mathscr E_{\mu,B_{4R}}(u)
 \le C a_0^{-1}R\mathscr E_{\mu,B_{4R}}(u).
 \label{eq:local-variance}
\end{align}
The last inequality uses only the lower bound at the one fixed center:
$m_R\ge a_0(4R)^n$. No lower bound on balls centered at $U$ or $V$
is being used.

For $V\in\supp\mu$ define
\[
 Q_R(V)=\int
 \frac{|\chi_R(U)-\chi_R(V)|^2}{|U-V|^{n+1}}\,d\mu(U).
\]
The scaled cutoff satisfies
$|\chi_R(U)-\chi_R(V)|\le C\min\{|U-V|/R,1\}$.
Splitting at distance $R$ and using the two growth integrals gives
\begin{align}
 Q_R(V)
 &\le CR^{-2}\int_{0<|U-V|<R}|U-V|^{1-n}\,d\mu(U)\notag\\
 &\quad+C\int_{|U-V|\ge R}|U-V|^{-n-1}\,d\mu(U)
 \le C A_0/R.
 \label{eq:cutoff-kernel-bound}
\end{align}
The estimate is uniform in the center $V$ and the scale $R$.

For $U,V\in B_{4R}$ the identity
\[
 v_R(U)-v_R(V)=\chi_R(U)(u(U)-u(V))
 +(u(V)-c_R)(\chi_R(U)-\chi_R(V))
\]
and $|A+B|^2\le2|A|^2+2|B|^2$ imply
\begin{align}
 \mathscr E_{\mu,B_{4R}}(v_R)
 &\le 2\mathscr E_{\mu,B_{4R}}(u)
 +\int_{B_{4R}}|u(V)-c_R|^2Q_R(V)\,d\mu(V)\notag\\
 &\le C\mathscr E_\mu(u).
 \label{eq:cutoff-interior-detail}
\end{align}
For pairs outside $B_{4R}^2$, a nonzero difference of $v_R$ requires
one variable, say $U$, to lie in $B_{2R}$ and the other to lie outside
$B_{4R}$. Their distance is at least $2R$. By symmetry the factor $1/2$
in the energy cancels the two possible orders of such pairs, and
\begin{align}
 \mathscr E_\mu(v_R)-\mathscr E_{\mu,B_{4R}}(v_R)
 &=\int_{B_{2R}}|v_R(U)|^2
       \int_{B_{4R}^c}|U-V|^{-n-1}\,d\mu(V)d\mu(U)\notag\\
 &\le C A_0R^{-1}\int_{B_{4R}}|u-c_R|^2\,d\mu
 \le C\mathscr E_\mu(u).
 \label{eq:cutoff-exterior-detail}
\end{align}
This proves the first assertion of \eqref{eq:uniform-cutoff-energy}.

Let $D_R=(\R^d)^2\setminus(B_R\times B_R)$ and set
\[
 e_R=\frac12\iint_{D_R}
    \frac{|u(U)-u(V)|^2}{|U-V|^{n+1}}\,d\mu(U)d\mu(V).
\]
The integrand belongs to $L^1(\mu\otimes\mu)$ and the indicators of
$D_R$ decrease pointwise to zero. Therefore $e_R\to0$.
On $B_R$, $u-v_R=c_R$; hence the difference of $u-v_R$ vanishes for
pairs in $B_R^2$. Cauchy--Schwarz on the remaining pairs gives
\begin{align}
 |\mathscr E_\mu(u)-\mathscr B_\mu(u,v_R)|
 &=|\mathscr B_\mu(u,u-v_R)|\notag\\
 &\le e_R^{1/2}\mathscr E_\mu(u-v_R)^{1/2}\notag\\
 &\le e_R^{1/2}
       \bigl(\mathscr E_\mu(u)^{1/2}+\mathscr E_\mu(v_R)^{1/2}\bigr)
 \le C e_R^{1/2}\mathscr E_\mu(u)^{1/2}\longrightarrow0.
 \label{eq:cutoff-energy-limit}
\end{align}
This argument does not claim strong energy convergence of $v_R$ to $u$
or convergence of $c_R$. Only the displayed pairing limit is needed.
If $u$ is an ambient Lipschitz function, then $u-c_R$ is bounded on the
compact support of $\chi_R$, and its product with $\chi_R$ is a globally
Lipschitz function of compact support. Thus it is an admissible test.
\end{proof}

The same lemma holds for $dx$ on $\R^n$ and any symmetric weight
$w(x,y)$ comparable to $|x-y|^{-n-1}$.  One first proves the uniform
cutoff bound for the unweighted seminorm, uses comparability, and then
uses the weighted energy tail in \eqref{eq:cutoff-energy-limit}.  The
averages may be ordinary Lebesgue averages in this version.

\begin{lemma}
\label{lem:expanding-balance}
Under the same growth assumptions there exist $\beta_t\in\Lip_c(\R^d)$,
$t\ge R_*$, such that
\begin{equation}\label{eq:beta-energy}
 \int\beta_t\,d\mu=1,\qquad
 \mathscr E_\mu(\beta_t)\le Ct^{-n-1}.
\end{equation}
\end{lemma}
\begin{proof}
Choose $\eta\in C_c^\infty(B_2)$ with $0\le\eta\le1$ and $\eta=1$
on $B_1$, and put
\[
 \eta_t(U)=\eta(U/t),\qquad
 m_t=\int\eta_t\,d\mu,\qquad \beta_t=m_t^{-1}\eta_t.
\]
For $t\ge R_*$ the hypotheses give
\begin{equation}\label{eq:balance-mass-comparison}
 a_0t^n\le\mu(B_t)\le m_t\le\mu(B_{2t})\le 2^nA_0t^n.
\end{equation}
In particular the normalization is well defined, positive, and satisfies
$\int\beta_t\,d\mu=1$. The pointwise bounds
\[
 \|\beta_t\|_\infty\le a_0^{-1}t^{-n},\qquad
 \Lip(\beta_t)\le C a_0^{-1}t^{-n-1}
\]
alone do not prove the required energy estimate on an infinite measure;
one must also use the compact support of the numerator.

The difference $\eta_t(U)-\eta_t(V)$ is zero when both variables are
outside $B_{2t}$. Thus symmetry and \eqref{eq:cutoff-kernel-bound} give
\begin{align}
 \mathscr E_\mu(\eta_t)
 &\le\int_{B_{2t}}\int
   \frac{|\eta_t(U)-\eta_t(V)|^2}{|U-V|^{n+1}}\,d\mu(V)d\mu(U)\notag\\
 &\le \mu(B_{2t})\sup_{U\in\supp\mu}
   \int\frac{|\eta_t(U)-\eta_t(V)|^2}{|U-V|^{n+1}}\,d\mu(V)\notag\\
 &\le C A_0^2t^{n-1}.
 \label{eq:balance-unnormalized-energy}
\end{align}
Finally
\[
 \mathscr E_\mu(\beta_t)=m_t^{-2}\mathscr E_\mu(\eta_t)
 \le C A_0^2a_0^{-2}t^{-2n}t^{n-1}
 =C A_0^2a_0^{-2}t^{-n-1}.
\]
The test has unit mass but vanishing energy. Its support need not lie
in a flat part of the measure, and the proof uses no geometrical
information about that support beyond the two growth assumptions.
\end{proof}

\subsection{Proof of the support theorem}
\begin{lemma}
\label{lem:normal-pairing-representation}
Let $e\in\R^d$ be a unit vector and $h_e(U)=e\cdot U$.
Assume $\mathscr E_\mu(h_e)<\infty$.  For every $g\in\Lip_c(\R^d)$,
\begin{equation}\label{eq:coordinate-J}
 J_e(g):=\frac12\iint
 \frac{e\cdot(U-V)}{|U-V|^{n+1}}(g(V)-g(U))\,d\mu(U)d\mu(V)
 =-\mathscr B_\mu(h_e,g)
\end{equation}
is absolutely convergent.  If $\int g\,d\mu=0$, it is the $e$-component
of \eqref{eq:JN-balanced-extension}.
\end{lemma}
\begin{proof}
First verify that $g$ has finite energy. Take $H$ with
$\supp g\subset B_H$. If both variables are outside $B_H$, the
difference of $g$ is zero. For the other pairs,
\begin{align*}
 \mathscr E_\mu(g)
 &\le\int_{B_H}\left[
 \Lip(g)^2\int_{|U-V|<1}|U-V|^{1-n}\,d\mu(V)
 +4\|g\|_\infty^2\int_{|U-V|\ge1}|U-V|^{-n-1}\,d\mu(V)
 \right]d\mu(U)\\
 &\le C A_0\mu(B_H)\bigl(\Lip(g)^2+\|g\|_\infty^2\bigr)<\infty.
\end{align*}
Since $h_e(U)-h_e(V)=e\cdot(U-V)$, Cauchy--Schwarz now gives
\begin{equation}\label{eq:coordinate-absolute-bound}
 \frac12\iint
 \frac{|e\cdot(U-V)|\,|g(V)-g(U)|}{|U-V|^{n+1}}\,d\mu(U)d\mu(V)
 \le\mathscr E_\mu(h_e)^{1/2}\mathscr E_\mu(g)^{1/2}.
\end{equation}
This proves absolute convergence and the sign in
$J_e(g)=-\mathscr B_\mu(h_e,g)$.

Suppose now that $\int g\,d\mu=0$. For large $k$, the compact cutoff
$\chi_k$ equals one on $\supp g$. The difference factor in the compact
pairing is
$g(V)\chi_k(U)-g(U)\chi_k(V)$. If both values of $g$ are nonzero,
this is exactly $g(V)-g(U)$. If only $g(V)$ is nonzero, its absolute
value is at most $|g(V)|=|g(V)-g(U)|$, and the case where only $g(U)$
is nonzero is identical. If both are zero, both expressions vanish.
Thus in all cases
\[
 |g(V)\chi_k(U)-g(U)\chi_k(V)|\le|g(V)-g(U)|.
\]
The right side, multiplied by the projected kernel, is the integrable
majorant in \eqref{eq:coordinate-absolute-bound}. Dominated convergence
therefore gives
\[
 e\cdot\langle\mathcal R^n(g\mu),1\rangle_\mu
 =\lim_k e\cdot\langle\mathcal R^n(g\mu),\chi_k\rangle_\mu=J_e(g).
\]
The absolute scalar integral $J_e(g)$ exists even without the zero-mean
condition. Only its identification with the full balanced vector pairing
uses that condition. This distinction will be maintained in the proof
below.
\end{proof}

\begin{proof}[Proof of Theorem~\ref{prop:ambient-planarity}]
Let $e_1,\ldots,e_m$ be an orthonormal basis of $P^\perp$. Then
\[
 \mathscr E_P^\perp(\mu)=\sum_{\ell=1}^m\mathscr E_\mu(h_{e_\ell}),
 \qquad h_e(U)=e\cdot U.
\]
Fix one $e=e_\ell$ and write $h=h_e$. The energy of $h$ is finite and
$h\in L^2_{\rm loc}(\mu)$ because $h$ is bounded on bounded sets.

For a fixed $g\in\Lip_c(\R^d)$ set $M_g=\int g\,d\mu$ and
$g_t=g-M_g\beta_t$. Each $g_t$ is an ambient compact Lipschitz function,
although its support is allowed to grow with $t$, and
\[
 \int g_t\,d\mu=M_g-M_g\int\beta_t\,d\mu=0.
\]
Since $g_t$ has zero mean, reflectionlessness applies to this test.
By Lemma~\ref{lem:normal-pairing-representation},
\begin{align}
 0=e\cdot\langle\mathcal R^n(g_t\mu),1\rangle_\mu
 &=J_e(g_t)\notag\\
 &=-\mathscr B_\mu(h,g)+M_g\mathscr B_\mu(h,\beta_t).
 \label{eq:planarity-balanced-test-detail}
\end{align}
The splitting on the last line is a splitting of absolutely convergent
energy integrals, not of unbalanced full-vector pairings. Hence
\[
 |\mathscr B_\mu(h,g)|
 \le |M_g|\mathscr E_\mu(h)^{1/2}\mathscr E_\mu(\beta_t)^{1/2}
 \le C|M_g|\mathscr E_\mu(h)^{1/2}t^{-(n+1)/2}\longrightarrow0.
\]
The quantities $M_g$ and $\mathscr E_\mu(h)$ are fixed when $t$ tends
to infinity. We have proved
\begin{equation}\label{eq:all-compact-tests-annihilated}
 \mathscr B_\mu(h,g)=0\qquad(g\in\Lip_c(\R^d)).
\end{equation}

Choose next
$g=v_R=\chi_R(h-c_R)$ from Lemma~\ref{lem:energy-cutoff}. For each
fixed $R$ this is an admissible test in
\eqref{eq:all-compact-tests-annihilated}. Consequently
\[
 0=\mathscr B_\mu(h,v_R)\longrightarrow\mathscr E_\mu(h).
\]
Thus $\mathscr E_\mu(h)=0$. The limit in $t$ is taken first for each fixed $R$; no uniform bound
for $M_{v_R}$ is required.

The kernel $|U-V|^{-n-1}$ is strictly positive off the product diagonal,
which is null for $\mu\otimes\mu$. Nonnegativity therefore implies
$h(U)=h(V)$ for almost every pair. By Fubini there is a $V_e$ for which
$h(U)=h(V_e)=:c_e$ for $\mu$-almost every $U$. The measure is nonzero
by the lower growth hypothesis. If some $U_0\in\supp\mu$ had
$h(U_0)\ne c_e$, continuity of $h$ would give a ball of positive
$\mu$-mass on which $|h-c_e|$ is bounded below, a contradiction.
Hence $h=c_e$ on the entire support.

Put $b=\sum_{\ell=1}^m c_{e_\ell}e_\ell\in P^\perp$. All the preceding
coordinate equalities say $\Pi_{P^\perp}U=b$ for $U\in\supp\mu$, or
equivalently $\supp\mu\subset b+P$, as required. The argument proves
support rigidity only; it makes no assertion about the density on that
plane.
\end{proof}

The support theorem uses the coordinate-testing method of
\cite[Section~10]{JayeNazarovII}; the finite orthogonal energy and the
uniform mean-subtracted cutoff bound are the hypotheses and estimates
used here to accommodate an infinite background.

\subsection{The density on the plane}
\begin{lemma}\label{lem:plane-reflectionless}
Let $\mu=a\,dx$ on $\R^n$, where $0<c_0\le a\le C_0$ and
$\rho=a-a_\infty\in L^2$.  If $\mu$ is reflectionless, then
\begin{equation}\label{eq:flat-reflectionless}
 \int f(x)a(x)\mathcal R\rho(x)\,dx=0
 \quad\text{whenever }f\in\Lip_c(\R^n),\quad\int fa=0.
\end{equation}
\end{lemma}
\begin{proof}
Fix a real $f\in\Lip_c(\R^n)$ with $\int fa=0$. To use the ambient
weak definition, take $\theta\in C_c^\infty(\R^m)$ equal to one near
zero and extend $f$ by $\widetilde f(x,z)=f(x)\theta(z)$. On the plane
this agrees with $f$. Put $h=fa$. Then
\[
 \supp h\subset\supp f,\qquad h\in L^1\cap L^2,\qquad\int h=0.
\]
In particular $h$ need not be Lipschitz: boundedness of $a$ is sufficient
for these properties, and flat $L^2$ Riesz theory applies to $h$.
Choose $R_f\ge1$ containing its support. For $|x|>2R_f$,
\begin{align*}
 \mathcal Rh(x)
 &=c_{n,R}^{-1}\int\bigl(K_n(x-y)-K_n(x)\bigr)h(y)\,dy,\\
 K_n(x-y)-K_n(x)&=-\int_0^1 DK_n(x-ty)y\,dt.
\end{align*}
Since $|x-ty|\ge|x|/2$, this proves
\begin{equation}\label{eq:h-balanced-tail}
 |\mathcal Rh(x)|\le C|x|^{-n-1}\int|y||h(y)|\,dy.
\end{equation}
On $B_{2R_f}$, Cauchy--Schwarz and the $L^2$ Riesz bound give
\[
 \int_{B_{2R_f}}|\mathcal Rh|\,dx
 \le |B_{2R_f}|^{1/2}\|\mathcal Rh\|_2<\infty.
\]
Outside that ball use
$\int_{2R_f}^\infty r^{-n-1}r^{n-1}\,dr<\infty$.
Thus $\mathcal Rh\in L^1(dx)$ and, since $a\le C_0$, also in
$L^1(a\,dx)$. The flat principal value exists almost everywhere and
represents the same $L^2$ operator. Lemma~\ref{lem:potential-representation}
and reflectionlessness now give, after division by $c_{n,R}$,
\begin{equation}\label{eq:flat-potential-zero}
 0=\int a(x)\mathcal Rh(x)\,dx.
\end{equation}

It remains to remove the constant background. We do not assign an
unrenormalized whole-space Riesz integral to the constant function.
Instead, the zero mean of $h$ gives
\begin{align}
 \widehat h(\xi)
 &=\int(e^{-iy\cdot\xi}-1)h(y)\,dy,\notag\\
 |\widehat h(\xi)|&\le|\xi|\int|y||h(y)|\,dy.
 \label{eq:flat-h-low-frequency}
\end{align}
For $\xi\ne0$ the $L^2$ multiplier formula is
\[
 \widehat{\mathcal R_jh}(\xi)=-i\frac{\xi_j}{|\xi|}\widehat h(\xi).
\]
Its right side is continuous away from zero and has limit zero at zero
by \eqref{eq:flat-h-low-frequency}. The left side has a continuous
representative because $\mathcal R_jh\in L^1$. Equality almost everywhere
of these two continuous representatives implies equality everywhere.
Evaluating at zero therefore yields
\begin{equation}\label{eq:R-h-integral-zero}
 \int\mathcal R_jh(x)\,dx=\widehat{\mathcal R_jh}(0)=0.
\end{equation}
Finally $\rho,h\in L^2$, so $\int\rho\mathcal Rh$ and
$\int h\mathcal R\rho$ are absolutely integrable by Cauchy--Schwarz.
Skew-adjointness, component by component, gives
\begin{align*}
 0=\int a\mathcal Rh
 &=a_\infty\int\mathcal Rh+\int\rho\mathcal Rh\\
 &=-\int h\mathcal R\rho=-\int fa\mathcal R\rho.
\end{align*}
The compactness used in this proof is that of the test density $h=fa$.
The perturbation $\rho$ is used only in $L^2$, and need not be integrable
or compactly supported.
\end{proof}

\begin{lemma}
\label{lem:weighted-annihilation}
Let $0<c_0\le a\le C_0$ and $F\in L^2_{\rm loc}(\R^n)$.
If $\int faF=0$ for every $f\in\Lip_c$ with $\int fa=0$, then $F$ is
constant a.e.
\end{lemma}
\begin{proof}
It suffices to argue for one real component of $F$. Choose a nonnegative
$\eta_0\in C_c^\infty$, $\eta_0\not\equiv0$. The lower bound for $a$
ensures $0<\int\eta_0a<\infty$. Put
\[
 \eta=\frac{\eta_0}{\int\eta_0a},\qquad
 c=\int\eta(x)a(x)F(x)\,dx.
\]
This constant is well defined because $a$ is bounded and $F$ is locally
square integrable. For arbitrary $q\in C_c^\infty$ define
$g=q-(\int qa)\eta$. Then $g\in\Lip_c$ and
\[
 \int ga=\int qa-\left(\int qa\right)\int\eta a=0.
\]
Applying the hypothesis to this $g$ gives
\begin{align*}
 0=\int gaF
 &=\int qaF-\left(\int qa\right)\left(\int\eta aF\right)\\
 &=\int q(x)a(x)(F(x)-c)\,dx.
\end{align*}
Thus the locally integrable function $a(F-c)$ is the zero distribution,
and so it is zero almost everywhere. Since $a\ge c_0>0$, $F=c$ almost
everywhere. The tests $q$ and $\eta$ are smooth. The measurable density occurs
only in the integrals, so no regularity of $1/a$ is required. For vector $F$, apply
the same construction componentwise.
\end{proof}

\begin{proof}[Proof of Theorem~\ref{thm:global-rigidity}]
Proposition~\ref{prop:graph-structural} gives the upper growth and the
large-scale lower bound, with center $X_\varphi(0)$. Set
$P=\R^n\times\{0\}$. The pushforward formula, applied first to
nonnegative functions and then by monotone convergence, yields
\begin{align}
 \mathscr E_P^\perp(\mu_{\varphi,a})
 &=\frac12\iint_{\R^n\times\R^n}
 \frac{|\varphi(x)-\varphi(y)|^2a(x)a(y)}
 {|X_\varphi(x)-X_\varphi(y)|^{n+1}}\,dx\,dy\notag\\
 &\le\frac{C_0^2}{2}\iint
       \frac{|\varphi(x)-\varphi(y)|^2}{|x-y|^{n+1}}\,dx\,dy
 =\frac{C_0^2}{2}[\varphi]_{1/2}^2<\infty.
 \label{eq:graph-energy-to-ambient}
\end{align}
Every point of $\Gamma_\varphi$ belongs to $\supp\mu_{\varphi,a}$:
the preimage of a ball around $X_\varphi(x)$ contains an ordinary ball
around $x$, and the density on that ball is at least $c_0$.
Conversely the graph is closed and carries the whole measure.
Consequently $\supp\mu_{\varphi,a}=\Gamma_\varphi$.
Theorem~\ref{prop:ambient-planarity} now gives
\[
 (x,\varphi(x))\in\R^n\times\{b\}\quad(x\in\R^n),
\]
so $\varphi(x)=b$ for every $x$.

Translate this plane vertically to $P$. Such a rigid motion preserves
the growth bounds, the symmetric kernel pairing, and reflectionlessness.
The remaining measure is $a(x)\,dx$ on $\R^n$. Put
$\rho=a-a_\infty\in L^2$. Lemma~\ref{lem:plane-reflectionless} supplies
the zero-mean annihilation identity, and
Lemma~\ref{lem:weighted-annihilation} gives
$\mathcal R\rho=c$ almost everywhere for one constant vector $c$.
Plancherel implies $\mathcal R\rho\in L^2$. If $c\ne0$, then
\[
 \|\mathcal R\rho\|_2^2
 \ge\int_{B_t}|c|^2\,dx=|c|^2\omega_nt^n\longrightarrow\infty,
\]
a contradiction. Hence all Riesz components vanish. By
\eqref{eq:Riesz-coercivity},
\[
 \|\rho\|_2^2=\sum_{j=1}^n\|\mathcal R_j\rho\|_2^2=0.
\]
This proves $a=a_\infty$ almost everywhere and the claimed equality of
measures. Finally, a nonzero constant map is not in $L^2(\R^n)$ and
cannot equal zero outside a compact set. Either additional normalization
therefore forces $b=0$.
\end{proof}

\begin{corollary}
\label{cor:compact-graph-L2}
In Theorem~\ref{thm:global-rigidity}, finite graph energy may in particular
be replaced by $\varphi\in W^{1,\infty}$ with compact support.
The density perturbation still need only belong to $L^2$, independently
of the size of the graph perturbation.
\end{corollary}
\begin{proof}
Near the diagonal use $|\varphi(x)-\varphi(y)|^2\le M^2|x-y|^2$.
If one variable is far from the support of $\varphi$, use the tail
$|x-y|^{-n-1}$.  These bounds show $[\varphi]_{1/2}<\infty$.
Apply the theorem and then the exterior zero normalization.
\end{proof}

\section{Proof of the Theorem~\ref{thm:compact-flat}: Compact perturbations of a flat measure and stability}\label{sec:compact-flat}
This section proves Theorem \ref{sec:compact-flat} by combining the support theorem with a Fourier argument for the anti-locality of the Riesz transform. Weak closure and a compactness argument then yield a criterion relating approximation by the same flat measure on finitely many exterior annuli to closeness at the unit scale. For weighted Lipschitz graphs, we also use the energy form to establish half-order estimates for the weak orthogonal field and obtain \(L^2\) consequences under additional hypotheses. These results connect the rigidity theory with quantitative estimates preceding the higher-order Sobolev analysis.
\subsection{A consequence of the anti-locality of the Riesz transform}
The following observation is a compact-support special case of the
fractional unique-continuation principle; compare
\cite[Theorem~1.2]{GhoshSaloUhlmann}.  We include a Fourier proof to keep
the density argument self-contained.

\begin{lemma}
\label{lem:anti-local}
Let $\rho\in L^2(\R^n)$ have compact support.  If every
$\mathcal R_j\rho$ also has compact support, then $\rho=0$.
\end{lemma}
\begin{proof}
Choose $R$ large enough to contain the supports of $\rho$ and all
$\mathcal R_j\rho$. They are in $L^1$ by Cauchy--Schwarz on $B_R$.
For $z\in\mathbb C^n$ set
\[
 F(z)=\int e^{-ix\cdot z}\rho(x)\,dx,\qquad
 F_j(z)=\int e^{-ix\cdot z}\mathcal R_j\rho(x)\,dx.
\]
On a compact complex-frequency set with $|\operatorname{Im}z|\le H$,
\[
 |x^\alpha e^{-ix\cdot z}\rho(x)|
 \le R^{|\alpha|}e^{RH}|\rho(x)|\mathbf1_{B_R}(x),
\]
and the analogous bound holds for $\mathcal R_j\rho$. Differentiation
under the integral is therefore justified to every order. In particular
$F$ and $F_j$ are entire; this elementary fact is the only
compact-support Fourier analyticity used here.

The $L^2$ Riesz identity gives
$F_j(\xi)=-i\xi_j|\xi|^{-1}F(\xi)$ for almost every real $\xi\ne0$.
Both sides are continuous there, so the identity holds for every such
$\xi$. Define
\[
 G(z)=i\sum_{j=1}^n z_jF_j(z).
\]
This is entire and, for real $\xi\ne0$,
\begin{equation}\label{eq:entire-modulus-identity}
 G(\xi)=\sum_{j=1}^n\frac{\xi_j^2}{|\xi|}F(\xi)=|\xi|F(\xi).
\end{equation}
Fix $\theta\in\mathbb S^{n-1}$ and form the entire one-variable
functions $p(t)=G(t\theta)$ and $q(t)=F(t\theta)$. For real $t>0$,
\[
 p(t)=tq(t).
\]
The entire function $p(t)-tq(t)$ vanishes on a real interval and hence
vanishes identically by the identity theorem. For real $t<0$,
\eqref{eq:entire-modulus-identity} also gives $p(t)=-tq(t)$.
Subtracting the two identities yields $2tq(t)=0$ for $t<0$.
The identity theorem again implies $q\equiv0$. Every nonzero real
frequency lies on such a line, and continuity covers zero, so
$\widehat\rho=0$ on $\R^n$. Plancherel or Fourier uniqueness gives
$\rho=0$. For $n=1$ the same line argument applies without change.
\end{proof}

\subsection{Proof of the classification theorem}
\begin{proof}[Proof of Theorem~\ref{thm:compact-flat}]

After a rigid motion taking the given affine plane to
$P=\R^n\times\{0\}$, enlargement of the exceptional ball, and division
of the measure by $c$, we may assume
\[
 \mu=\mu_0:=\mathcal H^n|_P\quad\hbox{on }\overline B_{R_0}^{\,c},
 \qquad R_0\ge1.
\]
All these operations preserve the weak reflectionless identity;
rescaling the measure multiplies its pairing by $c^{-2}$ and only
changes its operator bound by a finite constant. For $r\ge2R_0$,
\begin{equation}\label{eq:exterior-large-mass}
 \mu(B_r)\ge\mathcal H^n(P\cap(B_r\setminus\overline B_{R_0}))
 =\omega_n(r^n-R_0^n)\ge\omega_n(1-2^{-n})r^n.
\end{equation}
Thus the lower growth assumption of the support theorem is available.

Fix $e\in P^\perp$, $|e|=1$, and let $h_e(U)=e\cdot U$.
The integrand defining $\mathscr E_\mu(h_e)$ is zero for almost every
pair with both variables outside $\overline B_{R_0}$, since both points
then lie on $P$. On $B_{3R_0}^2$ it is bounded by $|U-V|^{1-n}$.
For $|U-V|<1$, \eqref{eq:near-growth} gives a finite inner integral,
uniform in $U$. For $|U-V|\ge1$ the kernel is at most one and the
product measure of the bounded set is finite. This proves finiteness
of the bounded part of the energy.

In the remaining nonzero region one variable belongs to
$\overline B_{R_0}$; by symmetry assume it is $U$. For $V\in P$,
$|V|>3R_0$, one has
\[
 |h_e(U)-h_e(V)|=|h_e(U)|\le R_0,\qquad
 |U-V|\ge\tfrac23|V|.
\]
Hence the far contribution is bounded by
\begin{align}
 C R_0^2\mu(\overline B_{R_0})
   \int_{P\cap\{|V|>3R_0\}}|V|^{-n-1}\,d\mathcal H^n(V)
 &=C_nR_0^2\mu(\overline B_{R_0})\int_{3R_0}^\infty r^{-2}\,dr\notag\\
 &<\infty.\label{eq:exterior-energy-tail-detail}
\end{align}
Summing over a basis of $P^\perp$ gives finite orthogonal energy.
Theorem~\ref{prop:ambient-planarity} yields $\supp\mu\subset b+P$.
There are support points in the exterior part of $P$ because the
measure there equals $\mathcal H^n|_P$. Parallel planes $b+P$ and $P$
with $b\in P^\perp$ intersect only if $b=0$. Thus $\supp\mu\subset P$.

Identify $P$ with $\R^n$. If $E\subset P$ has Lebesgue measure zero,
then for every $\delta>0$ it admits a countable cover by planar cubes
$Q_i$ with $\sum_i\ell(Q_i)^n<\delta$. Each cube lies in an ambient
ball of radius $C_n\ell(Q_i)$, and so
\[
 \mu(E)\le\sum_i\mu(Q_i)\le C_nA_0\sum_i\ell(Q_i)^n
 <C_nA_0\delta.
\]
It follows that $\mu\ll dx$. By the Radon--Nikodym theorem and
Lebesgue differentiation, for almost every $x$,
\[
 a(x)=\lim_{r\downarrow0}\frac{\mu(B_P(x,r))}{\omega_nr^n}
 \le A_0/\omega_n.
\]
Consequently
\begin{equation}\label{eq:plane-density-holes}
 d\mu=a(x)\,dx,\qquad 0\le a\le C_nA_0.
\end{equation}
The density may still vanish in part of the exceptional ball. In
particular, Lemma~\ref{lem:weighted-annihilation} cannot yet be applied
to all of $P$. Instead, exterior flatness gives
\begin{equation}\label{eq:compact-rho}
 \rho:=a-1\in L_c^\infty(\R^n)\subset L^1\cap L^2.
\end{equation}

Fix $R_1>R_0$ and put $\Omega=\{x:|x|>R_1\}$. For
$f\in C_c^\infty(\Omega)$ satisfying $\int f\,dx=0$, the identity
$a=1$ on $\Omega$ gives $fa=f$ and $\int f\,d\mu=0$.
An ambient compact extension of $f$ is obtained as in the flat lemma.
For this smooth $f$, symmetric flat truncations converge locally
uniformly to $\mathcal Rf$. Their integrability at infinity follows
from the zero mean, exactly as in \eqref{eq:h-balanced-tail}.
The representation lemma, followed by the flat skew-adjoint identity,
therefore gives
\begin{equation}\label{eq:exterior-flat-pairing}
 0=\int a\mathcal Rf
 =\int\mathcal Rf+\int\rho\mathcal Rf
 =-\int f\mathcal R\rho.
\end{equation}
Here $\int\mathcal Rf=0$ follows from
\eqref{eq:R-h-integral-zero}, and the products with $\rho$ are $L^2$
pairings. A positive lower bound for $a$ is not used in this argument.

Choose $\eta\in C_c^\infty(\Omega)$ with $\int\eta=1$ and put
$c_\Omega=\int\eta\mathcal R\rho$. Given any
$q\in C_c^\infty(\Omega)$, take $f=q-(\int q)\eta$ in
\eqref{eq:exterior-flat-pairing}; it follows that
\[
 \int q(x)(\mathcal R\rho(x)-c_\Omega)\,dx=0.
\]
Thus the same constant vector $c_\Omega$ occurs throughout $\Omega$.
This test uses one fixed $\eta$ and remains valid when $\Omega$ has
two connected components, as it does for $n=1$.

For $|x|>2R_1$, compact support of $\rho$ gives the ordinary integral
estimate
\[
 |\mathcal R\rho(x)|
 \le C_n\int\frac{|\rho(y)|}{|x-y|^n}\,dy
 \le C_n|x|^{-n}\|\rho\|_1\longrightarrow0.
\]
The field is continuous on the exterior, so its almost-everywhere
constant value is zero there. Hence every $\mathcal R_j\rho$ is an
$L^2$ function supported in $\overline B_{R_1}$. Apply
Lemma~\ref{lem:anti-local} to obtain $\rho=0$. Therefore $a=1$ almost
everywhere, including the exceptional region. Undoing the rigid motion
and the density normalization proves the original conclusion.
\end{proof}

No parameterization of the exceptional set was used.  The upper growth
condition ensures both diagonal integrability and the bounded planar
density in \eqref{eq:plane-density-holes}.  The exterior hypothesis has
two separate roles: it supplies large-scale mass and finite orthogonal
energy, and it supplies the open exterior region needed to eliminate
possible holes in the planar density.

\subsection{Weak limits in the uniformly bounded class}
We record the precise compactness statement used below.
The weak-identity part is the closure result of
\cite[Section~8, Corollary~8.5]{JayeNazarovI}; the proof below also records
the uniform operator-bound issue for our convention.

\begin{lemma}\label{lem:weak-closure}
Suppose $\mu_k$ satisfy \eqref{eq:upper-growth} with one constant $A_0$,
are reflectionless, and obey
$\sup_{k,\varepsilon}\|\mathcal R^n_{\mu_k,\varepsilon}\|_{2\to2}\le M$.
If $\mu_k\rightharpoonup\mu$ vaguely, then $\mu$ satisfies the weak
reflectionless identity and has uniformly bounded Riesz truncations,
with bound controlled by $M+A_0$.  The zero measure is allowed as a limit.
\end{lemma}
\begin{proof}

Let $B=B(U,r)$ and $\delta>0$. Choose a continuous compact cutoff
$0\le\zeta_\delta\le1$ equal to one on $\overline B$ and supported in
$B(U,r+\delta)$. Then
\[
 \mu(B)\le\int\zeta_\delta\,d\mu
 =\lim_k\int\zeta_\delta\,d\mu_k\le A_0(r+\delta)^n.
\]
Letting $\delta\downarrow0$ proves upper growth for $\mu$.
Assume first $\mu\ne0$. Fix $f\in\Lip_c$ with $\int f\,d\mu=0$
and choose $\psi\in\Lip_c$ such that $\int\psi\,d\mu\ne0$.
For large $k$ put
\[
 c_k=\frac{\int f\,d\mu_k}{\int\psi\,d\mu_k},\qquad
 f_k=f-c_k\psi.
\]
Vague convergence implies $c_k\to0$, and hence
\[
 \int f_k\,d\mu_k=0,\quad
 \|f_k-f\|_\infty+\Lip(f_k-f)\to0.
\]
All supports belong to one ball $B_H$. Their first weighted moments
satisfy
\[
 \sup_k\int|U||f_k(U)|\,d\mu_k(U)
 \le H\sup_k\|f_k\|_\infty A_0H^n<\infty.
\]

Fix $S>2H$ and a smooth cutoff $\chi_S$ equal to one on $B_S$.
For $0<\epsilon<1$ the near-diagonal contribution, in a fixed ball
$B_K$ containing both test supports, obeys
\begin{align}
 &\iint_{\substack{U,V\in B_K\\0<|U-V|<2\epsilon}}
 |K_n(U-V)H_{f_k,\chi_S}(U,V)|\,d\mu_k(U)d\mu_k(V)\notag\\
 &\qquad\le C_S\mu_k(B_K)
       \sup_U\int_{|U-V|<2\epsilon}|U-V|^{1-n}\,d\mu_k(V)
 \le C_S\epsilon.
 \label{eq:closure-diagonal-error}
\end{align}
The constant is uniform in $k$; the same bound holds for $\mu$.
Multiply the kernel by a smooth cutoff which vanishes for
$|U-V|\le\epsilon$ and equals one for $|U-V|\ge2\epsilon$.
At fixed $S,\epsilon$ the resulting integrand is continuous and has
compact support. Product measures converge on compact sets: to see this,
approximate a continuous function of $(U,V)$ uniformly by finite sums
$\sum_i p_i(U)q_i(V)$ on the relevant compact product, and use the
uniform mass bounds together with vague convergence on each factor.
The replacement of $f_k$ by $f$ also has vanishing error because
$f_k\to f$ uniformly and the smoothed kernel is bounded. Thus, first
letting $k\to\infty$ and then $\epsilon\downarrow0$, we obtain
\[
 \langle\mathcal R^n(f_k\mu_k),\chi_S\rangle_{\mu_k}
 \longrightarrow\langle\mathcal R^n(f\mu),\chi_S\rangle_\mu.
\]

By \eqref{eq:balanced-tail-quantitative} and the uniform first-moment
bound from Step 1,
\begin{equation}\label{eq:closure-uniform-tail}
 |\langle\mathcal R^n(f_k\mu_k),1-\chi_S\rangle_{\mu_k}|
 +|\langle\mathcal R^n(f\mu),1-\chi_S\rangle_\mu|
 \le C/S.
\end{equation}
For each $k$ the full balanced pairing is zero. Step 2 and
\eqref{eq:closure-uniform-tail} imply
$|\langle\mathcal R^n(f\mu),1\rangle_\mu|\le C/S$.
Let $S\to\infty$ to obtain reflectionlessness. The precise order is:
fix $S,\epsilon$, pass to the weak measure limit, remove the inner
smoothing, and finally send $S$ to infinity. The weak identity is
automatic if $\mu=0$.

Let $\vartheta\in C^\infty([0,\infty))$ be nondecreasing, zero on
$[0,1]$ and one on $[2,\infty)$. Define
\[
 K^{\rm sm}_\epsilon(Z)=\vartheta(|Z|/\epsilon)K_n(Z).
\]
Since $\int_1^2\vartheta'(t)\,dt=1$, for $Z\ne0$,
\[
 K^{\rm sm}_\epsilon(Z)
 =\int_1^2\mathbf1_{\{|Z|>t\epsilon\}}K_n(Z)\vartheta'(t)\,dt.
\]
The associated operators for $\mu_k$ have norm at most $M$, by
Minkowski's inequality and the uniform hard-truncation bounds. For scalar $f\in C_c(\R^d)$ and
$g\in C_c(\R^d;\R^d)$,
\begin{align*}
 &\left|\iint f(V)K^{\rm sm}_\epsilon(U-V)\cdot g(U)
                    \,d\mu_k(U)d\mu_k(V)\right|\\
 &\hspace{20mm}\le M\|f\|_{L^2(\mu_k)}
                         \|g\|_{L^2(\mu_k;\R^d)}.
\end{align*}
The left side has a continuous compactly supported integrand, and
both squared norms converge by vague convergence. Passing to the limit
and then using density defines an $L^2(\mu)$ operator of norm at most
$M$ for this fixed smoothing.

The difference between the smoothed kernel and the hard cutoff at
$\epsilon$ is bounded by
\[
 C\epsilon^{-n}\mathbf1_{\{\epsilon<|U-V|<2\epsilon\}}.
\]
The integrals of this positive majorant in either variable are at most
$C A_0$. Schur's test therefore bounds the difference operator by
$C A_0$, uniformly in $\epsilon$. The hard truncations for $\mu$ have
norm at most $M+C A_0$, as claimed. Exact preservation of the numerical
constant $M$ is not needed for the subsequent argument.
\end{proof}

\subsection{Approximation on annuli}
For $j\ge1$ put
\[
 A_j=B(0,2^{j+2})\setminus\overline B(0,2^{j-1}).
\]
For a flat measure $\sigma=c\mathcal H^n|_L$ define
\begin{equation}\label{eq:annular-distance}
 d_j(\mu,\sigma)=2^{-jn}
 \sup_{\substack{f\in\Lip_c(A_j)\\
       \|f\|_\infty\le1,\ 2^j\Lip(f)\le1}}
 \left|\int f\,d(\mu-\sigma)\right|.
\end{equation}
For tests supported in $B_2$, set
\begin{equation}\label{eq:core-distance}
 D_2(\mu,\sigma)=
 \sup_{\substack{f\in\Lip_c(B_2)\\
       \|f\|_\infty\le1,\ \Lip(f)\le1}}
 \left|\int f\,d(\mu-\sigma)\right|.
\end{equation}

\begin{proposition}
\label{thm:annular}
Fix $n,d,A_0,M$, $0<c_*\le c^*<\infty$, and $D_*<\infty$.
For each $\tau>0$ there exist $N\in\mathbb N$ and $\epsilon>0$,
depending only on these parameters and $\tau$, with the following
property.  Suppose $\mu$ is reflectionless, satisfies
\eqref{eq:upper-growth}, and has Riesz $L^2$ bound at most $M$.
Suppose one flat measure
\begin{equation}\label{eq:controlled-flat}
 \sigma=c\mathcal H^n|_L,\qquad
 c_*\le c\le c^*,\qquad\dist(0,L)\le D_*,
\end{equation}
satisfies $d_j(\mu,\sigma)<\epsilon$ for $1\le j\le N$.
Then $D_2(\mu,\sigma)<\tau$.
No AD lower bound is needed for this measure-distance assertion.
\end{proposition}
\begin{proof}
Fix $\tau>0$ and suppose the assertion fails. For each integer $k$ take
$N=k$ and $\epsilon=1/k$ in its negation. There are then measures
$\mu_k$ and $\sigma_k=c_k\mathcal H^n|_{L_k}$ satisfying all the fixed
parameter bounds such that
\begin{equation}\label{eq:annular-contradiction-sequence}
 d_j(\mu_k,\sigma_k)<1/k\quad(1\le j\le k),\qquad
 D_2(\mu_k,\sigma_k)\ge\tau.
\end{equation}
Upper growth bounds $\mu_k(B_H)$ for each fixed $H$, so a diagonal
weak-compactness argument gives a vaguely convergent subsequence
$\mu_k\rightharpoonup\mu_\infty$.

Write $L_k=b_k+P_k$, where $b_k\perp P_k$ is the nearest point of
$L_k$ to zero. The orthogonal projections onto the $n$-planes $P_k$
form a compact set of matrices, $|b_k|\le D_*$, and
$c_*\le c_k\le c^*$. After a further subsequence,
\[
 P_k\to P_\infty,\qquad b_k\to b_\infty\in P_\infty^\perp,
 \qquad c_k\to c_\infty\in[c_*,c^*].
\]
One may choose orthonormal frames $E_k:\R^n\to P_k$ converging to a
frame $E_\infty$ after another subsequence. For $f\in C_c(\R^d)$,
\[
 \int f\,d\sigma_k=c_k\int_{\R^n}f(b_k+E_kx)\,dx
 \longrightarrow c_\infty\int_{\R^n}f(b_\infty+E_\infty x)\,dx.
\]
Dominated convergence is applicable because $b_k$ are bounded and the
frames are isometries, so all nonzero integrands lie in one bounded
parameter ball. Thus $\sigma_k\rightharpoonup\sigma_\infty$, where
$\sigma_\infty=c_\infty\mathcal H^n|_{L_\infty}$ and
$L_\infty=b_\infty+P_\infty$.
Lemma~\ref{lem:weak-closure} makes $\mu_\infty$ reflectionless with the
required upper growth and a finite uniform operator bound.

For fixed $j$ and $f\in\Lip_c(A_j)$ set
$B_f=\max\{\|f\|_\infty,2^j\Lip(f)\}$. If $B_f>0$, then $f/B_f$
is admissible in the definition of $d_j$, so for $k\ge j$,
\[
 \left|\int f\,d(\mu_k-\sigma_k)\right|
 \le B_f2^{jn}d_j(\mu_k,\sigma_k)
 \le B_f2^{jn}/k\longrightarrow0.
\]
The equality is trivial if $B_f=0$. Passing to the weak limits and
approximating continuous compact tests by Lipschitz tests shows
$\mu_\infty|_{A_j}=\sigma_\infty|_{A_j}$. The open annuli cover
$\R^d\setminus\overline B_1$. A partition of unity on the compact
support of an exterior test therefore gives
\begin{equation}\label{eq:annular-limit-exterior}
 \mu_\infty=\sigma_\infty\quad\text{on }\R^d\setminus\overline B_1.
\end{equation}
Because $c_\infty\ge c_*>0$, this limit is nonzero. Apply
Theorem~\ref{thm:compact-flat} to conclude
$\mu_\infty=\sigma_\infty$ on all of $\R^d$.

It remains to turn convergence on each fixed test into the supremum
in $D_2$. Extend every test in \eqref{eq:core-distance} by zero outside $B_2$.
The resulting functions are uniformly bounded, uniformly Lipschitz,
and supported in $\overline B_2$. By Arzel\`a--Ascoli this family is
totally bounded in the uniform norm. For any $\delta>0$, choose a
finite uniform $\delta$-net $f_1,\ldots,f_J$ from it. Then
\begin{align}
 D_2(\mu_k,\sigma_k)
 &\le\max_{1\le i\le J}\left|\int f_i\,d(\mu_k-\sigma_k)\right|
 +\delta\bigl(\mu_k(B_2)+\sigma_k(B_2)\bigr)\notag\\
 &\le\max_{1\le i\le J}\left|\int f_i\,d(\mu_k-\sigma_k)\right|
 +C\delta.
 \label{eq:core-finite-net-detail}
\end{align}
The finite maximum tends to zero because the two limiting measures
are equal. Taking a limsup and then $\delta\downarrow0$ contradicts
\eqref{eq:annular-contradiction-sequence}. This proves the existence of
$N$ and $\epsilon$ with the stated dependencies; it does not provide
an explicit modulus.
\end{proof}

\begin{corollary}\label{cor:annular-support}
Let $\mathcal A_n(\R^d)$ denote the set of affine $n$-planes in $\R^d$.
Under the fixed parameters in Proposition~\ref{thm:annular}, small
$D_2(\mu,\sigma)$ gives
\begin{equation}\label{eq:beta-core}
 \beta_{2,\mu}(B_1)^2
 :=\inf_{L'\in\mathcal A_n(\R^d)}
     \int_{B_1}\dist(U,L')^2\,d\mu(U)
 \le C D_2(\mu,\sigma).
\end{equation}
If in addition $\mu(B(U,r))\ge a_*r^n$ for
$U\in\supp\mu$ and $0<r\le1$, then, when $D_2$ is sufficiently small,
\begin{align}
 &\sup_{U\in\supp\mu\cap B_1}\dist(U,L)
 +\sup_{V\in L\cap B_1}\dist(V,\supp\mu)
 \le C D_2(\mu,\sigma)^{1/(n+1)}.
       \label{eq:bilateral-core}
\end{align}
The additional constants depend on $a_*$ and $c_*$.  In
\eqref{eq:bilateral-core}, a supremum over an empty set is taken to be zero.
\end{corollary}
\begin{proof}
Let $0\le\chi\le1$ be smooth, supported in $B_2$, and equal to one on
$B_1$.  The function $\chi(U)\dist(U,L)^2$ has uniformly bounded
supremum and Lipschitz norm, depending on $D_*$, and its integral
against $\sigma$ is zero.  Normalize it in
\eqref{eq:core-distance} to obtain \eqref{eq:beta-core}.

If $U\in\supp\mu\cap B_1$ has $t=\dist(U,L)>0$, let
$r=\min\{t/4,1/4\}$ and test with $f(Z)=(r-|Z-U|)_+$.
The test is admissible, vanishes on $L$, and
\[
 \int f\,d\mu\ge (r/2)\mu(B(U,r/2))\ge c a_*r^{n+1}.
\]
For sufficiently small $D_2$, the choice $r=1/4$ is excluded, and
$t\le C D_2^{1/(n+1)}$.  For a point $V\in L\cap B_1$ at distance $t$
from $\supp\mu$, use the same tent function centered at $V$.
Now its $\mu$ integral is zero and its $\sigma$ integral is at least
$cc_*r^{n+1}$.  This proves the other direction.
\end{proof}

Proposition~\ref{thm:annular} is a compactness consequence of exterior-flat
rigidity and the known weak closure mechanism.  The comparison plane
and its density must be the same on all tested annuli.  The result does
not provide an explicit modulus, a geometric decay factor between
successive scales, or an elimination of Carleson-packed exceptional
cubes.  None of those statements is used or claimed here.

\subsection{The orthogonal field and half-order estimates}\label{sec:vertical-energy}

Keep $\operatorname{Lip}(\varphi)\le M$ and $0<c_0\le a\le C_0$,
without any density normalization at infinity.  For vector-valued
functions set
\[
 \mathscr B_{\varphi,a}(u,v)
 =\frac12\iint (u(x)-u(y))\cdot(v(x)-v(y))
 w_{\varphi,a}(x,y)\,dx\,dy,
 \qquad \mathscr E_{\varphi,a}(u)=\mathscr B_{\varphi,a}(u,u).
\]
From \eqref{eq:graph-biLip},
\begin{equation}\label{eq:weight-comparison}
 \frac{c_0^2}{(1+M^2)^{(n+1)/2}}|x-y|^{-n-1}
 \le w_{\varphi,a}(x,y)\le C_0^2|x-y|^{-n-1}.
\end{equation}
In particular,
\begin{equation}\label{eq:energy-comparison}
 c[u]_{1/2}^2\le\mathscr E_{\varphi,a}(u)\le C[u]_{1/2}^2,
 \qquad
 |\mathscr B_{\varphi,a}(u,v)|\le C[u]_{1/2}[v]_{1/2}.
\end{equation}

\begin{equation}\label{eq:dual-half-norm}
 \|\mathcal F\|_{\dot H^{-1/2}}
 =\sup\left\{ |\langle\mathcal F,v\rangle|:
 v\in C_c^\infty(\R^n;\R^m),\ [v]_{1/2}\le1\right\}.
\end{equation}

\begin{proposition}\label{thm:half-stability}
Suppose $\varphi$ is globally Lipschitz, $[\varphi]_{1/2}<\infty$,
$\Lip(\varphi)\le M$, and $0<c_0\le a\le C_0$. Then
\[
 c[\varphi]_{1/2}\le
 \|\mathcal F_{\varphi,a}\|_{\dot H^{-1/2}}
 \le C[\varphi]_{1/2},
\]
where the positive constants depend only on $n,m,M,c_0,C_0$.
There is no compact-support, density-tail, or H\"older continuity
assumption. The weak field is defined by \eqref{eq:weak-normal-field}.
\end{proposition}
\begin{proof}
Fix the graph and density in the definition of $\mathscr B_{\varphi,a}$. In particular,
this bilinear form is linear in its two test arguments, even though its
kernel depends nonlinearly on the fixed graph. From
\eqref{eq:weight-comparison},
\begin{align}
 |\langle\mathcal F_{\varphi,a},v\rangle|
 &\le \frac{C_0^2}{2}\iint
 \frac{|\varphi(x)-\varphi(y)|\,|v(x)-v(y)|}{|x-y|^{n+1}}\,dx\,dy\notag\\
 &\le \frac{C_0^2}{2}[\varphi]_{1/2}[v]_{1/2}.
 \label{eq:dual-upper-detail}
\end{align}
Thus the weak field is a continuous functional for the homogeneous
seminorm, and the upper estimate follows from the definition of its
norm. This definition applies to bounded measurable densities.

 Set
\[
 c_R=\fint_{B_{4R}}\varphi(x)\,dx,\qquad
 v_R=\chi_R(\varphi-c_R).
\]
Apply the Lebesgue version of Lemma~\ref{lem:energy-cutoff} to each
coordinate and then use \eqref{eq:weight-comparison}. This gives
\begin{equation}\label{eq:dual-test-uniform}
 [v_R]_{1/2}\le C[\varphi]_{1/2},\qquad
 \mathscr B_{\varphi,a}(\varphi,v_R)
 \longrightarrow\mathscr E_{\varphi,a}(\varphi).
\end{equation}
For the second assertion, one may repeat the tail argument of that
lemma with $w_{\varphi,a}$: the differences of $\varphi-v_R$ vanish
on $B_R\times B_R$, the weighted energy of $v_R$ is uniformly bounded,
and the weighted energy of $\varphi$ outside $B_R\times B_R$ tends
to zero. This does not assert strong convergence of $v_R$ in the
homogeneous space, and does not require $c_R$ to converge.

For each
fixed $R$, $v_R$ is compactly supported and Lipschitz. Let $\zeta$ be
a nonnegative smooth mollifier of integral one and put
$v_{R,\epsilon}=\zeta_\epsilon*v_R$. Then
$v_{R,\epsilon}\in C_c^\infty$. To verify the required seminorm
convergence, define
\[
 \mathcal D v(x,y)=\frac{v(x)-v(y)}{|x-y|^{(n+1)/2}}.
\]
Since $\mathcal Dv_R\in L^2(\R^{2n})$, continuity of simultaneous
translations in this $L^2$ space gives
\[
 [v_R(\cdot-h)-v_R]_{1/2}
 =\|\mathcal Dv_R(\cdot-h,\cdot-h)-\mathcal Dv_R\|_2\to0
 \quad(h\to0).
\]
Minkowski's inequality, followed by this translation continuity, yields
\begin{align}
 [v_{R,\epsilon}-v_R]_{1/2}
 &\le\int\zeta_\epsilon(h)
       [v_R(\cdot-h)-v_R]_{1/2}\,dh\longrightarrow0.
 \label{eq:mollifier-energy-detail}
\end{align}
The same construction converges in $L^2$, although that fact is not
needed in the homogeneous argument. By \eqref{eq:dual-upper-detail}
and the definition of the dual norm,
\[
 |\mathscr B_{\varphi,a}(\varphi,v_R)|
 \le\|\mathcal F_{\varphi,a}\|_{\dot H^{-1/2}}[v_R]_{1/2}.
\]
Here the inequality is first applied to $v_{R,\epsilon}$ and then
$\epsilon\downarrow0$ is taken with $R$ fixed.

Finally, \eqref{eq:energy-comparison} and
\eqref{eq:dual-test-uniform} imply
\[
 c[\varphi]_{1/2}^2
 \le\mathscr E_{\varphi,a}(\varphi)
 =\lim_{R\to\infty}\mathscr B_{\varphi,a}(\varphi,v_R)
 \le C\|\mathcal F_{\varphi,a}\|_{\dot H^{-1/2}}[\varphi]_{1/2}.
\]
Divide when $[\varphi]_{1/2}>0$. If it vanishes, the upper bound gives
$\mathcal F_{\varphi,a}=0$ as a functional. This also covers the critical
case $n=1$ without choosing an $L^p$ representative of a homogeneous
Sobolev class.
\end{proof}

\begin{proposition}
\label{prop:conditional-half}
Under the hypotheses of Proposition~\ref{thm:half-stability}, assume also
$\varphi\in H^{1/2}$ and that the distribution
$\mathcal F_{\varphi,a}$ is represented by $a\mathcal N$ with
$\mathcal N\in L^2(a\,dx)$.  Then
\begin{equation}\label{eq:energy-identity}
 \int\varphi\cdot\mathcal N\,a\,dx
 =\mathscr E_{\varphi,a}(\varphi),
\end{equation}
and
\begin{equation}\label{eq:conditional-half}
 [\varphi]_{1/2}^2\le
 C\|\varphi\|_{L^2(a\,dx)}\|\mathcal N\|_{L^2(a\,dx)}.
\end{equation}
In particular, if $\|\varphi\|_{L^2(a\,dx)}\le A$, then
$[\varphi]_{1/2}\le C A^{1/2}\|\mathcal N\|_{L^2(a\,dx)}^{1/2}$.
\end{proposition}
\begin{proof}
The representation hypothesis means precisely that
\[
 \mathscr B_{\varphi,a}(\varphi,v)
 =\int v(x)\cdot\mathcal N(x)a(x)\,dx
 \qquad(v\in C_c^\infty).
\]
Both sides extend continuously to $H^{1/2}$: the left side by
\eqref{eq:energy-comparison}, and the right side because
\[
 \left|\int v\cdot\mathcal N a\right|
 \le\|v\|_{L^2(a\,dx)}\|\mathcal N\|_{L^2(a\,dx)}
 \le C_0^{1/2}\|v\|_2\|\mathcal N\|_{L^2(a\,dx)}.
\]
Choose $\varphi_j\in C_c^\infty$ converging to $\varphi$ in the
inhomogeneous space $H^{1/2}$; see \cite{DiNezzaPalatucciValdinoci}.
Then
\begin{align*}
 |\mathscr B_{\varphi,a}(\varphi,\varphi_j)
       -\mathscr E_{\varphi,a}(\varphi)|
 &\le C[\varphi]_{1/2}[\varphi_j-\varphi]_{1/2}\to0,\\
 \left|\int(\varphi_j-\varphi)\cdot\mathcal N a\right|
 &\le C_0^{1/2}\|\varphi_j-\varphi\|_2
                  \|\mathcal N\|_{L^2(a\,dx)}\to0.
\end{align*}
Passing to the limit in the test identity proves
\eqref{eq:energy-identity}. Consequently,
\[
 c[\varphi]_{1/2}^2
 \le\mathscr E_{\varphi,a}(\varphi)
 \le\|\varphi\|_{L^2(a\,dx)}\|\mathcal N\|_{L^2(a\,dx)}.
\]
This is \eqref{eq:conditional-half}. Under the additional bound
$\|\varphi\|_{L^2(a\,dx)}\le A$, take square roots to obtain the stated
conditional estimate. The representation and $L^2$ membership of
$\mathcal N$ are assumptions here: neither is inferred merely from
finite half-order energy.
\end{proof}

For $\varepsilon>0$ write
\begin{equation}\label{eq:normal-truncated}
 \mathcal N^\varepsilon_{\varphi,a}(x)=
 \int_{|X_\varphi(x)-X_\varphi(y)|>\varepsilon}
 \frac{\varphi(x)-\varphi(y)}
 {|X_\varphi(x)-X_\varphi(y)|^{n+1}}a(y)\,dy.
\end{equation}

\begin{proposition}
\label{prop:normal-pv}
Let $\varphi\in C_c^{1,\alpha}(\R^n;\R^m)$, $0<\alpha<1$,
$\supp\varphi\subset B_R$, and $0<c_0\le a\le C_0$.
The truncations \eqref{eq:normal-truncated} converge in $L^2(a\,dx)$
to a field $\mathcal N_{\varphi,a}$.  They also have the usual
principal-value representative a.e.  Outside $B_{2R}$,
\begin{equation}\label{eq:normal-decay}
 |\mathcal N_{\varphi,a}(x)|
 \le C\|a\varphi\|_{L^1(B_R)}|x|^{-n-1}.
\end{equation}
No regularity or decay assumption on $a$ beyond these bounds is required.
\end{proposition}
\begin{proof}
We do not expand the density $a$ by Taylor's formula. It is only assumed
to be bounded. Instead, we first establish a linear $L^2$ realization
and then localize the nonintegrable function $a$.

Let
\[
 S_\epsilon f(x)=\int_{|X_\varphi(x)-X_\varphi(y)|>\epsilon}
 \frac{\varphi(x)-\varphi(y)}{|X_\varphi(x)-X_\varphi(y)|^{n+1}}
 f(y)\,dy.
\]
This is the orthogonal component of the graph Riesz operator with
respect to $\mu_{\varphi,1}$. Proposition~\ref{prop:graph-structural}
gives
\begin{equation}\label{eq:rough-linear-L2}
 \sup_{\epsilon>0}\|S_\epsilon f\|_2\le C_M\|f\|_2.
\end{equation}
For $f\in C_c^\infty$, Appendix~\ref{app:principal-values}, with the
smooth signed function $f$ in place of $a$, gives convergence locally
uniformly in $x$. For $|x|>2R$, the numerator vanishes unless
$y\in B_R$, and
\begin{equation}\label{eq:smooth-input-orthogonal-tail}
 |S_\epsilon f(x)|
 \le C|x|^{-n-1}\int_{B_R}|\varphi(y)f(y)|\,dy.
\end{equation}
This bound holds for every $\epsilon$ and is square integrable in the
exterior. Thus $S_\epsilon f$ has a strong $L^2$ limit for smooth compact
$f$.

For general $f\in L^2$, choose $f_j\in C_c^\infty$ with $f_j\to f$
in $L^2$. For $\epsilon,\epsilon'>0$,
\[
 \|S_\epsilon f-S_{\epsilon'}f\|_2
 \le2C_M\|f-f_j\|_2
       +\|S_\epsilon f_j-S_{\epsilon'}f_j\|_2.
\]
First make the first term small by fixing $j$, and then take
$\epsilon,\epsilon'\downarrow0$. Completeness defines a bounded
operator $S:L^2\to L^2$, with $S_\epsilon f\to Sf$ strongly.

We next justify the maximal estimate for these truncations. Let
$\mathcal M$ denote the Hardy--Littlewood maximal operator on $\R^n$,
and set
\[
 k(x,y)=\frac{\varphi(x)-\varphi(y)}
              {|X_\varphi(x)-X_\varphi(y)|^{n+1}},\qquad
 \widetilde S_\epsilon f(x)=\int_{|x-y|>\epsilon}k(x,y)f(y)\,dy.
\]
The Lipschitz bound and \eqref{eq:graph-biLip} imply
\begin{align}
 |k(x,y)|&\le C_M|x-y|^{-n},\notag\\
 |k(x,y)-k(x',y)|+|k(y,x)-k(y,x')|
   &\le C_M\frac{|x-x'|}{|x-y|^{n+1}},
       \quad |x-x'|\le\tfrac12|x-y|.
       \label{eq:parameter-kernel-standard}
\end{align}
For example, apply the mean value theorem to
$Z\mapsto\Pi^\perp Z|Z|^{-n-1}$. Along the segment joining the two
graph increments, the horizontal coordinates have modulus at least
$|x-y|/2$, and the difference of the increments is at most
$L_M|x-x'|$. These observations prove the second estimate.
The two truncations can differ only when
$\epsilon/L_M<|x-y|\le\epsilon$. Consequently,
\begin{equation}\label{eq:graph-parameter-maximal-comparison}
 |S_\epsilon f(x)-\widetilde S_\epsilon f(x)|
 \le C_M\epsilon^{-n}
        \int_{B(x,\epsilon)}|f(y)|\,dy
 \le C_M\mathcal Mf(x).
\end{equation}
For smooth compactly supported $f$, Appendix~\ref{app:principal-values}
identifies the principal values for the two truncations. Their strong
$L^2$ limits therefore agree by density. Thus $S$ is an $L^2$-bounded
Calder\'on--Zygmund operator with the standard kernel
\eqref{eq:parameter-kernel-standard}.

We recall the weak $(1,1)$ estimate used in Cotlar's inequality;
these are the standard Calder\'on--Zygmund arguments (see \cite{Stein}).
Apply the Calder\'on--Zygmund decomposition to
$f\in L^1\cap L^2$ at height $\lambda>0$:
\[
 f=g+\sum_Q b_Q,\qquad
 \|g\|_2^2\le C_n\lambda\|f\|_1,\quad
 \sum_Q|Q|\le\lambda^{-1}\|f\|_1,\quad
 \sum_Q\|b_Q\|_1\le C_n\|f\|_1.
\]
Here $b_Q$ is supported in $Q$ and has integral zero. Let $Q^*$ be a
fixed sufficiently large concentric dilation, and let $y_Q$ be the
center of $Q$. By \eqref{eq:parameter-kernel-standard},
\[
 \int_{(Q^*)^c}|Sb_Q(x)|\,dx
 \le\int_Q |b_Q(y)|
    \int_{(Q^*)^c}|k(x,y)-k(x,y_Q)|\,dx\,dy
 \le C_M\|b_Q\|_1.
\]
The $L^2$ estimate for $Sg$, the measure of $\bigcup_Q Q^*$, and
Chebyshev's inequality for the last display imply
\begin{equation}\label{eq:rough-weak-one}
 |\{x:|Sf(x)|>\lambda\}|\le C_M\lambda^{-1}\|f\|_1.
\end{equation}
Thus $S$ is of weak type $(1,1)$. We now prove the form of Cotlar's
inequality needed here. Fix $0<\gamma<1$. For compact smooth $f$,
\begin{equation}\label{eq:Cotlar-gamma}
 \widetilde S_*f(x)
 \le C_{M,\gamma}
 \left\{\bigl[\mathcal M(|Sf|^\gamma)(x)\bigr]^{1/\gamma}
                       +\mathcal Mf(x)\right\}.
\end{equation}
Indeed, for a fixed $\epsilon$ put
$B=B(x,\epsilon/4)$, $f_1=f\mathbf1_{B(x,2\epsilon)}$ and
$f_2=f-f_1$. Summing the kernel regularity estimate over the annuli
$|x-y|\asymp2^j\epsilon$, and estimating the remaining annulus by
size, gives
\[
 |Sf_2(z)-\widetilde S_\epsilon f(x)|
       \le C_M\mathcal Mf(x),\qquad z\in B.
\]
Since $Sf_2=Sf-Sf_1$ almost everywhere, averaging the $\gamma$th powers
on $B$ reduces \eqref{eq:Cotlar-gamma} to
\[
 \left(\fint_B|Sf_1|^\gamma\right)^{1/\gamma}
 \le C_\gamma |B|^{-1}\|Sf_1\|_{L^{1,\infty}}
 \le C_{M,\gamma}|B|^{-1}\|f_1\|_1
 \le C_{M,\gamma}\mathcal Mf(x).
\]
The first inequality follows by integrating the distribution function
of $|Sf_1|$; the second is the weak $(1,1)$ estimate. The argument
applies to each component of $S$ and hence to its Euclidean norm.
The boundedness of $\mathcal M$ on $L^{2/\gamma}$ and on $L^2$, together
with \eqref{eq:graph-parameter-maximal-comparison}, now yields
\begin{equation}\label{eq:rough-maximal-L2}
 \|S_*f\|_2\le C_M\|f\|_2,
 \qquad S_*f=\sup_{\epsilon>0}|S_\epsilon f|.
\end{equation}
Approximation extends this inequality to every $f\in L^2$.
The oscillation of $S_\epsilon f$ as $\epsilon\downarrow0$ is bounded
by $2S_*(f-f_j)$, since that of $f_j$ is zero. Chebyshev's inequality
therefore shows, for every $\lambda>0$,
\[
 |\{x:\limsup_{\epsilon,\epsilon'\downarrow0}
           |S_\epsilon f(x)-S_{\epsilon'}f(x)|>\lambda\}|
 \le C_M\lambda^{-2}\|f-f_j\|_2^2\to0.
\]
Thus the ordinary principal value exists almost everywhere. It agrees
with $Sf$, because a subsequence of the strong $L^2$ convergence also
converges almost everywhere.

 Choose
$0\le\zeta\le1$ smooth, supported in $B_{5R}$ and equal to one on
$B_{4R}$. For $x\in B_{2R}$ and $0<\epsilon<R$,
\begin{equation}\label{eq:rough-density-splitting}
 \mathcal N^\epsilon_{\varphi,a}(x)=S_\epsilon(a\zeta)(x)+H(x),\qquad
 H(x)=\varphi(x)\int
 \frac{(1-\zeta(y))a(y)}{|X_\varphi(x)-X_\varphi(y)|^{n+1}}\,dy.
\end{equation}
Indeed, on the support of $1-\zeta$ one has $\varphi(y)=0$ and
$|x-y|\ge2R>\epsilon$, so this portion of the integral is independent
of $\epsilon$. Moreover,
\[
 |H(x)|\le C C_0\|\varphi\|_\infty
          \int_{|y|>4R}|y|^{-n-1}\,dy
 \le C C_0R^{-1}\|\varphi\|_\infty.
\]
Since $a\zeta\in L^2$, Step 1 gives $L^2(B_{2R})$ and almost-everywhere
convergence of the first term. No regularity of $a$ is needed.

 If $|x|>2R$ and $y\in B_R$,
then $|X_\varphi(x)-X_\varphi(y)|\ge|x-y|\ge|x|/2>R$.
For $\epsilon<R$, no contributing point in this region is truncated.
Consequently,
\[
 \mathcal N^\epsilon_{\varphi,a}(x)
 =-\int_{B_R}
 \frac{\varphi(y)a(y)}{(|x-y|^2+|\varphi(y)|^2)^{(n+1)/2}}\,dy,
\]
and its absolute value is bounded by the right side of
\eqref{eq:normal-decay}. The exterior formula is independent of all
sufficiently small $\epsilon$ and belongs to $L^2$. Combining the
interior and exterior statements proves convergence in $L^2(dx)$.
Finally,
\[
 c_0\|u\|_2^2\le\|u\|_{L^2(a\,dx)}^2\le C_0\|u\|_2^2
\]
transfers both the strong convergence and the almost-everywhere
representative to the weighted measure. In particular, the rough density
has not been inserted into the H\"older expansion in Appendix A.
\end{proof}

\begin{proposition}\label{prop:energy-identity}
Under the hypotheses of Proposition~\ref{prop:normal-pv},
\begin{equation}\label{eq:compact-energy-identity}
 \int\varphi\cdot\mathcal N_{\varphi,a}\,a\,dx
 =\frac12\iint|\varphi(x)-\varphi(y)|^2
             w_{\varphi,a}(x,y)\,dx\,dy.
\end{equation}
Moreover $\mathcal F_{\varphi,a}=a\mathcal N_{\varphi,a}$ as distributions.
\end{proposition}
\begin{proof}
For $\epsilon>0$ let
\[
 D_\epsilon=\{(x,y):|X_\varphi(x)-X_\varphi(y)|>\epsilon\},\qquad
 I_\epsilon=\int\varphi(x)\cdot
                  \mathcal N^\epsilon_{\varphi,a}(x)a(x)\,dx.
\]
Absolute convergence before symmetrization must be checked. The variable
$x$ is restricted to $B_R$. On a bounded $y$ region, the truncation
bounds the singular kernel, and both $a$ and $\varphi$ are bounded.
For $|y|>2R$ the second graph value is zero and
\[
 \left|\frac{\varphi(x)\cdot(\varphi(x)-\varphi(y))}
 {|X_\varphi(x)-X_\varphi(y)|^{n+1}}a(x)a(y)\right|
 \le C|\varphi(x)|^2|y|^{-n-1}.
\]
This is integrable. Fubini and interchange of the two variables are
therefore legitimate at fixed $\epsilon$.

Since $w_{\varphi,a}(y,x)=w_{\varphi,a}(x,y)$ and $D_\epsilon$ is
unchanged by $(x,y)\mapsto(y,x)$,
\begin{align*}
 I_\epsilon
 &=\iint_{D_\epsilon}
   \varphi(x)\cdot(\varphi(x)-\varphi(y))w_{\varphi,a}(x,y)\,dx\,dy,\\
 I_\epsilon
 &=\iint_{D_\epsilon}
   \varphi(y)\cdot(\varphi(y)-\varphi(x))w_{\varphi,a}(x,y)\,dx\,dy.
\end{align*}
Add these equalities and divide by two. The numerator becomes
\[
 |\varphi(x)|^2+|\varphi(y)|^2-2\varphi(x)\cdot\varphi(y)
 =|\varphi(x)-\varphi(y)|^2.
\]
Hence
\begin{equation}\label{eq:truncated-energy-identity}
 I_\epsilon
 =\frac12\iint_{D_\epsilon}|\varphi(x)-\varphi(y)|^2
                         w_{\varphi,a}(x,y)\,dx\,dy
 =:\mathscr E_\epsilon.
\end{equation}
This is the truncated energy, not the untruncated energy for fixed
$\epsilon$.

As $\epsilon\downarrow0$, the domains increase to $\{x\ne y\}$.
The integrand is nonnegative and its full integral is finite by the
compact Lipschitz energy estimate. Monotone convergence gives
\[
 \mathscr E_\epsilon\uparrow\mathscr E_{\varphi,a}(\varphi).
\]
Independently, Proposition~\ref{prop:normal-pv} and weighted
Cauchy--Schwarz give
\begin{align}
 \left|I_\epsilon-\int\varphi\cdot\mathcal N_{\varphi,a}a\,dx\right|
 &\le\int|\varphi|\,|\mathcal N^\epsilon_{\varphi,a}
                       -\mathcal N_{\varphi,a}|a\,dx\notag\\
 &\le\|\varphi\|_{L^2(a\,dx)}
        \|\mathcal N^\epsilon_{\varphi,a}-\mathcal N_{\varphi,a}
          \|_{L^2(a\,dx)}\longrightarrow0.
 \label{eq:energy-left-limit-detail}
\end{align}
Identifying these two limits of the same number $I_\epsilon$ proves
\eqref{eq:compact-energy-identity}.

For the distributional assertion, take $v\in C_c^\infty(\R^n;\R^m)$.
At fixed $\epsilon$ the same Fubini argument yields
\[
 \int v\cdot\mathcal N^\epsilon_{\varphi,a}a\,dx
 =\frac12\iint_{D_\epsilon}
 (v(x)-v(y))\cdot(\varphi(x)-\varphi(y))w_{\varphi,a}(x,y)\,dx\,dy.
\]
The right side is dominated by an absolutely integrable mixed-energy
function, by Cauchy--Schwarz and \eqref{eq:energy-comparison}.
Dominated convergence applies to that side, whereas the left side
converges by the weighted $L^2$ convergence already proved. The result is
\[
 \int v\cdot\mathcal N_{\varphi,a}a\,dx
 =\mathscr B_{\varphi,a}(\varphi,v)
 =\langle\mathcal F_{\varphi,a},v\rangle.
\]
This proves $\mathcal F_{\varphi,a}=a\mathcal N_{\varphi,a}$ in the
stated concrete representation.
\end{proof}

\begin{corollary}
\label{cor:compact-half}
If $\varphi\in C_c^{1,\alpha}$, $\supp\varphi\subset B_R$,
$\|D\varphi\|_\infty\le M$, and $0<c_0\le a\le C_0$, then
\begin{equation}\label{eq:compact-half-stability}
 [\varphi]_{1/2}\le
 C(R,M,c_0,C_0,n,m)\|\mathcal N_{\varphi,a}\|_{L^2(a\,dx)}.
\end{equation}
\end{corollary}
\begin{proof}
Combining \eqref{eq:energy-comparison},
\eqref{eq:compact-energy-identity}, and
Lemma~\ref{lem:fractional-Poincare} gives
\[
 c[\varphi]_{1/2}^2
 \le \|\varphi\|_{L^2(a)}\|\mathcal N_{\varphi,a}\|_{L^2(a)}
 \le C R^{1/2}[\varphi]_{1/2}\|\mathcal N_{\varphi,a}\|_{L^2(a)}.
\]
Divide unless $[\varphi]_{1/2}=0$, in which case the result is immediate.
The constants in the lower estimate do not involve H\"older norms,
although H\"older regularity of the graph was used for this realization.
\end{proof}

\section{The flat linearization and Multilinear estimates}\label{sec:linearization}
This section develops the linear and multilinear estimates needed to study the normalized Riesz field near a flat measure. We identify the flat linearization through the Riesz transform and \(|D|\), and use their Fourier multipliers to control the graph gradient and density perturbation. One-dimensional Calder\'on commutator estimates, the method of rotations, and interpolation then yield Sobolev bounds that are uniform in the truncation parameter. These estimates justify the operator limits and provide the tools for controlling the nonlinear terms in Section \ref{sec:nonlinear-remainder}.
\subsection{The two model operators}
With the normalization in Section~\ref{sec:reflectionless-framework},
\begin{align}
 \pv\int\frac{x-y}{|x-y|^{n+1}}\sigma(y)\,dy
 &=c_{n,R}\mathcal R\sigma(x),\label{eq:flat-Riesz}\\
 \pv\int\frac{\psi(x)-\psi(y)}{|x-y|^{n+1}}\,dy
 &=c_{n,\Lambda}|D|\psi(x).\label{eq:flat-Lambda}
\end{align}
The first identity is understood componentwise and the second for each
graph component.  Both constants are nonzero and depend only on $n$.
For clarity, the order of the second operator can be read directly from
its symmetric kernel. For a Schwartz scalar function, centered inner
and outer cutoffs, the Fourier multiplier is
\[
 m_{\epsilon,R}(\xi)
 =\int_{\epsilon<|z|<R}
        \frac{1-e^{-iz\cdot\xi}}{|z|^{n+1}}\,dz
 =\int_{\epsilon<|z|<R}
        \frac{1-\cos(z\cdot\xi)}{|z|^{n+1}}\,dz.
\]
The imaginary part vanishes by oddness. The bounds
$0\le1-\cos t\le\min\{t^2/2,2\}$ show integrability at both ends:
near zero the radial integrand is bounded, and at infinity it is
$O(r^{-2})$. Rotation and the substitution $w=|\xi|z$ give, for
$\xi\ne0$,
\begin{equation}\label{eq:half-laplacian-symbol-detail}
 \lim_{\epsilon\downarrow0,\ R\uparrow\infty}m_{\epsilon,R}(\xi)
 =|\xi|\int_{\R^n}\frac{1-\cos w_1}{|w|^{n+1}}\,dw
 =c_{n,\Lambda}|\xi|,
 \qquad 0<c_{n,\Lambda}<\infty.
\end{equation}
The vector Riesz multiplier is the normalized multiplier fixed in
Section~2. Thus the orthogonal linear operator has order one, whereas
the tangential density operator has order zero.

\begin{proposition}\label{prop:linearization}
On $H^{s+1}(\R^n;\R^m)\times H^s(\R^n)$, $s>n/2$ an integer,
the differential of $\mathbf T$ at zero is
\begin{equation}\label{eq:linearization}
 \mathcal L(\psi,\sigma):=D\mathbf T(0,0)[\psi,\sigma]
 =(c_{n,R}\mathcal R\sigma,c_{n,\Lambda}|D|\psi).
\end{equation}
\end{proposition}
\begin{proof}
We first compute a candidate differential with smooth compact functions at
fixed inner and outer truncations. Put $p=(n+1)/2$. Direct differentiation
gives
\[
 DA(q)[v]=-2p(1+|q|^2)^{-p-1}q\cdot v,\qquad
 DB(q)[v]=vA(q)+qDA(q)[v].
\]
In particular,
\[
 A(0)=1,\quad DA(0)=0,\quad B(0)=0,\quad DB(0)=I.
\]
For the tangential factor, the product rule yields
\begin{align*}
 \frac d{dt}\bigl[A(tQ_\psi)(1+t\sigma(y))-1\bigr]
 &=DA(tQ_\psi)[Q_\psi](1+t\sigma(y))
   +A(tQ_\psi)\sigma(y),\\
 \left.\frac d{dt}\right|_{t=0}
       \bigl[A(tQ_\psi)(1+t\sigma(y))-1\bigr]&=\sigma(y).
\end{align*}
There is no linear graph contribution in the tangential component.
For the orthogonal factor,
\begin{align*}
 \frac d{dt}\bigl[B(tQ_\psi)(1+t\sigma(y))\bigr]
 &=DB(tQ_\psi)[Q_\psi](1+t\sigma(y))
    +B(tQ_\psi)\sigma(y),\\
 \left.\frac d{dt}\right|_{t=0}
       \bigl[B(tQ_\psi)(1+t\sigma(y))\bigr]&=Q_\psi.
\end{align*}
Hence the only first-order terms are
\[
 \int\frac{x-y}{|x-y|^{n+1}}\sigma(y)\,dy,
 \qquad
 \int\frac{\psi(x)-\psi(y)}{|x-y|^{n+1}}\,dy,
\]
with the same temporary truncations. The model identities
\eqref{eq:flat-Riesz}--\eqref{eq:flat-Lambda} identify their limits
with the right side of \eqref{eq:linearization}.

The calculation alone is not yet a Fr\'echet differentiability proof.
That justification is supplied independently by the multilinear estimates
and the summation in Sections~\ref{sec:linearization}--\ref{sec:nonlinear-remainder}.
They give, for $U=(\varphi,\rho)$ near zero in
$H^{s+1}\times H^s$,
\[
 \|\mathbf T(U)-\mathcal LU\|_{H^s}
 \le C_s\bigl(\|D\varphi\|_\infty+\|\rho\|_\infty\bigr)Z_s(U).
\]
Since $s>n/2$, Sobolev embedding bounds the first parenthesis by
$C_s\|U\|_{H^{s+1}\times H^s}$, and $Z_s(U)$ is bounded by the same
product norm. Therefore
\[
 \frac{\|\mathbf T(U)-\mathcal LU\|_{H^s}}
      {\|U\|_{H^{s+1}\times H^s}}
 \le C_s\|U\|_{H^{s+1}\times H^s}\longrightarrow0.
\]
This proves the asserted differential once the independently constructed
field is used, and explains the order of the argument. In particular,
we do not use differentiability to prove the multilinear bound needed
for differentiability.
\end{proof}

\subsection{The linear estimate}
Set
\begin{equation}\label{eq:Zs-definition}
 Z_s(\psi,\sigma)=\|D\psi\|_{H^s}+\|\sigma\|_{H^s}.
\end{equation}
We continue to take the functions in the inhomogeneous product space;
we do not need its completion under $Z_s$.

\begin{proposition}
\label{prop:linear-coercivity}
For every $s\ge0$ and every $\psi\in H^{s+1}$, $\sigma\in H^s$,
\begin{equation}\label{eq:linear-coercivity}
 Z_s(\psi,\sigma)\le C_n\|\mathcal L(\psi,\sigma)\|_{H^s}.
\end{equation}
Also
\begin{equation}\label{eq:linear-full-low}
 \|\psi\|_{H^{s+1}}+\|\sigma\|_{H^s}
 \le C_s\bigl(\|\mathcal L(\psi,\sigma)\|_{H^s}+\|\psi\|_2\bigr).
\end{equation}
If $\supp\psi\subset B_R$, the last $L^2$ term can be omitted with a
constant depending on $R$.  No support restriction on $\sigma$ is used.
\end{proposition}
\begin{proof}
Use the Fourier convention \eqref{eq:Fourier-convention} and the
Sobolev norm \eqref{eq:Sobolev-Fourier-convention}. For $\xi\ne0$,
\[
 \sum_{j=1}^n\left|\frac{\xi_j}{|\xi|}\right|^2=1.
\]
Multiplying by $(1+|\xi|^2)^s|\widehat\sigma(\xi)|^2$ and integrating
therefore gives
\begin{equation}\label{eq:linear-Riesz-Hs-detail}
 \|\mathcal R\sigma\|_{H^s(\R^n;\R^n)}^2
 =(2\pi)^{-n}\int(1+|\xi|^2)^s|\widehat\sigma(\xi)|^2\,d\xi
 =\|\sigma\|_{H^s}^2.
\end{equation}
Similarly, summing over all derivatives and graph components,
\begin{equation}\label{eq:linear-gradient-Hs-detail}
 \|D\psi\|_{H^s}^2
 =(2\pi)^{-n}\int(1+|\xi|^2)^s|\xi|^2|\widehat\psi(\xi)|^2\,d\xi
 =\||D|\psi\|_{H^s}^2.
\end{equation}
The orthogonality of the output components implies
\[
 \|\mathcal L(\psi,\sigma)\|_{H^s}^2
 =|c_{n,R}|^2\|\sigma\|_{H^s}^2
   +|c_{n,\Lambda}|^2\|D\psi\|_{H^s}^2.
\]
The inequality $A+B\le\sqrt2(A^2+B^2)^{1/2}$ now proves
\eqref{eq:linear-coercivity}. The value of a multiplier at $\xi=0$
does not affect any of these $L^2$ identities.

For $s\ge0$ the weights satisfy
\[
 (1+|\xi|^2)^{s+1}\asymp_s
 1+|\xi|^2(1+|\xi|^2)^s.
\]
One checks this separately on $|\xi|\le1$ and $|\xi|>1$.
After integration,
\[
 \|\psi\|_{H^{s+1}}\le C_s(\|\psi\|_2+\|D\psi\|_{H^s}),
\]
which proves \eqref{eq:linear-full-low}.

If $\supp\psi\subset B_R$, then $\psi\in H_0^1(B_{2R})$ after
restriction. The zero-boundary Poincar\'e inequality on this ball gives
$\|\psi\|_2\le C_nR\|D\psi\|_2$. No trace regularity at
$\partial B_R$ is needed, since the larger ball is used.
Substitution into the preceding inequalities completes the
compact-geometry conclusion. The density estimate never uses a
support condition on $\sigma$.
\end{proof}

Neither the vector Riesz transform nor $|D|$ loses information at a
nonzero frequency.  The failure of uniform control of $\|\psi\|_2$ is due to small
frequencies; the multiplier has no nontrivial kernel in $L^2$.  This distinction is important when a support restriction
is removed.

\subsection{Multilinear estimates}\label{sec:tame-commutators}

We prove multilinear estimates in which a differentiated coefficient
is measured by $\|Db\|_{H^s}$. The operators are higher-order Calder\'on
commutators and can also be written as Christ--Journ\'e forms. We first
prove the $L^2$ estimates needed below, keeping track of their dependence
on the number of coefficients and on the truncation parameter.

\subsection{Difference quotients and multilinear estimates}
For $x\ne y$ and a locally Lipschitz function $u$, write
\begin{equation}\label{eq:segment-mean}
 m_{x,y}u=\int_0^1u(tx+(1-t)y)\,dt.
\end{equation}
The fundamental theorem of calculus gives, with
$\theta=(x-y)/|x-y|$,
\begin{equation}\label{eq:difference-as-segment-mean}
 \frac{b(x)-b(y)}{|x-y|}
 =\sum_{\ell=1}^n\theta_\ell\,m_{x,y}(\partial_\ell b).
\end{equation}
Thus products of graph difference quotients are finite sums of
Christ--Journ\'e segment-mean forms.

We use only the two angular factors
\begin{equation}\label{eq:angular-families}
 \Omega(\theta)=1
 \quad\hbox{or}\quad
 \Omega(\theta)=\theta_\ell,
 \qquad 1\le \ell\le n.
\end{equation}
For $q\ge1$, define the distributional principal-value operator
\begin{equation}\label{eq:Cq-distribution}
 \mathfrak C_{q,\Omega}(b_1,\ldots,b_q;f)(x)
 =\operatorname{p.v.}\int_{\R^n}
 \frac{\Omega((x-y)/|x-y|)}{|x-y|^n}
 \prod_{j=1}^q
 \frac{b_j(x)-b_j(y)}{|x-y|}
 f(y)\,dy.
\end{equation}
Choose once and for all a radial function
$\chi\in C^\infty([0,\infty))$ with $\chi=0$ on $[0,1]$ and
$\chi=1$ on $[2,\infty)$.  For $0<\varepsilon<1/2$ set
\begin{equation}\label{eq:radial-two-sided-cutoff}
 \vartheta_\varepsilon(r)
 =\chi(r/\varepsilon)\,\chi(1/(\varepsilon r)).
\end{equation}
Thus $\vartheta_\varepsilon$ is supported in a centered annulus and
tends pointwise to one on $(0,\infty)$.  The smooth symmetric annular
regularizations are
\begin{equation}\label{eq:Cq-eps}
 \mathfrak C_{q,\Omega}^{\varepsilon}
 (b_1,\ldots,b_q;f)(x)
 =\int_{\R^n}
 \vartheta_\varepsilon(|x-y|)
 \frac{\Omega((x-y)/|x-y|)}{|x-y|^n}
 \prod_{j=1}^q
 \frac{b_j(x)-b_j(y)}{|x-y|}
 f(y)\,dy.
\end{equation}
For smooth compactly supported coefficient functions these regularizations
converge to \eqref{eq:Cq-distribution} in distributions.  Hard symmetric
truncations give the same principal value whenever the latter is used
pointwise.
In every application below the parity condition
\begin{equation}\label{eq:CJ-parity}
 \Omega(-\theta)=(-1)^{q+1}\Omega(\theta)
\end{equation}
holds: tangential terms have $\Omega(\theta)=\theta_\ell$ and even
$q$, while orthogonal terms have $\Omega\equiv1$ and odd $q$.
Consequently, after expanding \eqref{eq:difference-as-segment-mean},
each resulting homogeneous angular kernel has spherical mean zero and
is a regular Calder\'on--Zygmund convolution kernel.

\begin{lemma}\label{lem:angular-kernel-CZ}
For a coordinate multi-index
$\boldsymbol\ell=(\ell_1,\ldots,\ell_q)$ put
\[
 \kappa_{q,\boldsymbol\ell}(z)
 =\frac{\Omega(z/|z|)}{|z|^n}
   \prod_{j=1}^q\frac{z_{\ell_j}}{|z|}.
\]
If \eqref{eq:CJ-parity} holds, then
$\kappa_{q,\boldsymbol\ell}$ is odd and has zero spherical mean.  For
every fixed integer $N\ge0$,
\begin{equation}\label{eq:angular-CZ-seminorm}
 \sup_{z\ne0}|z|^{n+|\beta|}
 |\partial^\beta\kappa_{q,\boldsymbol\ell}(z)|
 \le C_{N,n}(1+q)^N,
 \qquad |\beta|\le N.
\end{equation}
In addition, for some fixed $N_*=N_*(n)$,
\begin{equation}\label{eq:angular-multiplier-bound}
 \|\widehat{\kappa_{q,\boldsymbol\ell}}\|_{L^\infty}
 \le C_n(1+q)^{N_*}.
\end{equation}
The annular kernels
$\vartheta_\varepsilon(|z|)\kappa_{q,\boldsymbol\ell}(z)$ satisfy the
same scale-invariant size, regularity, and Fourier-multiplier bounds,
with constants independent of $0<\varepsilon<1/2$.
\end{lemma}

\begin{proof}
Let $\Theta(\theta)=\Omega(\theta)\prod_{j=1}^q\theta_{\ell_j}$.
The parity assumption gives
\[
 \Theta(-\theta)=(-1)^{q+1}(-1)^q\Theta(\theta)=-\Theta(\theta).
\]
Thus $\int_{\mathbb S^{n-1}}\Theta=0$, and every centered annular
integral of the kernel is zero. For $n=1$, this is the same cancellation
between the two points of $\mathbb S^0$.

For every multi-index $\gamma$, the coordinate factors satisfy
\[
 \left|\partial^\gamma\frac{z_\ell}{|z|}\right|
 \le C_{\gamma,n}|z|^{-|\gamma|},
\]
Equivalently, this follows by
homogeneity and smoothness on $1/2\le|z|\le2$.
A derivative of order $k\le N$ of the product defining $\kappa$ is a
sum over distributions of $k$ derivatives among its $q$ coordinate
factors, the radial factor, and $\Omega$. The number with multiplicities
is bounded by $C_{N,n}(q+2)^N$. Each summand has size at most
$C_{N,n}|z|^{-n-k}$. This proves \eqref{eq:angular-CZ-seminorm}.

We include a dyadic argument to make the uniform multiplier bound
explicit. Choose a smooth radial function $\gamma$ supported in
$1/2<|z|<2$ with
$\sum_{j\in\mathbb Z}\gamma(2^{-j}z)=1$ for $z\ne0$. Put
\[
 k_j(z)=\gamma(2^{-j}z)\kappa(z),\qquad
 \Psi(w)=\gamma(w)\kappa(w).
\]
Homogeneity yields $k_j(z)=2^{-jn}\Psi(2^{-j}z)$, and hence
$\widehat k_j(\xi)=\widehat\Psi(2^j\xi)$.
The function $\Psi$ has integral zero. For $|\eta|\le1$,
\[
 |\widehat\Psi(\eta)|
 =\left|\int(e^{-iw\cdot\eta}-1)\Psi(w)\,dw\right|
 \le |\eta|\int|w\Psi(w)|\,dw\le C_n|\eta|.
\]
For $|\eta|>1$, integration by parts $N$ times in a coordinate
whose frequency has modulus at least $|\eta|/\sqrt n$ gives
\[
 |\widehat\Psi(\eta)|
 \le C_{N,n}(1+q)^N|\eta|^{-N}.
\]
Taking, for instance, $N=n+2$ and splitting the dyadic sum at
$2^j|\xi|=1$, we obtain
\begin{align}
 \sum_{j\in\mathbb Z}|\widehat k_j(\xi)|
 &\le C_n\sum_{2^j|\xi|\le1}2^j|\xi|
   +C_{N,n}(1+q)^N
                \sum_{2^j|\xi|>1}(2^j|\xi|)^{-N}\notag\\
 &\le C_{N,n}(1+q)^N\qquad(\xi\ne0).
 \label{eq:dyadic-angular-multiplier-detail}
\end{align}
The dyadic sum defines the principal-value convolution distribution,
so this proves \eqref{eq:angular-multiplier-bound}.

For the two-sided smooth truncations replace $\Psi(w)$ in the $j$th
annulus by
\[
 \Psi_{j,\epsilon}(w)
 =\gamma(w)\vartheta_\epsilon(2^j|w|)\kappa(w).
\]
The cutoff remains radial, and therefore its integral against the odd
kernel is zero. Its derivatives in $w$ of every fixed order are bounded
uniformly in $j,\epsilon$: in a transition annulus the dimensionless
argument of the differentiated cutoff is comparable to one, and outside
that annulus the derivative vanishes. The preceding low- and
high-frequency estimates consequently hold for $\Psi_{j,\epsilon}$
with the same constants. Summing them proves a multiplier bound uniform
in $\epsilon$. The pointwise size and regularity estimates follow from
the identical derivative observation. This proves all assertions.
\end{proof}

For the homogeneous kernels used here, the required quantitative estimate
can be obtained from the one-dimensional Calder\'on commutators by the
method of rotations. We give the reduction, including the treatment of
the smooth truncations. This specifies the dependence on the order of
the commutator without introducing a norm on the more general kernel
classes of Christ--Journ\'e and Seeger--Smart--Street
\cite{ChristJourne,SeegerSmartStreet}.

The one-dimensional result we use is \cite[Theorem~1.1, p.~1091]{MuscaluII}.
In the notation of that paper, the order of the commutator is denoted by
$d$; here it is $q$ and is unrelated to either ambient dimension.
For functions $a_1,\ldots,a_q$ on $\R$, write
\begin{equation}\label{eq:one-dimensional-CJ}
 \mathcal C_q(\mathbf a;f)(u)
 =\operatorname{p.v.}\int_{\R}\frac{f(u-t)}{t}
       \prod_{j=1}^q\left(\int_0^1a_j(u-\alpha t)\,d\alpha\right)dt.
\end{equation}
If $A_j'=a_j$, the segment average is
$(A_j(u)-A_j(u-t))/t$. Thus \eqref{eq:one-dimensional-CJ} is the
multilinear version of the commutator in \cite[(1.2)--(1.5)]{MuscaluII},
up to a fixed normalization of the Hilbert transform. Theorem~1.1 there
gives a bound of the form
\[
 C(q)C(l)\prod_{i=0}^q C(p_i),
\]
where $C(q)$ is at most polynomial in $q$, $l$ is the number of finite
 exponents, and $C(\infty)=1$. We shall first use only the cases
in which exactly one exponent is $2$ and all the others are
$\infty$. In these cases $l=1$, so there are absolute constants
$C,N_0$ such that, for every $0\le i\le q$,
\begin{equation}\label{eq:one-dimensional-vertices}
 \|\mathcal C_q(a_1,\ldots,a_q;f)\|_2
 \le C(1+q)^{N_0}\|u_i\|_2\prod_{j\ne i}\|u_j\|_\infty,
 \qquad (u_0,u_1,\ldots,u_q)=(f,a_1,\ldots,a_q).
\end{equation}
The extension from Schwartz functions to $L^\infty$ entries is part of
\cite[Section~2, pp.~1092--1093]{MuscaluII}. In particular, the endpoint
$f=1$ is permitted. No assertion that compactly supported functions are
norm dense in $L^\infty$ is needed.

\begin{theorem}\label{thm:CJ-holder}
There are constants $A=A(n)\ge1$ and $M=M(n)\ge0$ with the following
property. Let $q\ge1$, let $\Omega$ be one of the functions in
\eqref{eq:angular-families}, and assume \eqref{eq:CJ-parity}.
If $p_0,p_1,\ldots,p_q\in[2,\infty]$ satisfy
\begin{equation}\label{eq:CJ-operator-exponents}
 \sum_{j=0}^q\frac1{p_j}=\frac12,
\end{equation}
then, for $b_1,\ldots,b_q,f\in C_c^\infty(\R^n)$,
\begin{equation}\label{eq:CJ-holder}
 \|\mathfrak C_{q,\Omega}(b_1,\ldots,b_q;f)\|_2
 \le A^q(1+q)^M\|f\|_{p_0}
                         \prod_{j=1}^q\|Db_j\|_{p_j}.
\end{equation}
Equivalently, for $g\in C_c^\infty(\R^n)$ the exponents obey
\begin{equation}\label{eq:CJ-form-exponents}
 \frac1{p_0}+\sum_{j=1}^q\frac1{p_j}+\frac12=1,
\end{equation}
and
\begin{equation}\label{eq:CJ-form-bound}
 \left|\int\mathfrak C_{q,\Omega}(\mathbf b;f)(x)g(x)\,dx\right|
 \le A^q(1+q)^M\|f\|_{p_0}\|g\|_2
                         \prod_{j=1}^q\|Db_j\|_{p_j}.
\end{equation}
The same estimates hold for the smooth annular regularizations
\eqref{eq:Cq-eps}, with constants independent of $\varepsilon$.
They also hold with $f=1$ and $p_0=\infty$.
\end{theorem}

\begin{proof}

For $\theta\in\mathbb S^{n-1}$ define
\[
 \mathcal C_{q,\theta}(a_1,\ldots,a_q;f)(x)
 =\operatorname{p.v.}\int_{\R}\frac{f(x-t\theta)}{t}
   \prod_{j=1}^q\left(\int_0^1a_j(x-\alpha t\theta)\,d\alpha\right)dt.
\]
Write $x=z+u\theta$, where $z\in\theta^\perp$. For each fixed $z$,
this is \eqref{eq:one-dimensional-CJ} applied to functions of $u$.
Applying \eqref{eq:one-dimensional-vertices} and integrating the square
in $z$ gives
\begin{equation}\label{eq:directional-vertices}
 \|\mathcal C_{q,\theta}(u_1,\ldots,u_q;u_0)\|_{L^2(\R^n)}
 \le K_q\|u_i\|_{L^2(\R^n)}
                  \prod_{j\ne i}\|u_j\|_{L^\infty(\R^n)},
 \qquad K_q=C(1+q)^{N_0}.
\end{equation}
The constant is independent of $i$ and $\theta$. At the $i$th estimate,
the other factors on each line are bounded by their global supremum
norms, so Fubini's theorem applies directly to the one $L^2$ factor.
When $n=1$, the space $\theta^\perp$ is a point and no transverse
integration is needed.

The multilinear operator acts on
$(u_0,\ldots,u_q)$. Interpolate between its $q+1$ estimates
\eqref{eq:directional-vertices}. Indeed, the weights
$\lambda_i=2/p_i$ are nonnegative and sum to one, and the convex
combination of the reciprocal-exponent vectors is
$(1/p_0,\ldots,1/p_q)$. All output spaces are $L^2$.
Multilinear interpolation therefore gives
\begin{equation}\label{eq:directional-holder}
 \|\mathcal C_{q,\theta}(u_1,\ldots,u_q;u_0)\|_2
 \le K_q\prod_{i=0}^q\|u_i\|_{p_i}.
\end{equation}
The interpolation constant is at most
$\prod_iK_q^{\lambda_i}=K_q$. Thus interpolation introduces no
uncontrolled dependence on the number of finite exponents. The
$L^\infty$ entries have the compatible extensions described above.

Let $w_\varepsilon(t)=\vartheta_\varepsilon(|t|)$, where
$\vartheta_\varepsilon$ is defined in
\eqref{eq:radial-two-sided-cutoff}. There is a constant depending only
on the fixed cutoff such that
\begin{equation}\label{eq:cutoff-Fourier-L1}
 \sup_{0<\varepsilon<1/2}\|\widehat w_\varepsilon\|_{L^1(\R)}<\infty.
\end{equation}
To prove this, set
\[
 \omega_0(t)=1-\chi(|t|),\qquad
 \omega_1(t)=\chi(1/|t|),\quad\omega_1(0)=1.
\]
Both functions are smooth and compactly supported. On the support of
$\omega_0(t/\varepsilon)$, one has
$\varepsilon|t|\le2\varepsilon^2<1/2$, so
$\omega_1(\varepsilon t)=1$. Consequently
\[
 w_\varepsilon(t)=\omega_1(\varepsilon t)-\omega_0(t/\varepsilon).
\]
The Fourier $L^1$ norm is invariant under these dilations. Hence
\[
 \|\widehat w_\varepsilon\|_1
 \le\|\widehat\omega_1\|_1+\|\widehat\omega_0\|_1,
\]
which proves \eqref{eq:cutoff-Fourier-L1}.

For the one-dimensional operator, multiplication of the kernel by
$w_\varepsilon(t)$ gives
\begin{equation}\label{eq:modulated-cutoff-formula}
 \mathcal C_q^{w_\varepsilon}(\mathbf a;f)(u)
 =(2\pi)^{-1}\int_{\R}\widehat w_\varepsilon(\tau)e^{i\tau u}
       \mathcal C_q(\mathbf a;e^{-i\tau\cdot}f)(u)\,d\tau.
\end{equation}
For smooth compactly supported $a_j,f$, this identity follows from
Fourier inversion for $w_\varepsilon$. One may first remove a centered
interval about zero. To pass to the principal value, subtract the
constant term of the smooth integrand near $t=0$; the remainder for a
modulated endpoint is bounded by $C(1+|\tau|)$. The Fourier transform
of $w_\varepsilon$ is Schwartz at each fixed $\varepsilon$, so this
subtraction also justifies the interchange in the distributional
pairing. At infinity, the segment averages of compact smooth $a_j$
give an integrable tail for $q\ge1$. The same argument applies to
$f=1$; its modulation is an allowed $L^\infty$ function.

Modulation preserves all functions Lebesgue norms. Minkowski's
inequality, \eqref{eq:one-dimensional-vertices}, and
\eqref{eq:cutoff-Fourier-L1} show that the truncated one-dimensional
operators satisfy the same vertex estimates, up to a fixed factor.
Apply Fubini and the interpolation of Step~1 to obtain
\begin{equation}\label{eq:directional-truncated-holder}
 \|\mathcal C_{q,\theta}^{w_\varepsilon}
                         (u_1,\ldots,u_q;u_0)\|_2
 \le C K_q\prod_{i=0}^q\|u_i\|_{p_i},
\end{equation}
uniformly in $\varepsilon$ and $\theta$. This proves uniformity for the
particular regularizations used in the paper, rather than presuming it
from a bound for the untruncated operator.

For smooth coefficient functions put $a_{j,\theta}=\partial_\theta b_j$.
The fundamental theorem of calculus gives
\[
 \frac{b_j(x)-b_j(x-t\theta)}{t}
   =\int_0^1\partial_\theta b_j(x-\alpha t\theta)\,d\alpha,
      \qquad t\ne0.
\]
Polar coordinates and the parity condition
$\Omega(-\theta)=(-1)^{q+1}\Omega(\theta)$ yield
\begin{equation}\label{eq:CJ-rotations}
 \mathfrak C_{q,\Omega}^{\varepsilon}(\mathbf b;f)(x)
 =\frac12\int_{\mathbb S^{n-1}}\Omega(\theta)
   \mathcal C_{q,\theta}^{w_\varepsilon}
        (\partial_\theta b_1,\ldots,\partial_\theta b_q;f)(x)
       \,d\theta.
\end{equation}
At fixed $\varepsilon$, all integrals in this identity are ordinary
annular integrals. The part with $t>0$ is exactly the polar-coordinate
formula. In the part with $t<0$, set $t=-r$ and replace $\theta$ by
$-\theta$. The factor $1/t$, the $q$ directional derivatives, and
$\Omega$ contribute the sign
$(-1)(-1)^q(-1)^{q+1}=1$. Thus the two parts are equal, which accounts
for the factor $1/2$.

Since $|\partial_\theta b_j|\le|Db_j|$, Minkowski's inequality and
\eqref{eq:directional-truncated-holder} give
\begin{align}
 \|\mathfrak C_{q,\Omega}^{\varepsilon}(\mathbf b;f)\|_2
 &\le\frac12\int_{\mathbb S^{n-1}}|\Omega(\theta)|
       \|\mathcal C_{q,\theta}^{w_\varepsilon}
                 (\partial_\theta\mathbf b;f)\|_2\,d\theta\notag\\
 &\le C_n K_q\|f\|_{p_0}\prod_{j=1}^q\|Db_j\|_{p_j}.
       \label{eq:rotations-L2-bound}
\end{align}
Here $|\Omega|\le1$ for the two families in
\eqref{eq:angular-families}. In particular the constant is bounded by
$A^q(1+q)^M$ with $A,M$ depending only on $n$ and the fixed cutoff.

For compact smooth coefficients the annular integrals converge to the
principal value: the leading homogeneous term near zero is odd, and
after subtracting it the kernel is locally integrable. At infinity the
$q$ differences give decay $|x-y|^{-n-q}$. The same argument applies
when $f=1$. Passing to a pointwise limit in
\eqref{eq:rotations-L2-bound} and using Fatou's lemma proves
\eqref{eq:CJ-holder} for these assumptions. Duality gives
\eqref{eq:CJ-form-bound}. Strong convergence and the extension to the
Sobolev spaces are proved in Proposition~\ref{prop:canonical-realization}.
\end{proof}

\begin{remark}\label{rem:CJ-scope}
The identity \eqref{eq:difference-as-segment-mean} expresses the operators
above as sums of Christ--Journ\'e segment-average forms. The general
theory is developed in \cite{ChristJourne,SeegerSmartStreet}.
For the homogeneous kernels in this paper, the proof of
Theorem~\ref{thm:CJ-holder} uses only
\cite[Theorem~1.1 and Section~2]{MuscaluII}, the method of rotations,
and the cutoff argument \eqref{eq:modulated-cutoff-formula}. In
particular, no quantitative assertion about an unspecified norm on a
general convolution-kernel class is needed in the sequel. The theorem
has been stated only for output exponent $2$ and the exponents
used in the Sobolev estimates.
\end{remark}

\subsection{Interpolation inequalities}
We use a Gagliardo--Nirenberg estimate with the low norm chosen to be $L^\infty$.

\begin{lemma}\label{lem:GN-derivative}
Let $s\in\mathbb N$, $s\ge1$, and let $0\le k\le s$.  If $u\in H^s(\mathbb R^n)\cap L^\infty(\mathbb R^n)$, then, with
\[
 p_k=\begin{cases}
 \infty,&k=0,\\[1mm]
 2s/k,&1\le k\le s,
 \end{cases}
\]
one has
\begin{equation}\label{eq:GN-f}
 \|D^ku\|_{L^{p_k}}
 \le C_s
 \|u\|_{L^\infty}^{1-k/s}
 \|u\|_{H^s}^{k/s}.
\end{equation}
If $b\in H^{s+1}$ and $\nabla b\in L^\infty$, then
\begin{equation}\label{eq:GN-b}
 \|D^{k+1}b\|_{L^{p_k}}
 \le C_s
 \|\nabla b\|_{L^\infty}^{1-k/s}
 \|Db\|_{H^s}^{k/s}.
\end{equation}
The inequalities hold componentwise for vector-valued functions.
\end{lemma}

\begin{proof}
For $k=0$, the first estimate is the identity
$\|u\|_\infty=\|u\|_\infty$. For $k=s$, it is
$\|D^su\|_2\le C_s\|u\|_{H^s}$. For $0<k<s$, apply the classical
Gagliardo--Nirenberg interpolation inequality between $L^\infty$ and
$\dot W^{s,2}$. Its interpolation parameter and exponent obey
\[
 \theta=\frac{k}{s},\qquad
 \frac1p-\frac{k}{n}
 =\theta\left(\frac12-\frac{s}{n}\right)
 +(1-\theta)\frac1\infty.
\]
Substitution gives $1/p=k/(2s)$ and therefore $p=2s/k$.
The inequality reads
\[
 \|D^ku\|_{2s/k}
 \le C_{n,s}\|u\|_\infty^{1-k/s}\|D^su\|_2^{k/s}
 \le C_{n,s}\|u\|_\infty^{1-k/s}\|u\|_{H^s}^{k/s}.
\]
This proves \eqref{eq:GN-f}, first for smooth functions and then by
approximation. Apply it to each first derivative of $b$. The quantities
used are
\[
 \|Db\|_\infty,
 \quad\|D^s(Db)\|_2\le C_s\|Db\|_{H^s},
\]
which give \eqref{eq:GN-b}. The zero-order norm $\|b\|_2$ does not enter
this argument. Since the number of components and multi-indices is
fixed by $n,m,s$, combining their estimates changes only $C_{n,m,s}$.
The interpolation inequality itself is a standard result; the exponent
choice is recorded because it determines the later H\"older identity.
\end{proof}

\begin{lemma}\label{lem:weighted-AMGM}
Let $A_i,B_i\ge0$ and $\theta_i\ge0$, $i=0,\ldots,q$, with
$\sum_i\theta_i=1$.  Then
\begin{equation}\label{eq:weighted-AMGM}
 \prod_{i=0}^q A_i^{\theta_i}B_i^{1-\theta_i}
 \le\sum_{i=0}^q\theta_iA_i\prod_{j\ne i}B_j.
\end{equation}
\end{lemma}

\begin{proof}
When all $B_i>0$, write the product as
\[
 \left(\prod_{i=0}^qB_i\right)
 \prod_{i=0}^q\left(\frac{A_i}{B_i}\right)^{\theta_i}.
\]
The weighted arithmetic--geometric mean inequality, using
$\sum_i\theta_i=1$, bounds the second factor by
$\sum_i\theta_iA_i/B_i$. Multiplying out gives
\eqref{eq:weighted-AMGM}. Replace $B_i$ by $B_i+\epsilon$ and take
$\epsilon\downarrow0$ to allow vanishing $B_i$. Factors with exponent
zero are omitted. In the interpolation application this means that an
undifferentiated function contributes its low norm only; no expression
$0^0$ has to be evaluated.
\end{proof}

\subsection{Differentiating the commutator}
The change of variables $y=x-z$ is crucial: all $x$-derivatives preserve the difference-quotient structure.

\begin{lemma}\label{lem:CJ-derivative}
Let $\alpha$ be a multi-index and let the functions be smooth.  Then
\begin{equation}\label{eq:CJ-derivative}
 D^\alpha\mathfrak C_{q,\Omega}^{\varepsilon}
 (b_1,\ldots,b_q;f)
 =\sum_{\alpha_0+\cdots+\alpha_q=\alpha}
 \binom{\alpha}{\alpha_0,\ldots,\alpha_q}
 \mathfrak C_{q,\Omega}^{\varepsilon}
 (D^{\alpha_1}b_1,\ldots,D^{\alpha_q}b_q;
 D^{\alpha_0}f).
\end{equation}
\end{lemma}

\begin{proof}
With $z=x-y$, define the truncated convolution factor
\[
 k_\epsilon(z)=\vartheta_\epsilon(|z|)
       \frac{\Omega(z/|z|)}{|z|^n},\qquad
 \Delta_zb(x)=\frac{b(x)-b(x-z)}{|z|}.
\]
Then
\[
 \mathfrak C^\epsilon_{q,\Omega}(b_1,\ldots,b_q;f)(x)
 =\int k_\epsilon(z)\prod_{j=1}^q\Delta_zb_j(x)\,f(x-z)\,dz.
\]
For fixed $\epsilon$, the $z$ domain is bounded and bounded away from
zero.Differentiation under the integral is justified for smooth functions,
and every $x$-derivative acts only on the displayed function factors. In
particular,
\[
 D_x^\beta\Delta_zb(x)
 =\frac{D^\beta b(x)-D^\beta b(x-z)}{|z|}
 =\Delta_z(D^\beta b)(x),\qquad
 D_x^\beta f(x-z)=(D^\beta f)(x-z).
\]
The multinomial product rule gives
\[
 D^\alpha\left(f(x-z)\prod_{j=1}^q\Delta_zb_j(x)\right)
 =\sum_{\alpha_0+\cdots+\alpha_q=\alpha}
 \frac{\alpha!}{\alpha_0!\cdots\alpha_q!}
 (D^{\alpha_0}f)(x-z)
 \prod_{j=1}^q\Delta_z(D^{\alpha_j}b_j)(x).
\]
Integration gives \eqref{eq:CJ-derivative}. There are no derivatives
of the singular kernel or of the cutoff in this identity. Changing to
$z$ before differentiation is what makes that statement exact.
\end{proof}

For integer $s$, we use the equivalent Sobolev norm
\begin{equation}\label{eq:Hs-top-equivalence}
 \|u\|_{H^s}
 \asymp_s\|u\|_2+
 \sum_{|\alpha|=s}\|D^\alpha u\|_2.
\end{equation}
Thus only the undifferentiated estimate and the top derivatives need to be treated.

\subsection{Sobolev estimates}

\begin{lemma}\label{lem:tame-commutator}
Let $s\in\mathbb N$, $s>n/2$.  There are constants $A_s\ge1$ and $M_s\ge0$ such that, for every $q\ge1$ satisfying \eqref{eq:CJ-parity} and all smooth compactly supported functions,
\begin{align}
 \|\mathfrak C_{q,\Omega}(b_1,\ldots,b_q;f)\|_{H^s}
 \le A_s^q(1+q)^{M_s}
 \Bigg[&
 \prod_{j=1}^q\|\nabla b_j\|_{L^\infty}
 \|f\|_{H^s}\notag\\
 &+\sum_{\ell=1}^q
 \|D b_\ell\|_{H^s}
 \prod_{j\ne\ell}\|\nabla b_j\|_{L^\infty}
 \|f\|_{L^\infty}
 \Bigg].
 \label{eq:tame-f}
\end{align}
Moreover,
\begin{equation}\label{eq:tame-pure}
 \|\mathfrak C_{q,\Omega}(b_1,\ldots,b_q;1)\|_{H^s}
 \le A_s^q(1+q)^{M_s}
 \sum_{\ell=1}^q
 \|D b_\ell\|_{H^s}
 \prod_{j\ne\ell}\|\nabla b_j\|_{L^\infty}.
\end{equation}
The estimates are uniform for the smooth symmetric regularizations \eqref{eq:Cq-eps} and extend by density to the indicated Sobolev classes.
\end{lemma}

\begin{proof}
We divide the argument into four steps. Every exponent used below
belongs to $[2,\infty]$ and satisfies the condition in
Theorem~\ref{thm:CJ-holder}. Denote its constants by $A$ and $M$.
The factors arising from differentiating the products will be included
in $A_s$ and $M_s$.

Choose $p_0=2$ and $p_j=\infty$ for $j\ge1$ in Theorem \ref{thm:CJ-holder}.  Then
\begin{equation}\label{eq:CJ-L2-base}
 \|\mathfrak C_{q,\Omega}^{\varepsilon}(b_1,\ldots,b_q;f)\|_2
 \le A^q(1+q)^M
 \|f\|_2\prod_{j=1}^q\|\nabla b_j\|_\infty.
\end{equation}
This is the first term on the right of \eqref{eq:tame-f} at the $L^2$ level, after enlarging $A_s$ and $M_s$.

For the pure form, fix an index $\ell$ and take
$p_\ell=2$, $p_j=\infty$ for $j\ne\ell$, and $p_0=\infty$ with $f\equiv1$.  We obtain
\begin{equation}\label{eq:CJ-pure-L2}
 \|\mathfrak C_{q,\Omega}^{\varepsilon}(b_1,\ldots,b_q;1)\|_2
 \le A^q(1+q)^M
 \|\nabla b_\ell\|_2
 \prod_{j\ne\ell}\|\nabla b_j\|_\infty.
\end{equation}
Summing this valid estimate over $\ell$ and using
$\|\nabla b_\ell\|_2\le\|D b_\ell\|_{H^s}$ yields the $L^2$ part of \eqref{eq:tame-pure}.

Fix $|\alpha|=s$ and a term in \eqref{eq:CJ-derivative}, and put
$k_i=|\alpha_i|$, so
\begin{equation}\label{eq:derivative-partition}
 k_0+k_1+\cdots+k_q=s.
\end{equation}
For $k_i>0$, set $p_i=2s/k_i$, and set $p_i=\infty$ when $k_i=0$.  Then
\[
 \frac1{p_0}+\sum_{j=1}^q\frac1{p_j}
 =\frac1{2s}\sum_{i=0}^q k_i=\frac12.
\]
The H\"older estimate \eqref{eq:CJ-holder} therefore gives
\begin{align}
 &\|\mathfrak C_{q,\Omega}^{\varepsilon}
 (D^{\alpha_1}b_1,\ldots,D^{\alpha_q}b_q;
 D^{\alpha_0}f)\|_2\notag\\
 &\qquad\le A^q(1+q)^M
 \|D^{k_0}f\|_{L^{p_0}}
 \prod_{j=1}^q
 \|D^{k_j+1}b_j\|_{L^{p_j}}.
 \label{eq:CJ-partition-bound}
\end{align}
Apply Lemma~\ref{lem:GN-derivative}.  Introduce
\[
 A_0=\|f\|_{H^s},\quad B_0=\|f\|_\infty,
 \qquad
 A_j=\|D b_j\|_{H^s},\quad
 B_j=\|\nabla b_j\|_\infty,
\]
and $\theta_i=k_i/s$.  Because \eqref{eq:derivative-partition} gives
$\sum_i\theta_i=1$, the product in \eqref{eq:CJ-partition-bound} is bounded by
\[
 C_s\prod_{i=0}^qA_i^{\theta_i}B_i^{1-\theta_i}.
\]
Using Lemma~\ref{lem:weighted-AMGM}, this is at most
\begin{equation}\label{eq:tame-partition}
 C_s\left(
 A_0\prod_{j=1}^qB_j
 +\sum_{\ell=1}^qA_\ell B_0
 \prod_{j\ne\ell}B_j
 \right).
\end{equation}
At most $s$ factors have a positive derivative order, so the interpolation constants are bounded independently of $q$.  The right-hand side is independent of the derivative partition.

The sum of the multinomial coefficients is exactly
\[
 \sum_{\alpha_0+\cdots+\alpha_q=\alpha}
 \binom{\alpha}{\alpha_0,\ldots,\alpha_q}
 =(q+1)^{|\alpha|}=(q+1)^s.
\]
Thus the sum of the differentiation coefficients is polynomial in $q$
for fixed $s$.  Summing \eqref{eq:tame-partition} over \eqref{eq:CJ-derivative} yields
\begin{align}\label{eq:top-derivative-tame}
 \|D^\alpha\mathfrak C_{q,\Omega}^{\varepsilon}(b_1,\ldots,b_q;f)\|_2
 \le A_s^q(1+q)^{M_s}\Bigg[
 &\|f\|_{H^s}\prod_j\|\nabla b_j\|_\infty\\
 &+\sum_\ell\|D b_\ell\|_{H^s}
 \|f\|_\infty
 \prod_{j\ne\ell}\|\nabla b_j\|_\infty
 \Bigg].\notag
\end{align}

Now $f\equiv1$, so $\alpha_0=0$.  The same exponent choice and interpolation argument apply to the $q$ coefficient factors, with
$\sum_{j=1}^qk_j=s$.  The weighted geometric-mean reduction gives
\[
 \|D^\alpha\mathfrak C_{q,\Omega}^{\varepsilon}(b_1,\ldots,b_q;1)\|_2
 \le A_s^q(1+q)^{M_s}
 \sum_{\ell=1}^q
 \|D b_\ell\|_{H^s}
 \prod_{j\ne\ell}\|\nabla b_j\|_\infty.
\]
Together with \eqref{eq:CJ-pure-L2} and \eqref{eq:Hs-top-equivalence}, this proves the truncated version of \eqref{eq:tame-pure}.

For each fixed smooth regularization, Steps 1--3 have already proved the
estimates; in particular they do not rely on the existence of an
untruncated operator. The case of a zero low norm causes no problem:
if $\|Db_j\|_\infty=0$, the corresponding difference quotient is
identically zero, and the same is true for its high derivative norm.
The assertion for arbitrary Sobolev entries will be a norm-continuous
extension, rather than a pointwise differentiation of nonsmooth functions.

All estimates are uniform in $\varepsilon$ for the smooth annular
regularizations \eqref{eq:Cq-eps}.  For smooth compactly supported
functions, the odd leading homogeneous term cancels on every centered
annulus and the Taylor remainder is locally integrable after all
$x$-derivatives of order at most $s$.  Thus the regularizations converge
in distributions (indeed in $H^s$) to the principal-value form; the
complete convergence argument is recorded in Proposition
\ref{prop:canonical-realization}.  Lower semicontinuity gives
\eqref{eq:tame-f}--\eqref{eq:tame-pure}.

Finally, approximate the Sobolev spaces by smooth compactly supported
functions.  Multilinearity and the uniform tame estimate show that the
corresponding outputs are Cauchy in $H^s$, which defines the stated
extension and completes the proof.
\end{proof}

\subsection{The commutator with $|D|$}
The mixed quadratic term in the orthogonal field is exactly the first Calder\'on commutator.

\begin{lemma}\label{lem:Lambda-commutator}
Let $s>n/2$ be an integer.  For $\varphi\in H^{s+1}\cap W^{1,\infty}$ and
$\rho\in H^s\cap L^\infty$,
\begin{equation}\label{eq:Lambda-commutator-identity}
 \operatorname{p.v.}\int
 \frac{\varphi(x)-\varphi(y)}{|x-y|^{n+1}}
 \rho(y)\,dy
 =c_{n,\Lambda}[|D|,\varphi]\rho,
\end{equation}
and
\begin{equation}\label{eq:Lambda-commutator-estimate}
 \|[|D|,\varphi]\rho\|_{H^s}
 \le C_s\left(
 \|D\varphi\|_\infty\|\rho\|_{H^s}
 +\|D\varphi\|_{H^s}\|\rho\|_\infty
 \right).
\end{equation}
\end{lemma}

\begin{proof}
First take smooth compactly supported functions.  Using the singular-integral formula for $|D|$,
\begin{align*}
 |D|(\varphi\rho)(x)-\varphi(x)|D|\rho(x)
 &=c_{n,\Lambda}^{-1}\operatorname{p.v.}\int
 \frac{\varphi(x)\rho(x)-\varphi(y)\rho(y)
       -\varphi(x)(\rho(x)-\rho(y))}
 {|x-y|^{n+1}}\,dy\\
 &=c_{n,\Lambda}^{-1}\operatorname{p.v.}\int
 \frac{\varphi(x)-\varphi(y)}{|x-y|^{n+1}}
 \rho(y)\,dy.
\end{align*}
For the stated Sobolev spaces, $[|D|,\varphi]\rho$ is initially
understood as a distribution.  Approximation by smooth functions and the
$q=1$ bound identify it with the convergent $H^s$ extension of the
combined difference integral.  Its principal-value representative is
justified in Proposition~\ref{prop:canonical-realization} below.
No separate pointwise principal value for $|D|\rho$ is required.
This proves \eqref{eq:Lambda-commutator-identity} in this sense.  Estimate \eqref{eq:Lambda-commutator-estimate} is the $q=1$, $\Omega\equiv1$ case of Lemma \ref{lem:tame-commutator}.  It is a derivative-norm form of the commutator estimate.  We use the proved $q=1$ case rather than requiring a separate endpoint version of Kato--Ponce; compare \cite{KatoPonce,Muscalu}.
\end{proof}

\begin{proposition}[Convergence in Sobolev spaces]
\label{prop:canonical-realization}
Let $s>n/2$ be an integer.  For $b_j\in H^{s+1}$ and $f\in H^s$,
or with the endpoint fixed at $f=1$, the forms in
Lemma~\ref{lem:tame-commutator} have a unique bounded extension with values in $H^s$.
The smooth annular regularizations converge to it in $H^s$.
It agrees with the hard centered principal value.  The functions
need not have compact support.
\end{proposition}
\begin{proof}
We separate smooth convergence, Sobolev extension, and identification of
the pointwise representative. This order is needed because convergence
of truncated operators in distributions, by itself, would not establish
strong $H^s$ convergence.

 Assume first that all
nonconstant entries are smooth with compact support. Keep $f=1$ in the
pure case. By Lemma~\ref{lem:CJ-derivative}, each derivative of order at
most $s$ is a finite sum of forms with smooth differentiated entries.
For any one such form and $z=x-y$, Taylor's formula gives locally
uniformly in $x$,
\[
 \frac{b_j(x)-b_j(x-z)}{|z|}
 =Db_j(x)\cdot\theta+O(|z|),\qquad
 f(x-z)=f(x)+O(|z|),\qquad \theta=z/|z|.
\]
The elementary product identity
$\prod_j u_j-\prod_jv_j
 =\sum_\ell(u_\ell-v_\ell)\prod_{j<\ell}u_j\prod_{j>\ell}v_j$
shows that the integrand equals
\[
 \frac{\Omega(\theta)}{|z|^n}
   f(x)\prod_{j=1}^q(Db_j(x)\cdot\theta)
       +O(|z|^{-n+1}).
\]
The leading angular function is odd by \eqref{eq:CJ-parity}, so its
integral over any centered radial annulus is zero, even with a smooth
radial weight. Thus the part removed at the inner scale is bounded by
\begin{equation}\label{eq:CJ-smooth-inner-error}
 C\int_{|z|<2\epsilon}|z|^{-n+1}\,dz\le C\epsilon.
\end{equation}
The same conclusion applies to all the differentiated forms. Only
smooth functions have been Taylor expanded in this step.

 On a fixed
compact set of $x$, boundedness of the entries gives an integrand bound
$C|z|^{-n-q}$ for $|z|>1$. Consequently the outer cutoff error is at most
\begin{equation}\label{eq:CJ-smooth-outer-error}
 C\int_{|z|>(2\epsilon)^{-1}}|z|^{-n-q}\,dz
 \le C\epsilon^q.
\end{equation}
Combining \eqref{eq:CJ-smooth-inner-error} and
\eqref{eq:CJ-smooth-outer-error} gives local uniform convergence of all
output derivatives up to order $s$.

Take a ball containing the supports of all coefficient entries. For
$x$ outside twice that ball, $b_j(x)=0$ for every $j$, and a nonzero
product of differences requires $y$ to lie in the intersection of their
supports. The differentiated versions have the same property. Hence,
uniformly in $\epsilon$,
\begin{equation}\label{eq:CJ-smooth-global-majorant}
 |D^\gamma\mathfrak C^\epsilon_{q,\Omega}(x)|
 +|D^\gamma\mathfrak C_{q,\Omega}(x)|
 \le C(1+|x|)^{-n-q},\qquad |\gamma|\le s.
\end{equation}
This also holds when a compact endpoint function is present. The bound is
square integrable, since $q\ge1$ and $n\ge1$. Local uniform convergence
and dominated convergence in the exterior now prove convergence in
$L^2$ of every derivative up to order $s$, and thus in $H^s$.
Hard centered cutoffs give the same smooth limit because the leading
term still cancels and the estimates for the inner and outer errors are
unchanged.

 Let
$b_j^{(N)}\to b_j$ in $H^{s+1}$ and $f^{(N)}\to f$ in $H^s$, with
smooth compactly supported approximants. In the pure form the last
entry remains exactly $1$. Sobolev embedding, using $s>n/2$, gives
\[
 \|Db_j^{(N)}-Db_j\|_\infty\to0,
 \quad\|Db_j^{(N)}-Db_j\|_{H^s}\to0,
 \quad\|f^{(N)}-f\|_\infty+\|f^{(N)}-f\|_{H^s}\to0.
\]
For compact smooth functions, multilinearity gives the exact telescoping
identity
\begin{align*}
 C(\mathbf b;f)-C(\widetilde{\mathbf b};\widetilde f)
 ={}&C(\mathbf b;f-\widetilde f)\\
 &+\sum_{\ell=1}^q
 C(\widetilde b_1,\ldots,\widetilde b_{\ell-1},
 b_\ell-\widetilde b_\ell,b_{\ell+1},\ldots,b_q;\widetilde f).
\end{align*}
Here $C$ can be either a smooth regularization or the smooth limit.
The uniform estimates in Lemma~\ref{lem:tame-commutator} bound each term
by a difference norm times uniformly bounded norms of the remaining
entries. For fixed $q$ and a bounded family of functions, this yields
\begin{equation}\label{eq:CJ-Sobolev-telescoping}
 \|C(U)-C(\widetilde U)\|_{H^s}
 +\sup_\epsilon\|C^\epsilon(U)-C^\epsilon(\widetilde U)\|_{H^s}
 \le C_{q,U,\widetilde U}\|U-\widetilde U\|_{(H^{s+1})^q\times H^s}.
\end{equation}
First applying this to pairs of approximants makes the limits Cauchy
and defines $C(U)$ independently of the approximating sequence. The
regularized integral for the nonsmooth entries agrees with the
corresponding extension: at fixed $\epsilon$ its kernel is integrable
on a bounded annulus and the entries converge uniformly. For the pure
form omit the endpoint difference in the last display. No assertion of
norm density in $L^\infty$ is involved.

For the convergence of the regularizations, write
\begin{align*}
 \|C^\epsilon(U)-C(U)\|_{H^s}
 \le{}&\|C^\epsilon(U)-C^\epsilon(U^{(N)})\|_{H^s}\\
 &+\|C^\epsilon(U^{(N)})-C(U^{(N)})\|_{H^s}
 +\|C(U^{(N)})-C(U)\|_{H^s}.
\end{align*}
The first and third terms are uniformly small for large fixed $N$,
by \eqref{eq:CJ-Sobolev-telescoping}. The middle term tends to zero as
$\epsilon\downarrow0$ by Step 2. This proves strong $H^s$ convergence.

For later use in the nonlinear series, these approximations may also be
chosen in a common small-slope class. To see this, multiply each function by smooth cutoffs tending to one
and then convolve with a nonnegative mollifier. The resulting compact smooth functions converge in
$H^{s+1}$ for graph coefficients and in $H^s$ for density perturbations.
Since $s>n/2$, their gradients and densities converge uniformly.
If $\|D\varphi\|_\infty\le\eta_*/2$, all approximants with sufficiently large $N$ satisfy $\|D\varphi^{(N)}\|_\infty<\eta_*$.
If also $\|\rho\|_\infty\le1/4$, they may be chosen with
$\|\rho^{(N)}\|_\infty\le1/2$. Discarding finitely many terms gives
one approximation sequence with these bounds. The constant endpoint
$f=1$ in the pure forms is not approximated.

 Fix
$0<\alpha<\min\{1,s-n/2\}$. Sobolev embedding provides
$b_j\in C^{1,\alpha}$ and $f\in C^\alpha$, with the relevant uniform
norms controlled by their Sobolev norms. Taylor's formula with this
H\"older remainder gives
\[
 \frac{b_j(x)-b_j(x-z)}{|z|}
   =Db_j(x)\cdot\theta+O(|z|^\alpha),\qquad
 f(x-z)=f(x)+O(|z|^\alpha).
\]
The leading term is the same odd homogeneous kernel as in Step 1.
The remainder is now bounded by $C|z|^{-n+\alpha}$, whose integral
on $|z|<r$ is $C r^\alpha/\alpha$. The far integrand is bounded by
$C|z|^{-n-q}$ because all entries are bounded. This proves existence
of the centered principal value. The same bounds with constants
uniform in $N$ hold for the smooth approximants, and their differences
on a fixed annulus tend uniformly to zero. Splitting at small and large radii therefore shows that these pointwise
values converge locally uniformly to the value corresponding to the
limiting functions. Their $H^s$ limit
already equals $C(U)$ by Step 3. They are thus the same representative.
This argument requires only $s>n/2$, not pointwise smoothness of the top
Sobolev derivatives. It makes no endpoint assertion when $s=n/2$.
\end{proof}

\section{Estimates for the nonlinear terms}\label{sec:nonlinear-remainder}
This section estimates the difference between the normalized Riesz field and its flat linearization. Under a small-slope assumption, we expand the kernels into convergent multilinear series and apply Section \ref{sec:linearization} to estimate the tangential and orthogonal remainders, treating the quadratic graph-density interaction separately. The resulting bound contains the small factor \(\|D\phi\|_\infty+\|\rho\|_\infty\), with constants independent of the supports, and supplies the error estimate needed for the absorption argument in Section \ref{sec:high-order-proof}.

Throughout this section $s>n/2$ is an integer, $\varphi\in H^{s+1}$, and $\rho\in H^s$, without support restrictions.  Put
\begin{equation}\label{eq:eta-r0}
 \eta=\|D\varphi\|_{L^\infty},
 \qquad
 r_0=\|\rho\|_{L^\infty}.
\end{equation}
Write
\begin{equation}\label{eq:N-definition}
 \mathbf T(\varphi,\rho)
 =D\mathbf T(0,0)[\varphi,\rho]
 +N(\varphi,\rho),
 \qquad
 N=(N^\top,N^\perp).
\end{equation}
We prove the refined component estimates stated below.  In addition to yielding the small-factor bound needed for absorption, the refined form records the quadratic and cubic cancellations of the graph geometry.

\subsection{Expansion of the kernels}
Let $p=(n+1)/2$.  For $|t|<1$,
\begin{equation}\label{eq:binomial-series}
 (1+t)^{-p}=\sum_{k=0}^\infty\alpha_kt^k,
 \qquad
 \alpha_k=(-1)^k\frac{(p)_k}{k!}.
\end{equation}
For $k\ge1$, the absolute coefficient also satisfies
\[
 |\alpha_k|=\prod_{j=1}^k\left(1+\frac{p-1}{j}\right).
\]
Here $p=(n+1)/2\ge1$. If $p=1$, this product is one. Otherwise,
$\log(1+t)\le t$ for $t\ge0$ and the harmonic-sum bound give
\[
 \log|\alpha_k|
 \le(p-1)\sum_{j=1}^k\frac1j
 \le(p-1)(1+\log k).
\]
Consequently the coefficients have polynomial growth:
\begin{equation}\label{eq:alpha-growth}
 |\alpha_k|\le C_p(1+k)^{p-1}.
\end{equation}
Hence, for $|q|<1$,
\begin{equation}\label{eq:A-B-series}
 A(q)=1+\sum_{k\ge1}\alpha_k|q|^{2k},
 \qquad
 B(q)=q+\sum_{k\ge1}\alpha_kq|q|^{2k}.
\end{equation}

For an ordered $k$-tuple
$\mathbf i=(i_1,\ldots,i_k)\in\{1,\ldots,m\}^k$, write
\begin{equation}\label{eq:Q-monomial}
 Q_{\mathbf i}^{(2k)}
 =Q_{i_1}^2\cdots Q_{i_k}^2.
\end{equation}
Then
\begin{equation}\label{eq:Q-power-expand}
 |Q|^{2k}=\sum_{\mathbf i\in\{1,\ldots,m\}^k}
 Q_{\mathbf i}^{(2k)}.
\end{equation}
There are $m^k$ terms.  Every monomial in \eqref{eq:Q-power-expand} is represented by a Christ--Journ\'e form with $2k$ graph difference quotients.  In the orthogonal component, multiplication by the output factor $Q_a$ produces $2k+1$ graph difference quotients.

\begin{lemma}\label{lem:series-convergence}
Fix an integer $s>n/2$.  There exists $\eta_*=\eta_*(n,m,s)>0$ such that, whenever $\eta<\eta_*$, the series obtained by substituting \eqref{eq:A-B-series} into \eqref{eq:T-tangent}--\eqref{eq:T-normal} converge absolutely in $H^s$.  Their sums agree with the principal-value fields defined by \eqref{eq:T-tangent}--\eqref{eq:T-normal}.
\end{lemma}

\begin{proof}
Let $G=\|D\varphi\|_{H^s}$ and $P=\|\rho\|_{H^s}$.
At a fixed level $k$, the coordinate expansion has $m^k$ monomials.
Each tangential monomial has $2k$ graph factors and each orthogonal
monomial has $2k+1$. Applying Lemma~\ref{lem:tame-commutator} before
summing the components gives the following bounds. After enlarging fixed
constants, choose $b=mA_s^2\ge1$ and $J>0$, depending only on $n,m,s$,
so that for every $k\ge1$,
\begin{align}
 \|G_k^\top\|_{H^s}
   &\le C b^k(1+k)^J\eta^{2k-1}G,\notag\\
 \|H_k^\top\|_{H^s}
   &\le C b^k(1+k)^J
          (\eta^{2k}P+r_0\eta^{2k-1}G),\notag\\
 \|G_k^\perp\|_{H^s}
   &\le C b^k(1+k)^J\eta^{2k}G,\notag\\
 \|H_k^\perp\|_{H^s}
   &\le C b^k(1+k)^J
          (\eta^{2k+1}P+r_0\eta^{2k}G).
 \label{eq:five-series-majorants}
\end{align}
The remaining orthogonal mixed term has $q=1$ and satisfies
$C(r_0G+\eta P)$; it is not part of the sum over $k\ge1$.
The polynomial $(1+k)^J$ includes the coefficient bound
\eqref{eq:alpha-growth}, the number of choices of a coefficient measured in $H^s$,
and the polynomial in the commutator estimate. The factor $mA_s$ from
the additional orthogonal coordinate is independent of $k$ and is
absorbed into $C$.

Choose $\eta_*\le1/2$ so that $b\eta_*^2\le1/4$. For any fixed
nonnegative $J$,
\begin{equation}\label{eq:numerical-series-uniform}
 \sup_{0\le\theta\le1/4}
 \sum_{k\ge1}(1+k)^J\theta^{k-1}<\infty.
\end{equation}
For example, this follows from convergence of a polynomial times a
geometric sequence at $\theta=1/4$ and termwise domination.
For $0<\eta\le\eta_*$, factor the smallest powers explicitly:
\begin{align*}
 b^k\eta^{2k-1}&=b\eta(b\eta^2)^{k-1},&
 b^k\eta^{2k}&=b\eta^2(b\eta^2)^{k-1},\\
 b^k\eta^{2k+1}&=b\eta^3(b\eta^2)^{k-1}.&&
\end{align*}
There is no division by $\eta$ in these factorizations. If $\eta=0$,
then every graph difference is zero and all the displayed nonlinear
terms vanish, so the conclusion also holds at that endpoint.
Equations~\eqref{eq:five-series-majorants}--\eqref{eq:numerical-series-uniform}
prove absolute convergence in $H^s$ of every series, uniformly for the
smooth radial regularizations.

We next identify the sum. At fixed $\epsilon>0$ the regularized
integrals are over a bounded annulus avoiding zero. Since
$|Q_\varphi(x,y)|\le\eta_*<1$, the scalar binomial series is uniformly
absolutely convergent there. It may therefore be integrated term by
term, and the regularized nonlinear field is precisely the sum of the
regularized terms. Denote these terms by $F_k^\epsilon$ and their
$H^s$ limits by $F_k$, treating the finite $q=1$ term separately.
For any $K$,
\begin{align}
 \left\|\sum_{k\ge1}F_k^\epsilon-\sum_{k\ge1}F_k\right\|_{H^s}
 \le{}&\sum_{k=1}^K\|F_k^\epsilon-F_k\|_{H^s}\notag\\
 &+\sum_{k>K}\|F_k^\epsilon\|_{H^s}
  +\sum_{k>K}\|F_k\|_{H^s}.
 \label{eq:series-truncation-interchange}
\end{align}
The last two sums are uniformly small for large fixed $K$, by the
geometric majorants; the finite sum tends to zero by
Proposition~\ref{prop:canonical-realization}. Thus the infinite series
and the removal of the regularization commute in $H^s$.

Finally, the pointwise local singularity of the full nonlinear kernel
has the same cancellation. With $y=x-z$, one has
$Q_\varphi=D\varphi(x)\theta+O(|z|^\alpha)$ and
$\rho(x-z)=\rho(x)+O(|z|^\alpha)$ for
$0<\alpha<\min\{1,s-n/2\}$. The tangential frozen factor is the product
of $\theta$ and an even function of $D\varphi(x)\theta$; the orthogonal
frozen factor is an odd function of $D\varphi(x)\theta$.
Both are odd in $\theta$, and subtracting them leaves
$O(|z|^{-n+\alpha})$. The nonlinear far tails have at least two graph quotients tangentially
and at least one orthogonally; boundedness of the graph function and
density perturbation makes them absolutely integrable. Therefore hard centered
principal values have the same distributional limit as the smooth
radial regularizations. Equation~\eqref{eq:series-truncation-interchange}
identifies that limit with the $H^s$ sum. The linear density term is
handled by the ordinary Riesz $H^s$ operator. This completes the
identification without relying on pointwise convergence of a series
alone.
\end{proof}

\subsection{The tangential component}
From \eqref{eq:T-tangent} and \eqref{eq:A-B-series}, the linear term is the flat Riesz transform of $\rho$.  Thus
\begin{equation}\label{eq:Ntop-series}
 N^\top=N^\top_{\rm geom}+N^\top_{\rm mix},
\end{equation}
where, componentwise,
\begin{align}
 N^\top_{\rm geom}
 &=\sum_{k\ge1}\alpha_k
 \operatorname{p.v.}\int
 \frac{z}{|z|^{n+1}}|Q_\varphi(x,y)|^{2k}\,dy,
 \label{eq:Ntop-geom}\\
 N^\top_{\rm mix}
 &=\sum_{k\ge1}\alpha_k
 \operatorname{p.v.}\int
 \frac{z}{|z|^{n+1}}|Q_\varphi(x,y)|^{2k}
 \rho(y)\,dy.
 \label{eq:Ntop-mix}
\end{align}

For a fixed output coordinate $\ell\le n$ and a fixed tuple
$\mathbf i=(i_1,\ldots,i_k)$, the corresponding integral with endpoint $1$ is exactly
\[
 \mathfrak C_{2k,\theta_\ell}
 (\varphi_{i_1},\varphi_{i_1},\ldots,
  \varphi_{i_k},\varphi_{i_k};1).
\]
The mixed integral has the same coefficient entries and endpoint
$\rho$. The angular function is odd and the number $2k$ is even,
so \eqref{eq:CJ-parity} holds.

For a monomial in \eqref{eq:Ntop-geom}, all $2k$ graph factors are components of $\varphi$.  Estimate \eqref{eq:tame-pure} gives
\begin{equation}\label{eq:Ntop-geom-level}
 \|G_k^\top\|_{H^s}
 \le C_s|\alpha_k|m^kA_s^{2k}(1+2k)^{M_s}
 (2k)\eta^{2k-1}
 \|D\varphi\|_{H^s}.
\end{equation}
The factor $2k$ records the choice of the graph factor carrying the high Sobolev norm and can be absorbed into the polynomial factor.

For \eqref{eq:Ntop-mix}, \eqref{eq:tame-f} yields
\begin{align}
 \|H_k^\top\|_{H^s}
 \le C_s|\alpha_k|m^kA_s^{2k}(1+2k)^{M_s}
 \Bigl(&\eta^{2k}\|\rho\|_{H^s}\notag\\
 &+2k\eta^{2k-1}r_0
 \|D\varphi\|_{H^s}\Bigr).
 \label{eq:Ntop-mix-level}
\end{align}
Summation of \eqref{eq:Ntop-geom-level}--\eqref{eq:Ntop-mix-level} gives
\begin{equation}\label{eq:N-tangent}
 \|N^\top(\varphi,\rho)\|_{H^s}
 \le C_s\Bigl[
 (1+r_0)\eta\|D\varphi\|_{H^s}
 +\eta^2\|\rho\|_{H^s}
 \Bigr].
\end{equation}
Indeed, every geometric series begins with $k=1$, so the pure term contains one remaining factor $\eta$, while the term involving $\|\rho\|_{H^s}$ contains $\eta^2$.

\subsection{The orthogonal component}
The orthogonal field contains the linear graph term $c_{n,\Lambda}|D|\varphi$ and the mixed quadratic term produced by multiplying the linear factor $Q_\varphi$ by $\rho$.  We therefore write
\begin{equation}\label{eq:Nperp-decomp}
 N^\perp=N^\perp_{\rm geom}
 +N^\perp_{\rm low\,mix}
 +N^\perp_{\rm high\,mix}.
\end{equation}
The lowest mixed term is
\begin{equation}\label{eq:lowest-mixed}
 N^\perp_{\rm low\,mix}(x)
 =\operatorname{p.v.}\int
 \frac{\varphi(x)-\varphi(y)}{|x-y|^{n+1}}
 \rho(y)\,dy
 =c_{n,\Lambda}[|D|,\varphi]\rho.
\end{equation}
By Lemma \ref{lem:Lambda-commutator},
\begin{equation}\label{eq:KP-mixed}
 \|N^\perp_{\rm low\,mix}\|_{H^s}
 \le C_s\left(
 r_0\|D\varphi\|_{H^s}
 +\eta\|\rho\|_{H^s}
 \right).
\end{equation}

For a fixed orthogonal output coordinate $a\le m$ and a tuple
$\mathbf i$, the pure level-$k$ integral is
\[
 \mathfrak C_{2k+1,1}
 (\varphi_a,\varphi_{i_1},\varphi_{i_1},\ldots,
  \varphi_{i_k},\varphi_{i_k};1),
\]
and its higher mixed counterpart has endpoint $\rho$.
The angular factor is even and the number $2k+1$ is odd, so the parity hypothesis holds. The $k=0$ mixed term is already
\eqref{eq:lowest-mixed}; including it again in the higher mixed sum
would double count the quadratic interaction.

At level $k\ge1$, the pure orthogonal monomial has $2k+1$ graph factors.  The pure tame estimate gives
\begin{equation}\label{eq:Nperp-geom-level}
 \|G_k^\perp\|_{H^s}
 \le C_s|\alpha_k|m^{k+1}A_s^{2k+1}(1+2k)^{M_s}
 (2k+1)\eta^{2k}
 \|D\varphi\|_{H^s}.
\end{equation}
The fixed factor $m$ accounts for the output components and is
absorbed into the constant.  Summing from $k=1$ produces
\begin{equation}\label{eq:Nperp-geom-sum}
 \|N^\perp_{\rm geom}\|_{H^s}
 \le C_s\eta^2\|D\varphi\|_{H^s}.
\end{equation}

For the higher-order mixed orthogonal terms, \eqref{eq:tame-f} gives
\begin{align}
 \|H_k^\perp\|_{H^s}
 \le C_s|\alpha_k|m^{k+1}A_s^{2k+1}(1+2k)^{M_s}
 \Bigl(&\eta^{2k+1}\|\rho\|_{H^s}\notag\\
 &+(2k+1)\eta^{2k}r_0
 \|D\varphi\|_{H^s}\Bigr).
 \label{eq:Nperp-high-mix-level}
\end{align}
After summation and using $\eta\le1$,
\begin{equation}\label{eq:Nperp-high-mix-sum}
 \|N^\perp_{\rm high\,mix}\|_{H^s}
 \le C_s\left(
 \eta^3\|\rho\|_{H^s}
 +\eta^2r_0\|D\varphi\|_{H^s}
 \right).
\end{equation}
Combining \eqref{eq:KP-mixed}, \eqref{eq:Nperp-geom-sum}, and \eqref{eq:Nperp-high-mix-sum} yields
\begin{equation}\label{eq:N-normal}
 \|N^\perp(\varphi,\rho)\|_{H^s}
 \le C_s\Bigl[
 (\eta^2+r_0)\|D\varphi\|_{H^s}
 +\eta\|\rho\|_{H^s}
 \Bigr].
\end{equation}

We have proved the following proposition.

\begin{proposition}\label{prop:remainder}
Let $s>n/2$ be an integer.  There exists
$\eta_*=\eta_*(n,m,s)>0$ such that, whenever
$\eta\le\eta_*$, the estimates \eqref{eq:N-tangent} and \eqref{eq:N-normal} hold.  Consequently, if $\eta+r_0\le1$, then
\begin{equation}\label{eq:remainder-main}
 \|N(\varphi,\rho)\|_{H^s}
 \le C_s(\eta+r_0)
 \bigl(\|D\varphi\|_{H^s}+\|\rho\|_{H^s}\bigr).
\end{equation}
\end{proposition}

\begin{proof}
The preceding level estimates and
\eqref{eq:numerical-series-uniform} prove the two refined bounds.
We record the absorption-size consequence with all coefficients
visible. Set $G=\|D\varphi\|_{H^s}$ and $P=\|\rho\|_{H^s}$.
When $\delta=\eta+r_0\le1$,
\[
 (1+r_0)\eta\le2\eta\le2\delta,\qquad
 \eta^2\le\eta\le\delta,\qquad
 \eta^2+r_0\le\eta+r_0=\delta.
\]
Thus
\begin{align*}
 \|N\|_{H^s}
 &\le\|N^\top\|_{H^s}+\|N^\perp\|_{H^s}\\
 &\le C_s\bigl[((1+r_0)\eta+\eta^2+r_0)G
                  +(\eta^2+\eta)P\bigr]\\
 &\le C_s\delta(G+P),
\end{align*}
which is \eqref{eq:remainder-main}. The constant is independent of the supports of $\varphi$ and $\rho$
and of their higher-order Sobolev norms. Those high
norms occur linearly on the right; they are not among the small
parameters. This distinction is preserved in the absorption argument.
\end{proof}

\begin{remark}\label{rem:second-variation}
The same cancellations can be displayed without the full power series.  Put
$\mathscr F(t)=\mathbf T(t\varphi,t\rho)$.  Then
\[
 \mathbf T(\varphi,\rho)-D\mathbf T(0,0)[\varphi,\rho]
 =\int_0^1(1-t)\mathscr F''(t)\,dt.
\]
For the tangential field, $DA(0)=0$ ensures that every term in
$\mathscr F''$ contains at least two graph differences or one graph difference together with an additional small coefficient.  For the orthogonal field, $D^2B(0)=0$ makes the pure geometric remainder cubic, while the only quadratic term is the commutator \eqref{eq:lowest-mixed}.  The series proof above is useful because it reduces every contribution explicitly to the multilinear forms of Section \ref{sec:tame-commutators} and verifies convergence in $H^s$.
\end{remark}

\section{Proof of the Sobolev stability theorem}\label{sec:high-order-proof}
This section completes the proof of Theorem \ref{thm:high-stability} for integers \(s>n/2\). After establishing that the normalized field belongs to \(H^s\), we combine the linear lower estimate with the nonlinear remainder bound and absorb the error under the stated \(L^\infty\) smallness condition. This gives control of the graph gradient and density perturbation with constants independent of the supports, while control of the full graph norm requires an additional low-frequency condition. Analytic dependence further yields a two-point estimate and local injectivity under the stronger smallness condition on \(Z_s\).
\begin{proposition}\label{prop:local-mapping}
Let $s>n/2$ be an integer, $\varphi\in H^{s+1}$ and $\rho\in H^s$.
If $\eta=\|D\varphi\|_\infty\le\eta_*$, then
$\mathbf T(\varphi,\rho)\in H^s$ and
\begin{equation}\label{eq:local-mapping}
 \|\mathbf T(\varphi,\rho)\|_{H^s}
 \le C_s(1+\eta+\|\rho\|_\infty)
       Z_s(\varphi,\rho).
\end{equation}
\end{proposition}
\begin{proof}
The candidate linear term is in $H^s$ by the model operator identities.
Lemma~\ref{lem:series-convergence} constructs the nonlinear part in
$H^s$, independently of any lower estimate. Therefore
\[
 \|\mathbf T\|_{H^s}
 \le C_sZ_s(\varphi,\rho)
      +\|N^\top\|_{H^s}+\|N^\perp\|_{H^s}.
\]
Since $\eta_*\le1$, the refined estimates bound the right side by
$C_s(1+\eta+\|\rho\|_\infty)Z_s(\varphi,\rho)$.
This proves the mapping bound before using linear coercivity.

We make explicit why the original integrals have no additional
normalization ambiguity at infinity. Sobolev embedding gives bounded
$\varphi$ and $\rho$. For $|z|>1$,
\[
 |Q_\varphi(x,x-z)|\le2\|\varphi\|_\infty|z|^{-1},\qquad
 |A(Q_\varphi)-1|\le C\|\varphi\|_\infty^2|z|^{-2},\qquad
 |B(Q_\varphi)|\le C\|\varphi\|_\infty|z|^{-1}.
\]
The tangential integrand splits as
\[
 \frac z{|z|^{n+1}}\rho(x-z)
 +\frac z{|z|^{n+1}}(A(Q_\varphi)-1)(1+\rho(x-z)).
\]
The second term is bounded by
$C(1+\|\rho\|_\infty)\|\varphi\|_\infty^2|z|^{-n-2}$.
For the first term, Cauchy--Schwarz gives uniformly in $x$
\begin{align}
 \int_{|z|>1}|z|^{-n}|\rho(x-z)|\,dz
 &\le\left(\int_{|z|>1}|z|^{-2n}\,dz\right)^{1/2}\|\rho\|_2\notag\\
 &=\left(\frac{|\mathbb S^{n-1}|}{n}\right)^{1/2}\|\rho\|_2.
 \label{eq:density-L2-far-detail}
\end{align}
The orthogonal tail is bounded by
$C(1+\|\rho\|_\infty)\|\varphi\|_\infty|z|^{-n-1}$.
Both nonlinear bounds are integrable at infinity. Compact support is
not used in any of these tail estimates.

Near the diagonal, choose
$0<\alpha<\min\{1,s-n/2\}$. The graph function is $C^{1,\alpha}$ and the
density is $C^\alpha$. The odd frozen homogeneous terms cancel under
centered parameter truncations; the remainders are
$O(|z|^{-n+\alpha})$. Appendix~\ref{app:principal-values} identifies
these values with graph-distance truncations by comparison with the
frozen ellipsoidal cutoffs. This is a statement about the properly
normalized field: the flat tangential background is subtracted as in
\eqref{eq:T-tangent}. Thus both the pointwise meaning and the Sobolev
meaning of $\mathbf T$ are fixed before the lower estimate is proved.
\end{proof}

\begin{proof}[Proof of Theorem~\ref{thm:high-stability}]
Let $C_{\rm lin}$ be the constant in
Proposition~\ref{prop:linear-coercivity}, and let $C_{\rm rem}$ be the
constant in \eqref{eq:remainder-main}. These constants depend only on
$n,m,s$. Fix
\begin{equation}\label{eq:epsilon-choice}
 \varepsilon_s\le\min\left\{\frac14,\frac{\eta_*}{2},
                  \frac1{2C_{\rm lin}C_{\rm rem}}\right\}.
\end{equation}
The first restriction gives
$1+\rho\ge1-\|\rho\|_\infty\ge1/2$.
The second permits the analytic kernel expansion; the third will
absorb the nonlinear error. These choices do not involve the size of the higher Sobolev norms.

Set $Z=Z_s(\varphi,\rho)$ and
$\delta=\|D\varphi\|_\infty+\|\rho\|_\infty$.
By Proposition~\ref{prop:local-mapping}, all terms of
$\mathbf T=\mathcal LU+N(U)$ are already defined in $H^s$.
Consequently,
\begin{align}
 Z
 &\le C_{\rm lin}\|\mathcal L(\varphi,\rho)\|_{H^s}\notag\\
 &=C_{\rm lin}\|\mathbf T(\varphi,\rho)-N(\varphi,\rho)\|_{H^s}\notag\\
 &\le C_{\rm lin}\|\mathbf T(\varphi,\rho)\|_{H^s}
       +C_{\rm lin}C_{\rm rem}\delta Z.
 \label{eq:absorb-line}
\end{align}
For $\delta\le\varepsilon_s$, the last coefficient is at most $1/2$.
Subtract this term and divide by the remaining positive coefficient:
\[
 (1-C_{\rm lin}C_{\rm rem}\delta)Z
 \le C_{\rm lin}\|\mathbf T\|_{H^s},\qquad
 Z\le2C_{\rm lin}\|\mathbf T\|_{H^s}.
\]
This proves \eqref{eq:high-stability}. The proof also covers $Z=0$;
no division by $Z$ was made.

The elementary Fourier weight comparison from
Proposition~\ref{prop:linear-coercivity} gives
\[
 \|\varphi\|_{H^{s+1}}+\|\rho\|_{H^s}
 \le C_s(\|\varphi\|_2+Z)
 \le C_s(\|\varphi\|_2+\|\mathbf T\|_{H^s}),
\]
which is \eqref{eq:high-with-low-frequency}. If the graph is supported
in $B_R$, apply $\|\varphi\|_2\le C_nR\|D\varphi\|_2$ and the
already proved derivative estimate. The density remains unrestricted
in support. The use of the full inhomogeneous Sobolev space specifies
$\varphi$ and its behavior at infinity, whereas the estimate whose constant is independent of the supports
itself measures only derivatives of the graph.
\end{proof}

\begin{corollary}
\label{cor:full-Sobolev}
Under the smallness assumptions of Theorem~\ref{thm:high-stability}, if
\begin{equation}\label{eq:low-frequency-condition}
 \|\varphi\|_2\le L_0\|D\varphi\|_{H^s},
\end{equation}
then
\begin{equation}\label{eq:full-Sobolev}
 \|\varphi\|_{H^{s+1}}+\|\rho\|_{H^s}
 \le C_s(1+L_0)\|\mathbf T(\varphi,\rho)\|_{H^s}.
\end{equation}
In particular this holds with $L_0=C_nR$ for
$\supp\varphi\subset B_R$, regardless of the support of $\rho$.
\end{corollary}
\begin{proof}
Use $\|\varphi\|_{H^{s+1}}\le C_s(\|\varphi\|_2+
\|D\varphi\|_{H^s})$ and \eqref{eq:high-stability}.
\end{proof}

\begin{corollary}
\label{cor:near-reflectionless}
Under Theorem~\ref{thm:high-stability},
$\|\mathbf T(\varphi,\rho)\|_{H^s}\le\epsilon$ implies
$Z_s(\varphi,\rho)\le C_s\epsilon$.
If $\mathbf T(\varphi,\rho)=0$, then $\varphi=0$ and $\rho=0$.
\end{corollary}
\begin{proof}
The first claim is \eqref{eq:high-stability}.  In the zero case,
$D\varphi=0$ and $\rho=0$.  A constant $\varphi\in L^2(\R^n)$ is zero.
\end{proof}

\subsection{Local injectivity}
\begin{proposition}
\label{prop:analytic-derivative}
Let $s>n/2$ be an integer, and set
$\mathcal X_s=H^{s+1}(\R^n;\R^m)\times H^s(\R^n)$.
The map $\mathbf T:\mathcal X_s\to H^s$ is real analytic in a
neighborhood of zero. There are constants $r_s,C_s>0$ such that
\begin{equation}\label{eq:derivative-strong}
 \|D\mathbf T(U)[V]-\mathcal L V\|_{H^s}
 \le C_s Z_s(U)Z_s(V)
 \quad\text{if }U,V\in\mathcal X_s,\quad Z_s(U)<r_s.
\end{equation}
\end{proposition}
\begin{proof}
Write $U=(\varphi,\rho)$ and let $V=(\psi,\sigma)$.
Since $s>n/2$,
\begin{equation}\label{eq:strong-low-control-detail}
 \|D\psi\|_\infty+\|\sigma\|_\infty\le C_sZ_s(V).
\end{equation}
Group the nonlinear series by total degree in the pair $U$:
\[
 N(U)=\sum_{k\ge2}P_k(U).
\]
For fixed $k$, each $P_k$ is a finite sum of continuous multilinear
forms evaluated on copies of $U$. In the mixed terms at most one entry
is a density, because the forward field is affine in $\rho$.
By the coefficient count, the component count, and
Lemma~\ref{lem:tame-commutator}, there are constants $C,J$, independent
of $k$, such that these forms satisfy
\begin{equation}\label{eq:homogeneous-strong-bound-detail}
 \|P_k(U)\|_{H^s}\le C^k(1+k)^J Z_s(U)^k.
\end{equation}
For example, there are at most $k$ choices for the coefficient
measured in $H^s$, while all remaining coefficient and density factors
are bounded in $L^\infty$ by \eqref{eq:strong-low-control-detail}. All such
polynomial factors are included in $(1+k)^J$.

Differentiation of a degree-$k$ multilinear expression replaces one of
its $k$ copies of $U$ by $V$. Applying the same tame estimate with either
the $V$ entry or a $U$ entry carrying the high norm, and using
\eqref{eq:strong-low-control-detail} on the remaining entries, yields
\begin{equation}\label{eq:homogeneous-derivative-bound-detail}
 \|DP_k(U)[V]\|_{H^s}
 \le C^k(1+k)^{J+1}Z_s(U)^{k-1}Z_s(V).
\end{equation}
In particular, for $CZ_s(U)\le1/4$,
\begin{align*}
 \sum_{k\ge2}\|DP_k(U)[V]\|_{H^s}
 &\le C^2Z_s(U)Z_s(V)
     \sum_{k\ge2}(1+k)^{J+1}(CZ_s(U))^{k-2}\\
 &\le C_sZ_s(U)Z_s(V).
\end{align*}
The convergence is uniform when $Z_s(U)$ stays in a smaller fixed ball
and $Z_s(V)\le1$. This proves \eqref{eq:derivative-strong} and justifies
differentiation of the series term by term. Higher derivatives acquire
at most $k(k-1)\cdots(k-j+1)$ at order $j$ and obey the corresponding
uniformly summable bounds on smaller balls.

To specify the topology, the homogeneous polynomials above are
continuous on the Banach space $\mathcal X_s$, since
$Z_s(U)\le C_s\|U\|_{\mathcal X_s}$.
Their normally convergent polynomial series therefore defines a real
analytic map on a neighborhood of zero in $\mathcal X_s$.
The same derivative calculation applies at any $U\in\mathcal X_s$
with sufficiently small $Z_s(U)$, even if $\|\varphi\|_2$ is large:
$Z_s$ is continuous in the Banach topology, so a sufficiently small
Banach neighborhood of that point still lies within the convergence
region. We do not require or assert that $\mathcal X_s$ is complete
under $Z_s$. The polynomial bounds, not an inverse mapping theorem,
provide the derivative estimate.
\end{proof}

\begin{corollary}
\label{cor:local-injectivity}
There is $r_s'>0$ such that, for $U_1,U_2\in\mathcal X_s$ with
$Z_s(U_j)\le r_s'$, one has
\begin{equation}\label{eq:difference-stability}
 Z_s(U_1-U_2)\le C_s
       \|\mathbf T(U_1)-\mathbf T(U_2)\|_{H^s}.
\end{equation}
In particular $\mathbf T$ is injective on this class.  Under a fixed
support or a condition of the form \eqref{eq:low-frequency-condition}
for $\varphi_1-\varphi_2$, the left side can be replaced by the full
product Sobolev norm, with the corresponding constant.
\end{corollary}
\begin{proof}
Choose $r_s'$ smaller than the radius in
Proposition~\ref{prop:analytic-derivative} and so that
$C_{\rm lin}C_sr_s'\le1/2$. Let $V=U_1-U_2$ and
$U_t=(1-t)U_2+tU_1$. The triangle inequality gives
\[
 Z_s(U_t)\le(1-t)Z_s(U_2)+tZ_s(U_1)\le r_s'
 \qquad(0\le t\le1).
\]
The differentiated series converges uniformly on this segment. The
fundamental theorem of calculus for the Banach-valued map therefore
gives
\[
 \mathbf T(U_1)-\mathbf T(U_2)=\mathcal LV+E,
 \qquad E=\int_0^1\bigl(D\mathbf T(U_t)[V]-\mathcal LV\bigr)\,dt.
\]
The derivative bound implies
\[
 \|E\|_{H^s}\le\int_0^1 C_sZ_s(U_t)Z_s(V)\,dt
 \le C_sr_s'Z_s(V).
\]
Apply linear coercivity and absorb:
\begin{align*}
 Z_s(V)
 &\le C_{\rm lin}\|\mathcal LV\|_{H^s}\\
 &\le C_{\rm lin}\|\mathbf T(U_1)-\mathbf T(U_2)\|_{H^s}
      +C_{\rm lin}C_sr_s'Z_s(V),
\end{align*}
so $Z_s(V)\le2C_{\rm lin}\|\mathbf T(U_1)-\mathbf T(U_2)\|_{H^s}$.
If the two fields agree, then $D(\varphi_1-\varphi_2)=0$ in
distributions and $\rho_1=\rho_2$. The graph difference belongs to
$L^2(\R^n)$, and the only constant in that space is zero. This proves
injectivity.

The full norm follows if
$\|\varphi_1-\varphi_2\|_2
\le L_0\|D(\varphi_1-\varphi_2)\|_{H^s}$.
Apply the same Fourier weight comparison used in
Corollary~\ref{cor:full-Sobolev}. In particular a common fixed graph
support suffices. This last step requires a condition on the graph
\emph{difference}, not two separate unrelated choices of its constant
representative.
\end{proof}

The corollary requires smallness of $Z_s(U_j)$, whereas
Theorem~\ref{thm:high-stability} assumes only smallness of
$\|D\varphi\|_\infty+\|\rho\|_\infty$. It therefore gives local
injectivity on a smaller class. Neither estimate implies surjectivity
onto an open set in the space of vector fields.  For instance the tangential component of the flat differential
has Fourier vector parallel to $\xi$, which is a compatibility
restriction when $n\ge2$.

\appendix
\section{Frozen kernels and comparison of truncations}\label{app:principal-values}

The argument is local and does not require compact support.
Let $\varphi\in C^{1,\alpha}$ near a fixed compact set, and let
$a\in C^\alpha$ there, with $0<\alpha<1$.  A smooth signed function can
replace $a$ in this appendix.  Put $z=y-x$ and
\[
 \mathcal K_\varphi(x,x+z)
 =-a(x+z)\frac{\varphi(x+z)-\varphi(x)}
 {|X_\varphi(x+z)-X_\varphi(x)|^{n+1}}.
\]

\begin{lemma}\label{lem:frozen-subtraction}
Let $P=D\varphi(x)$ and $L_xz=(z,Pz)$.  The frozen kernel
\begin{equation}\label{eq:frozen-kernel}
 \mathcal K_x^{\rm lin}(z)=-a(x)\frac{Pz}{|L_xz|^{n+1}}
\end{equation}
is odd in $z$, and, uniformly for $x$ in the compact set and small $z$,
\begin{equation}\label{eq:frozen-error}
 |\mathcal K_\varphi(x,x+z)-\mathcal K_x^{\rm lin}(z)|
 \le C|z|^{-n+\alpha}.
\end{equation}
\end{lemma}
\begin{proof}
Fix a compact set $K$ of base points. Choose $r_0>0$ so that the
$2r_0$-neighborhood of $K$ is contained in a region where the stated
norms of $\varphi$ and $a$ are bounded. All constants in this proof are
uniform for $x\in K$ and $|z|<r_0$, and may depend on those local norms.
Writing $P=D\varphi(x)$, the integral form of Taylor's formula gives
\begin{align}
 \varphi(x+z)-\varphi(x)
 &=Pz+r_x(z),\qquad
 r_x(z)=\int_0^1(D\varphi(x+tz)-D\varphi(x))z\,dt,\notag\\
 |r_x(z)|
 &\le [D\varphi]_{C^\alpha}|z|^{1+\alpha}
                      \int_0^1t^\alpha\,dt
 =\frac{[D\varphi]_{C^\alpha}}{1+\alpha}|z|^{1+\alpha}.
 \label{eq:phi-Taylor}
\end{align}
Similarly,
\begin{equation}\label{eq:a-Taylor}
 a(x+z)=a(x)+b_x(z),\qquad
 |b_x(z)|\le[a]_{C^\alpha}|z|^\alpha.
\end{equation}
The actual and frozen distances are
\[
 D=\bigl|(z,Pz+r_x(z))\bigr|,
 \qquad d=|L_xz|=\bigl|(z,Pz)\bigr|.
\]
The reverse triangle inequality and the unchanged first $n$ coordinates
give
\begin{equation}\label{eq:X-Taylor}
 |D-d|\le|r_x(z)|\le C|z|^{1+\alpha},\qquad
 D\ge|z|,\quad d\ge|z|.
\end{equation}
Thus replacing the increment of the graph by $Pz$ and its density by
$a(x)$ produces the frozen kernel in \eqref{eq:frozen-kernel}.
Here the tangent plane at $X_\varphi(x)$ is used; it need not be
horizontal.

Subtracting the two kernels by adding and subtracting the appropriate
numerators gives the exact decomposition
\begin{align}
 \mathcal K_\varphi(x,x+z)-\mathcal K_x^{\rm lin}(z)
 ={}&-a(x+z)r_x(z)D^{-n-1}\notag\\
 &-b_x(z)PzD^{-n-1}\notag\\
 &-a(x)Pz\bigl(D^{-n-1}-d^{-n-1}\bigr).
 \label{eq:frozen-three-errors-detail}
\end{align}
The first term has size at most
$C|z|^{1+\alpha}|z|^{-n-1}=C|z|^{-n+\alpha}$.
The second has size at most
$C|z|^\alpha|z||z|^{-n-1}=C|z|^{-n+\alpha}$.
For the third, the mean value theorem applied to $t\mapsto t^{-n-1}$
on the interval between $D$ and $d$ gives
\[
 |D^{-n-1}-d^{-n-1}|
 \le(n+1)|D-d|\min(D,d)^{-n-2}
 \le C|z|^{-n-1+\alpha}.
\]
Multiplication by $|Pz|\le C|z|$ gives the same bound. Adding the
three estimates proves \eqref{eq:frozen-error}. Their radial integral
near the origin is finite because
\[
 \int_{|z|<r_0}|z|^{-n+\alpha}\,dz
 =|\mathbb S^{n-1}|\int_0^{r_0}r^{\alpha-1}\,dr
 =\frac{|\mathbb S^{n-1}|}{\alpha}r_0^\alpha.
\]
Finally, $P(-z)=-Pz$ and $L_x(-z)=-L_xz$. The numerator is odd and
the denominator is even, so
$\mathcal K_x^{\rm lin}(-z)=-\mathcal K_x^{\rm lin}(z)$.
This cancellation belongs to the frozen kernel; the actual kernel
need not be odd in $z$.
\end{proof}

\begin{lemma}\label{lem:cutoff-comparison}
Fix a sufficiently small $r_0>0$.  For $0<\epsilon<r_0$,
\begin{equation}\label{eq:cutoff-comparison}
 \left|\int_{\substack{|z|<r_0\\
       |X_\varphi(x+z)-X_\varphi(x)|>\epsilon}}
       \mathcal K_x^{\rm lin}(z)\,dz\right|
 \le C\epsilon^\alpha.
\end{equation}
\end{lemma}
\begin{proof}
Keep $x$ fixed and put
\[
 A_\epsilon=\{z:|z|<r_0,\ D(z)>\epsilon\},\qquad
 F_\epsilon=\{z:|z|<r_0,\ d(z)>\epsilon\}.
\]
Let $L_M=(1+M^2)^{1/2}$, where $M$ bounds the local derivative on the
neighborhood used in the preceding proof. The segment from $x$ to
$x+z$ remains in that neighborhood, so
\[
 |z|\le D(z),d(z)\le L_M|z|.
\]
Since $d(-z)=d(z)$, $F_\epsilon$ is centrally symmetric. Also
$d(z)>\epsilon$ implies $|z|>\epsilon/L_M$, so the frozen kernel is
absolutely integrable on $F_\epsilon$ at fixed $\epsilon$. The change
of variables $z\mapsto-z$ is therefore legitimate and gives
\[
 \int_{F_\epsilon}\mathcal K_x^{\rm lin}(z)\,dz
 =\int_{F_\epsilon}\mathcal K_x^{\rm lin}(-z)\,dz
 =-\int_{F_\epsilon}\mathcal K_x^{\rm lin}(z)\,dz=0.
\]
The actual domain $A_\epsilon$ is not generally centrally symmetric,
so it cannot be used in this equality.

Let $E_\epsilon=A_\epsilon\triangle F_\epsilon$. If $z\in E_\epsilon$,
then either $D>\epsilon\ge d$ or $d>\epsilon\ge D$. In either case,
$\epsilon$ lies between $D$ and $d$. Thus
\begin{equation}\label{eq:mismatch-shell}
 \frac{\epsilon}{L_M}\le|z|\le\epsilon,
 \qquad |d(z)-\epsilon|\le|D(z)-d(z)|
                     \le C\epsilon^{1+\alpha}.
\end{equation}
In polar coordinates $z=r\theta$, the frozen distance is exactly
\[
 d(r\theta)=r\lambda_x(\theta),\qquad
 \lambda_x(\theta)=\sqrt{1+|P\theta|^2}\in[1,L_M].
\]
For a fixed direction $\theta$, the radial section of $E_\epsilon$
is contained in
\[
 \left\{r:\left|r-\frac{\epsilon}{\lambda_x(\theta)}\right|
       \le\frac{C\epsilon^{1+\alpha}}{\lambda_x(\theta)}\right\}
 \cap[\epsilon/L_M,\epsilon].
\]
Its one-dimensional measure is therefore at most
$2C\epsilon^{1+\alpha}$. This is only an interval containing the
section: no monotonicity of the actual distance along the ray is
required. Polar integration gives
\begin{align}
 |E_\epsilon|
 &=\int_{\mathbb S^{n-1}}\int_0^{r_0}
       \mathbf1_{E_\epsilon}(r\theta)r^{n-1}\,dr\,d\theta\notag\\
 &\le\epsilon^{n-1}\int_{\mathbb S^{n-1}}
        2C\epsilon^{1+\alpha}\,d\theta
 \le C\epsilon^{n+\alpha}.
 \label{eq:shell-volume}
\end{align}
For $n=1$ the same argument integrates over the two rays and uses
$r^{n-1}=1$.

On $E_\epsilon$, the frozen kernel satisfies
\[
 |\mathcal K_x^{\rm lin}(z)|
 \le |a(x)|\|P\|\,|z|\,|z|^{-n-1}
 \le C L_M^n\epsilon^{-n}.
\]
Subtracting the zero frozen-domain integral now gives
\begin{align*}
 \left|\int_{A_\epsilon}\mathcal K_x^{\rm lin}(z)\,dz\right|
 &=\left|\int
    (\mathbf1_{A_\epsilon}-\mathbf1_{F_\epsilon})
                       \mathcal K_x^{\rm lin}(z)\,dz\right|\\
 &\le\int_{E_\epsilon}|\mathcal K_x^{\rm lin}(z)|\,dz
 \le C\epsilon^{-n}\epsilon^{n+\alpha}=C\epsilon^\alpha.
\end{align*}
This proves \eqref{eq:cutoff-comparison}. The estimate uses both the measure of the symmetric difference and
the size of the kernel on that set.
\end{proof}

\subsection*{Convergence rate and the normalization of the full field}
Let $E_x(z)=\mathcal K_\varphi(x,x+z)-\mathcal K_x^{\rm lin}(z)$.
The previous lemma does not say that the actual frozen integral is
identically zero, but only that it tends to zero. Since $E_x$ is
integrable, the actual local truncated integral decomposes as
\begin{align*}
 \int_{A_\epsilon}\mathcal K_\varphi(x,x+z)\,dz
 -\int_{|z|<r_0}E_x(z)\,dz
 =\int_{A_\epsilon}\mathcal K_x^{\rm lin}(z)\,dz
 -\int_{\substack{|z|<r_0\\D(z)\le\epsilon}}E_x(z)\,dz.
\end{align*}
The second integration region lies inside $|z|\le\epsilon$, because
$D(z)\ge|z|$. Hence
\begin{align}
 \left|\int_{A_\epsilon}\mathcal K_\varphi(x,x+z)\,dz
       -\int_{|z|<r_0}E_x(z)\,dz\right|
 &\le C\epsilon^\alpha
        +C\int_{|z|\le\epsilon}|z|^{-n+\alpha}\,dz\notag\\
 &\le C_\alpha\epsilon^\alpha.
 \label{eq:local-pv-rate}
\end{align}
In particular, if $0<\delta<\epsilon<r_0$, subtraction of the two
estimates gives
\[
 \left|\int_{A_\delta}\mathcal K_\varphi(x,x+z)\,dz
       -\int_{A_\epsilon}\mathcal K_\varphi(x,x+z)\,dz\right|
 \le C_\alpha(\delta^\alpha+\epsilon^\alpha).
\]
The constants are uniform on the prescribed compact set of $x$.
The portion $|z|\ge r_0$ is unchanged by the inner truncation once
$\epsilon<r_0$, and is treated by absolute convergence in each
application. This proves both the local principal value and its local
uniform convergence by applying domination to the remainder after
subtracting the odd homogeneous term.

For the full vector kernel, the frozen term is
\[
 \widetilde K_x^{\rm lin}(z)
 =-a(x)\frac{(z,Pz)}{|L_xz|^{n+1}}.
\]
It is odd and of size $C|z|^{-n}$. The same three-error decomposition
and the same shell estimate apply componentwise. Let
$G_\epsilon=\{\epsilon<|z|<r_0\}$.
The frozen integrals over both $F_\epsilon$ and $G_\epsilon$ vanish.
These two centered domains need not have a small symmetric difference;
smallness is not needed because both integrals are exactly zero.
The integrable remainder has the same limit on either domain, whereas
the actual-versus-$F_\epsilon$ error tends to zero by the shell estimate.
This proves agreement of parameter-distance and graph-distance
principal values.

For the normalized tangential component there is one additional
background subtraction. In the coordinates $z=y-x$, its flat kernel is
$-z/|z|^{n+1}$, which is odd and has the same size bound. Its integral
vanishes under either centered frozen cutoff, and the actual cutoff
differs from the frozen one only on $E_\epsilon$. Its contribution to
that difference is again $O(\epsilon^\alpha)$. Thus subtracting the
flat tangential background is compatible with the two inner-cutoff
conventions for \eqref{eq:T-tangent}--\eqref{eq:T-normal}. The
subtraction, rather than assigning an absolutely convergent transform
to the infinite constant plane, fixes the far-field convention.

All H\"older expansions in this appendix require the stated regularity
of the density. They are applied to smooth signed functions during
approximation and to $1+\rho$ in the Sobolev theorem. For the merely
bounded density in Proposition~\ref{prop:normal-pv}, the condition is instead
localized and treated by strong $L^2$ graph-operator convergence; no
H\"older expansion of that density is asserted.

\section{Intrinsic density and parameter density}\label{app:density-conversion}
\subsection*{Bounded-slope conversion and square-integrable tails}
Writing $\widetilde g=g\circ X_\varphi$, the area formula gives
\[
 a(x)=\widetilde g(x)J_\varphi(x),\qquad
 J_\varphi(x)=\sqrt{\det(I+D\varphi(x)^TD\varphi(x))}.
\]
Consequently
\begin{equation}\label{eq:density-conversion}
 a-a_\infty
 =J_\varphi(\widetilde g-a_\infty)+a_\infty(J_\varphi-1).
\end{equation}
If $\|D\varphi\|\le M$, the eigenvalues of $D\varphi^TD\varphi$
are nonnegative and at most $M^2$, so
\[
 1\le J_\varphi\le(1+M^2)^{n/2}=:C_M.
\]
The smooth function $F(P)=\sqrt{\det(I+P^TP)}$ satisfies
$F(0)=1$ and $DF(0)=0$. Taylor's integral formula on the matrix ball
$|P|\le M$ therefore gives
\[
 |F(P)-1|
 =\left|\int_0^1(1-t)D^2F(tP)[P,P]\,dt\right|
 \le C_M|P|^2.
\]
Thus $J_\varphi-1\in L^2$ whenever $D\varphi\in L^4$.
Under this additional condition, \eqref{eq:density-conversion} and
its rearrangement
\[
 \widetilde g-a_\infty
 =J_\varphi^{-1}(a-a_\infty)-a_\infty J_\varphi^{-1}(J_\varphi-1)
\]
show that the two density-tail conditions are equivalent. Intrinsic
$L^2$ and parameter $L^2$ are comparable because
\[
 \|g-a_\infty\|_{L^2(\Gamma_\varphi,d\mathcal H^n)}^2
 =\int|\widetilde g(x)-a_\infty|^2J_\varphi(x)\,dx.
\]
The finite-half-energy graph theorem is stated in terms of $a$.
It does not infer $J_\varphi-1\in L^2$ merely from the finite
half-order seminorm. The intrinsic density formulation of that theorem
therefore requires the additional Jacobian condition stated above.

\subsection*{The Sobolev derivative estimate in intrinsic coordinates}
For the high-order theorem $D\varphi\in H^s$, $s>n/2$, and hence
$D\varphi\in L^2\cap L^\infty\subset L^4$. More precisely,
$\|D\varphi\|_4^4\le\|D\varphi\|_\infty^2\|D\varphi\|_2^2$.
Here it is possible to compare the high norms as well.

We first record the product estimate being used. For smooth functions
$u_1,\ldots,u_q\in H^s\cap L^\infty$, Leibniz's rule, the exponent
choice $p_i=2s/k_i$ for derivative orders with $\sum k_i=s$, and
Lemmas~\ref{lem:GN-derivative}--\ref{lem:weighted-AMGM} give
\begin{equation}\label{eq:ordinary-product-tame-detail}
 \left\|\prod_{i=1}^q u_i\right\|_{H^s}
 \le C_s(1+q)^s\sum_{\ell=1}^q
       \|u_\ell\|_{H^s}\prod_{i\ne\ell}\|u_i\|_\infty.
\end{equation}
For the zero-order term use one factor in $L^2$ and all others in
$L^\infty$. For a top derivative use H\"older with reciprocals summing
to $1/2$, interpolate each differentiated factor, and apply the weighted
geometric-mean inequality. The total Leibniz coefficient is $q^s$;
at most $s$ interpolation constants occur. These observations prove
\eqref{eq:ordinary-product-tame-detail} without appealing to a singular
integral estimate for an ordinary product. Approximation extends it to
the stated Sobolev class. In particular, for $q=2$,
\[
 \|uv\|_{H^s}\le C_s(\|u\|_\infty\|v\|_{H^s}
                        +\|u\|_{H^s}\|v\|_\infty).
\]

The matrix functions
$F_+(P)=\det(I+P^TP)^{1/2}-1$ and
$F_-(P)=\det(I+P^TP)^{-1/2}-1$
are analytic near zero, with zero constant and linear coefficients.
In a sufficiently small fixed matrix ball their power series contain
only degrees at least two. Applying
\eqref{eq:ordinary-product-tame-detail} to each monomial and summing the
convergent coefficient series gives
\begin{equation}\label{eq:Jacobian-tame}
 \|J_\varphi-1\|_{H^s}+\|J_\varphi^{-1}-1\|_{H^s}
 \le C_s\|D\varphi\|_\infty\|D\varphi\|_{H^s}.
\end{equation}
The same expansion, without derivatives, gives
$\|J_\varphi-1\|_\infty+\|J_\varphi^{-1}-1\|_\infty\le C\eta^2$
for $\eta=\|D\varphi\|_\infty$.

Put $j=J_\varphi-1$, $\gamma=\widetilde g-1$ and $\rho=a-1$.
Then
\[
 \rho=\gamma+j+\gamma j.
\]
The two-factor product estimate and \eqref{eq:Jacobian-tame} give
\begin{align}
 \|\rho\|_{H^s}
 &\le(1+C_s\eta^2)\|\gamma\|_{H^s}
       +C_s\eta(1+\|\gamma\|_\infty)\|D\varphi\|_{H^s}.
 \label{eq:intrinsic-forward-norm-detail}
\end{align}
Using
$\gamma=\rho+(J_\varphi^{-1}-1)+\rho(J_\varphi^{-1}-1)$
gives the reverse estimate with $\rho,\gamma$ interchanged.
Consequently, for the corresponding small slope and density norms,
\[
 \|D\varphi\|_{H^s}+\|\rho\|_{H^s}
 \asymp\|D\varphi\|_{H^s}+\|\gamma\|_{H^s}.
\]
The low norms also match quantitatively:
\[
 \|\rho\|_\infty
 \le(1+C\eta^2)\|\gamma\|_\infty+C\eta^2,
\]
and the converse follows from $J_\varphi^{-1}$. Thus a suitably small
intrinsic density perturbation and slope satisfy the smallness required
in Theorem~\ref{thm:high-stability}, with possibly adjusted fixed
constants. This is the precise sense in which that theorem has an
intrinsic-density formulation. The observed field itself is unchanged;
only the coordinates used for its density variable are converted.


\begin{thebibliography}{99}

\bibitem{ChristJourne}
M.~Christ and J.-L.~Journ\'e,
\emph{Polynomial growth estimates for multilinear singular integral operators},
Acta Math. \textbf{159} (1987), 51--80.

\bibitem{DallaRivaLuzziniMusolino}
M.~Dalla Riva, P.~Luzzini, and P.~Musolino,
\emph{Shape analyticity and singular perturbations for layer potential operators},
ESAIM, Math. model. numer. anal. \textbf{56} (2022), 1889-1910.

\bibitem{DavidSemmes}
G.~David and S.~Semmes,
\emph{Analysis of and on Uniformly Rectifiable Sets},
Mathematical Surveys and Monographs, Vol.~38,
American Mathematical Society, Providence, RI, 1993.

\bibitem{DiNezzaPalatucciValdinoci}
E.~Di Nezza, G.~Palatucci, and E.~Valdinoci,
\emph{Hitchhiker's guide to the fractional Sobolev spaces},
Bull. Sci. Math. \textbf{136} (2012), no.~5, 521--573.


\bibitem{GhoshSaloUhlmann}
T.~Ghosh, M.~Salo, and G.~Uhlmann,
\emph{The Calder\'on problem for the fractional Schr\"odinger equation},
Anal. PDE \textbf{13} (2020), no.~2, 455--475.
\href{https://doi.org/10.2140/apde.2020.13.455}{doi:10.2140/apde.2020.13.455}.

\bibitem{JayeNazarovI}
B.~Jaye and F.~Nazarov,
\emph{Reflectionless measures for Calder\'on--Zygmund operators I: General theory},
J. Anal. Math. \textbf{135} (2018), no.~2, 599--638.

\bibitem{JayeNazarovII}
B.~Jaye and F.~Nazarov,
\emph{Reflectionless measures for Calder\'on--Zygmund operators II: Wolff potentials and rectifiability},
J. Eur. Math. Soc. \textbf{21} (2019), no.~2, 549--583.


\bibitem{KatoPonce}
T.~Kato and G.~Ponce,
\emph{Commutator estimates and the Euler and Navier--Stokes equations},
Comm. Pure Appl. Math. \textbf{41} (1988), no.~7, 891--907.

\bibitem{MatiocMatioc}
A.-V.~Matioc and B.-V.~Matioc,
\emph{A new reformulation of the Muskat problem with surface tension},
J. Differential Equations \textbf{350} (2023), 308--335.


\bibitem{Muscalu}
C.~Muscalu,
\emph{Calder\'on commutators and the Cauchy integral on Lipschitz curves revisited I: First commutator and generalizations},
Rev. Mat. Iberoam. \textbf{30} (2014), no.~2, 727--750.

\bibitem{MuscaluII}
C.~Muscalu,
\emph{Calder\'on commutators and the Cauchy integral on Lipschitz curves
revisited II. The Cauchy integral and its generalizations},
Rev. Mat. Iberoam. \textbf{30} (2014), no.~3, 1089--1122.


\bibitem{NazarovTolsaVolberg}
F.~Nazarov, A.~Volberg and  X.~Tolsa,
\emph{On the uniform rectifiability of AD-regular measures with bounded
Riesz transform operator: the case of codimension 1},
Acta Math. \textbf{213} (2014), no.~2, 237--321.

\bibitem{SeegerSmartStreet}
A.~Seeger, C.~K. Smart, and B.~Street,
\emph{Multilinear singular integral forms of Christ--Journ\'e type},
Mem. Amer. Math. Soc. \textbf{257} (2019), no.~1231.


\bibitem{Stein}
E.~M. Stein,
\emph{Singular Integrals and Differentiability Properties of Functions},
Princeton Mathematical Series, Vol.~30,
Princeton University Press, Princeton, NJ, 1970.

\bibitem{Tolsa2008}
X.~Tolsa,
\emph{Principal values for Riesz transforms and rectifiability},
J. Funct. Anal. \textbf{254} (2008), no.~7, 1811--1863.

\bibitem{TolsaICM}
X.~Tolsa,
\emph{Interactions between quantitative rectifiability, singular integrals, and boundary value problems for harmonic functions},
Proceedings of the International Congress of Mathematicians 2026,
Vol.~2, 313--340, SIAM/IMU, 2026.

\bibitem{TolsaNewCriteria}
X.~Tolsa,
\emph{New criteria for the rectifiability of Radon measures in terms of
Riesz transforms}, preprint (2025), revised 2026.
\href{https://arxiv.org/abs/2512.14534v2}{arXiv:2512.14534v2}.

\end{thebibliography}
\end{document}